\documentclass[11pt]{amsart}

\usepackage{mathptmx}
\usepackage{helvet}
\usepackage{courier}
\usepackage{graphicx}
\usepackage{multicol}

\usepackage[T1]{fontenc}
\usepackage{euscript}
\usepackage{amsmath}
\usepackage{amsthm}
\usepackage{amssymb}
\usepackage{amscd}
\usepackage{epic}
\usepackage{enumerate}
\usepackage{array}
\usepackage{tabularx}
\usepackage[T1]{fontenc}
\DeclareMathAlphabet{\pazocal}{OMS}{zplm}{m}{n}

\usepackage{amsfonts,amsthm,amsmath}
\usepackage{amssymb}

\numberwithin{equation}{section}
\numberwithin{equation}{subsection}

\theoremstyle{plain}

\newtheorem{theorem}[equation]{Theorem}
\newtheorem{lemma}[equation]{Lemma}
\newtheorem{proposition}[equation]{Proposition}

\newtheorem{corollary}[equation]{Corollary}

\theoremstyle{definition}

\newtheorem{example}[equation]{Example}
\newtheorem{remark}[equation]{Remark}

\newtheorem{definition}[equation]{Definition}

\newtheorem{bekezdes}[equation]{}

\newcommand{\bC}{{\mathbb C}}

\newcommand{\bZ}{{\mathbb Z}}

\newcommand{\calC}{{\mathcal C}}

\newcommand{\cV}{{\mathcal V}}
\newcommand{\cA}{{\mathcal A}}

\newcommand{\cS}{{\mathcal S}}\newcommand{\calS}{{\mathcal S}}
\newcommand{\cP}{\pi}

\newcommand{\calO}{{\mathcal O}}

\newcommand{\frR}{\mathfrak{R}}

\newcommand{\bPE}{{\bf PE}}

\newcommand{\calQ}{{\mathcal Q}}

\newcommand{\tX}{\widetilde{X}}

\newcommand{\C}{{\calc}}

\newcommand{\rank}{{\rm rank}\, }

\newcommand{\bt}{{\bf t}}

\newcommand{\bH}{{\mathbb H}}

\newcommand{\calv}{\mathcal{V}}

\newcommand{\Z}{\mathbb{Z}}
\newcommand{\Q}{\mathbb{Q}}
\newcommand{\R}{\mathbb{R}}

 \newcommand{\frX}{\mathfrak{X}}

\newcommand{\calt}{{\mathcal T}}

\def\C{\mathbb C}
\def\Q{\mathbb Q}
\def\R{\mathbb R}
\def\bH{\mathbb H}

\def\Z{\mathbb Z}

\newcommand{\h}{\hbar}

\author{Andr\'as N\'emethi}
\thanks{The author is partially supported by  ``\'Elvonal (Frontier)'' Grant KKP 144148}
\address{Alfr\'ed R\'enyi Institute of Math.,
Re\'altanoda utca 13-15, H-1053, Budapest, Hungary \newline
 \hspace*{3mm} ELTE - Univ. of Budapest, Dept. of Geo.,
 P\'azm\'any P\'eter s\'et\'any 1/A, 1117, Budapest, Hungary \newline \hspace*{3mm}
  BBU - Babe\c{s}-Bolyai Univ., Str, M. Kog\u{a}lniceanu 1, 400084 Cluj-Napoca, Romania
 \newline \hspace*{3mm}
BCAM - Basque Center for Applied Math.,
Mazarredo, 14 E48009 Bilbao, Basque Country, Spain}
\email{nemethi.andras@renyi.hu }

\title{Filtered  lattice homology of surface singularities}

\begin{document}


\begin{abstract}
Let $(X,o)$ be a complex analytic normal surface  singularity with rational homology sphere link $M$. The  `topological' lattice cohomology  $\bH^*=\oplus_{q\geq 0}\bH^q$
associated with $M$ and with any  of its spin$^c$ structures was  introduced in \cite{Nlattice}.
Each $\bH^q$ is a graded $\Z[U]$--module. Here we consider its homological version $\bH_*=\oplus_{q\geq 0}\bH_q$.
The construction uses a Riemann-Roch type weight function.
A key intermediate product is  a tower of spaces
$\{S_n\}_{n\in \Z}$ such that $\bH_q=\oplus_n H_q(S_n,\Z)$.

In this article we fix the embedded topological type of a
 reduced  curve singularity $(C,o)$ embedded into $(X,o)$, that is, a 1-dimensional link
  $L_C\subset M$. Each component of $L_C$ will also carry a non-negative integral decoration.
  For any fixed $n$,  the embedded link   $L_C$ provides
  a natural filtration of the space $S_n$, which induces
a homological spectral sequence converging to the homogeneous summand $H_q(S_n,\Z)$ of the lattice homology.
All the entries of all the pages of the spectral sequences are new invariants of the decorated $(M,L_C)$. Each page provides a triple graded $\Z[U]$-module.

We provide several concrete computations of these pages and
structure theorems for the corresponding multivariable Poincar\'e series associated
with the entries of the spectral sequences.

Connections with Jacobi theta series are also discussed.
\end{abstract}

\maketitle

\linespread{1.2}


\pagestyle{myheadings} \markboth{{\normalsize   A. N\'emethi}} {{\normalsize
Filtered lattice homology of surface  singularities}}

\section{Introduction}

\subsection{What we construct?}
The main goal of the present note is to create  the theoretical machinery
 of the   {\it filtered  lattice (co)homology}.
 Though the construction is general, here we will run the case
 of the {\it topological lattice homology}
  associated with the topological type of a complex analytic
 normal surface singularity  $(X,o)$.
  Recall that the topological type of such a singular germ $(X,o)$ is characterized by its link $M$, which is an oriented compact  plumbed 3--manifold associated with a connected negative definite plumbing graph. In this note we will assume that $M$ is a rational homology sphere.
  For details regarding the surface singularities see e.g. \cite{Nkonyv} and the references therein.
 The topological lattice (co)homology associates with  each    spin$^c$ structure of $M$ a double
 graded $\Z[U]$ module (and also a graded root), see e.g.
  \cite{Nlattice,Nkonyv}. This is a bridge between singularity theory and low dimensional topology.
  For certain applications in singularity theory see \cite{BoNe,GorNem2015,NGr,JEMS,Inv},
  for some applications  in low dimensional topology see e.g. \cite{AJK,DM,K2,KarakurtLidman,K1,KS,Zemke}.

In this note the construction of the lattice homology (for any fixed  spin$^c$ structure)
will be enhanced by  several filtrations.
In the simplest case, the  case  of the  $\Z$--filtration,
 one of the outputs of the theory is a $\Z$--filtration of the lattice homology,
 whose associated graded module --- the `graded lattice homology' ---
is  triple graded and carries a torsion $\Z[U]$--module structure.

In fact, the construction automatically provides a (multigraded)  homological spectral sequence  whose
$E^\infty$ page is exactly the `graded lattice homology' mentioned above.
However, the point is that for every  $k\geq 1$ the page $E^k$ of the spectral sequence provides  a
triple graded $\Z$--module with a torsion $\Z[U]$--module structure. The index $k$ where the
spectral sequence degenerates (i.e. $E^k=E^\infty$) usually is large, so in this way we obtain a long
tower of triple graded $\Z[U]$--modules indexed by $k$.
In general,  the computation of this degeneration index $k$ and of the intermediate pages $E^k$
are  very hard.
This degeneration index    is the analogue
of the $\tau$--invariant of Ozsv\'ath and Szab\'o \cite{OSztau}, or of the $s$--invariant of Rasmussen \cite{ras_s} in the context of
Heegaard Floer Link and Khovanov theories.

The filtration is induced geometrically. In the analytic context we can rely on  an analytic pair $(C,o)\subset (X,o)$, where $(C,o)$ is an isolated curve singularity embedded into $(X,o)$.
In the topological  context
we start with the induced topological pair, their embedded links $L_C\subset M$, where
$L_C=\sqcup S^1$ (and the number of $S^1$'s is the number $r$ or irreducible components of $(C,o)$).
In a convenient  plumbing graph of $M$ the embedded $L_C$ can be represented by $r$ arrowheads
(or unmarked vertices).

Since the pair $L_C\subset M$ can be represented by several plumbing graphs ---
 connected to each other  by a sequence  of
elementary blow ups ---, it is natural to test
 the stability of the filtered theory with respect to  such blow ups.
 (Recall that the lattice homology is stable with respect to any blow up, it is a 3-manifold
  invariant of $M$.)
 It turns out that the filtered theory is not stable with respect to blow ups of `edges supporting the arrowheads' (we will call such centers `base points'), but it is stable with respect to any other
 blow up. Hence, for each component of $L_C$ we will fix the number of how many times we blow up
 the supporting edge of the arrowhead of the corresponding component
 (the number of `base point blow ups' of each component).
 This provides an integral  decoration $(b_1,\ldots , b_r)$
 of the components $\{L_{C_i}\}_i$ of $L_C$. In particular, we obtain a well--defined theory for each
 decorated pair $(M, L_C,b)$, well defined up to decoration preserving
  (non-base point) blow ups. That is, {\bf for any
 negative definite plumbed 3--manifold and embedded (plumbed) link components $L_{C_1},\ldots, L_{C_r}$
 decorated by $(b_1,\ldots, b_r)\in(\Z_{\geq 0})^r$,   for any
 $k\geq 1$ we obtain  a triple graded $\Z[U]$--module $(E^k_{-d,q})_n$;} see \ref{ss:ss}.

 Its Poincar\'e series will be denoted by $PE_k(T,Q,\h)$. 

 Technically, it will be convenient to replace the information regarding the  embedded decorated
 $(L_C,b)$ by an equivalent object, namely
    by a semigroup element $s$ in the integral Lipman cone $\calS$
    of the lattice associated with the plumbing graph of $M$
    (up to an equivalence relation, see \ref{bek:semi}).

 \subsection{Motivations} \
 The construction has several  motivations. We mention here three of them.

 \vspace{2mm}

\noindent  {\bf (I)}\ {\it Improvements of lattice (co)homology}

Recently several (co)homology theories were introduced in low dimensional topology, gauge theory, etc.
In all these cases construction of additional filtrations and gradings successfully provided stronger invariants. This is the main principle what we follow in our case as well.

In present in the literature there are several versions of the lattice cohomology. The first
original one, introduced in \cite{Nlattice},
is the topological lattice cohomology associated with surface singularity links.
(For a presentation in a different language see also \cite{OSZSt2,OSZSt3,OSZSt}.)
It is the categorification of the Seiberg--Witten invariant of the link $M$. Later, the analytic version was also constructed
(as categorification of the geometric genus) \cite{AgostonNemethiI,AgostonNemethiIII},
even the higher dimensional analytic version were established \cite{AgostonNemethiIV},
all of them via resolutions.
By a different technique (namely, by normalization),  the construction was extended  for isolated curve singularities as well,
as categorification of the delta invariant \cite{AgostonNemethi}.

The present note provides a precise recipe how one can create additional filtrations on these theories.
In particular,  any filtration or grading of a homology theory provides automatically a grading (sum decomposition)
of the corresponding Euler characteristic. In fact, due to the presence of a spectral sequence,
a whole tower of  multivariable Poincar\'e series is provided.

 E.g., in the present case, the Euler characteristic $eu((\bH_*)_*)$ of the topological lattice
 homology is the (normalized) Seiberg--Witten invariant of $M$.
   By fixing a decorated pair $(M,L_C,b)$ we can consider
    the $E^\infty$ page of the spectral sequence, which  provides a triple graded object
    ${\rm Gr}_*((\bH_*)_*)$ with Poincar\'e series
 $PE_\infty(T,Q,\h)$. It turns out that $Pol_{SW}(Q):= PE_\infty(1,Q,-1)-1/(1-Q)$ is a polynomial in
 $Q$ which evaluated at $Q=1$ is the (normalized) Seiberg--Witten invariant of $M$. Hence,
 any embedded $(L_C,b)$ into $M$ defines a polynomial
  $Pol_{SW}(Q)$, a graded improvement of the  Seiberg--Witten invariant of $M$.

 \vspace{2mm}

  \noindent {\bf (II)} \ {\it The relative geometry of pairs}

In our case, the filtration (or induced grading) is induced by the relative geometry of the pair
$(X,C,o)$, or by $(M,L_C)$ (where $L_C$ is decorated). The study of such pairs has a long history and it is the subject of an intense activity. It starts with the very classical
theory of embedded knots/links in $S^3$. This even in the algebraic case, in the
classification of plane curve singularities $(C,o)\subset (\bC^2,0)$, is highly nontrivial
and it has many  connections with several areas of mathematics.
The natural generalizations to  pairs
$(C,o)\subset (X,o)$ target three research projects: (a) the analytic classification of such pairs,
(b) the topological classification of the embedded topological  types of such pairs, and (c)
 connections between the embedded topological types of such a  pair and the abstract analytic type of the curve $(C,o)$. E.g., regarding part (c),  \cite{CLMN1,CLMN2} ask whether the delta invariant of the abstract curve $(C,o)$ can be recovered from the embedded topological type of  $(M, L_C)$.
(E.g., if $M$ is an L-space , that is, if $(X,o)$ is rational, then the answer is yes.).
The present  note creates the filtered lattice homology for the pair $(M,L_C,b)$ (which naturally
extends to  the analytic case of $(X,C,b)$ as well). Note that
  the recent manuscript \cite{NemCurves} contains the construction of the filtered lattice homology of abstract curve singularities.
  In forthcoming manuscript we plan to compare these three filtered theories.

One can compare the results of the present note with the Heegaard Floer Link theory, the invariant associated with a link in a 3-manifold (sometimes in an L-space). Having in mind for singularity links the
coincidence of the lattice homology with the Heegaard Floer homology \cite{Zemke},
it  is very natural to compare  the present work with
the Heegaard Floer Link theory \cite{M} applied for algebraic links $(M,L_C)$,
or with another construction of a filtration in the lattice homology (under the restriction that  $L_C$ is a knot) \cite{OSZSt3}, or even with the recent manuscript \cite{BLZ}.
These comparisons   will also
 be considered in   forthcoming notes.

  \vspace{2mm}

 \noindent {\bf (III)} \ {\it  Jacobi Theta Series}

 The third motivation, interestingly enough, is arithmetical, and hopefully will create  a new bridge with
 number theory and  with the theory of theta series.  A classical theta series is a sum
  $\sum_{l\in \Z}q^{Q(l)}$, where $Q(l)$ is a quadratic function;
  its multivariable version is a sum of type
  $\sum _{l\in \Z^s} q^{\langle Al, l\rangle} $
  (where $\langle\,,\,\rangle$ is the standard scalar product and $A$ is a definite  matrix).
  Their generalizations, the Jacobi theta series
  are sums of type $\sum_{l\in\Z}q^{Q(l)}t^{cl}$, or
   $\sum _{l\in \Z^s} q^{\langle Al, l\rangle }t^{\langle Al, c\rangle }$,  (where
    $c$ is a constant or a scalar vector).

 It turns out that the Poincar\'e series associated with the pages $E^1$ and $E^\infty$ (the pages
 where we provide complete general descriptions)
 can be organized as finite sums of type
 $\sum _{l\in (\Z_{\geq 0})^s} Q^{\langle Al, l+c'\rangle }T^{\langle Al, c\rangle }$.
Here the quadratic part is essentially provided by the intersection matrix of the plumbing and the
constant vector $c'$  by the fixed spin$^c$ structure.

 In the body of the paper we provide many examples and prove structure theorems supporting this statement.
 In fact, in this note  once we finish the main construction, basically we focus on such structure theorems of the Poincar\'e series (and we postpone the other connections with Heegaard Floer theory, singularity theory, deformation theory to the next notes). In this way we wish to open a completely new bridge with new unexploited possibilities in the direction of such series.

 This part was greatly inspired by the theory of $\widehat{Z}$ series, introduced in
  \cite{GPPV} and intensively studied by several authors and  schools.  In that case the series  introduced can indeed be related with Jacobi theta series and  (mock) modular forms, however its categorification is missing
 (though there is a huge effort to find it). In the present note we wish to present a parallel
 situation, provided by singularity theory, in which case we have the series, it comes from
 a categorification (namely from the  lattice homology theory),
 however, the Poincar\'e series   are not
 `honest' Jacobi theta series:
  they are `one sided' (first quadrant)  sums of type \  $\sum _{l\in (\Z_{\geq 0})^s}$ of summands of
   `correct shape'.
   We really hope that such sums will also have their arithmetic rigidity consequences, as
  the modular forms have their applicability in different parts of geometry.

  \subsection{The main results, concrete statements}\label{ss:MR}

\bekezdes
We fix a resolution/plumbing  graph $\Gamma$ of $M$ with vertices $\cV$.
Let $L=\Z^{|\cV|}$ be  the free lattice
generated by the vertices of $\Gamma$ (with generators denoted by $\{E_i\}_i$) and let
 $L'$ be its dual lattice (generated by dual vectors $\{E_i^*\}_i$).
   We denote  the intersection form by $(\,.\,)$. Set $H:=L'/L$, it equals $H_1(M,\Z)$.
We also fix a class $h\in H$ (that is, a spin$^c$ structure on the link, cf. \cite{NOSz,Nkonyv} or
\ref{bek:invs} here). Let $Z_K\in L'$ be the anticanonical cycle, and $s_h\in L'$ the smallest representative of $h$ in the Lipman cone, cf.  \ref{bek:invs}. Finally,  let
the distinguished characteristic element $-Z_K+2s_h$ associated with  $h$ be denoted by $k_h$.
We define the
weight function of the lattice cohomology by the Riemann-Roch expression
 $\chi_{h}:L\to \Z$ by $\chi_{h}(l)=-\frac{1}{2}(l, l+k_h)$.
We consider the cubical decomposition of $\R^{|\cV|}$ with lattice point vertices,
and we define the weight of a $q$--cube $\square_q$ by
$w_h(\square_q):=\max\{\chi_{h}(v), \ \mbox{where $v$ is a vertex of $\square_q$}\}$, cf. \ref{9complex}. 

Finally,  for each $n\in \Z$ we
define $S_n=S_n(w_h)\subset \R^{|\cV|}$ as the union of all
the cubes $\square_q$ (of any dimension) with $w_h(\square_q)\leq
n$. Clearly, $S_n=\emptyset$, whenever $n<m_w:=\min_{l\in L}\{w_h(l)\}$. For any  $q\geq 0$, set
$$\bH_q(\R^{|\cV|},w_h):=\oplus_{n\geq m_w}\, H_q(S_n,\Z).$$
Then $\bH_q$ is $\Z$ (in fact, $2\Z$)--graded: the
$(-2n)$--homogeneous elements $(\bH_q)_{-2n}$ consist of  $H_q(S_n,\Z)$.
Also, $\bH_q$ is a $\Z[U]$--module; the $U$--action is the homological morphism
$H_q(S_{n},\Z)\to H_q(S_{n+1},\Z)$ induced by the inclusion $S_n\hookrightarrow S_{n+1}$.
This is the lattice homology $\bH_*$ of $M$, it is independent of the choice of the negative definite graph $\Gamma$.

Moreover, for
$q=0$, a fixed base-point in $S_{m_w}$ provides an augmentation
(splitting)
 $H_0(S_n,\Z)=
\Z\oplus \widetilde{H}_0(S_n,\Z)$, hence a splitting of the graded
$\Z[U]$-module
$$\bH_0=(\oplus_{n\geq m_w}\Z)\oplus (
\oplus_{n\geq m_w}\widetilde{H}_0(S_n,\Z))=(\oplus_{n\geq m_w}\Z)\oplus   \bH_{0,red}.$$
We also write $\bH_{q,red}=\bH_q$ for $q\geq 1$. Then the $\Z$--rank of $\bH_{*,red}$ is finite.
We define the Euler  characteristic of $\bH_*$ as
$$eu(\bH_*):=-\min_l\{w_h(l)\} +
\sum_q(-1)^q\ \rank_\Z(\bH_{q,red}).$$

\bekezdes Let us fix now a decorated pair $(M, L_C,b)$ coded in an `embedded' resolution/plumbing graph,
and an element $h\in L'/L$ as above. Hence, using a plumbing representation of $(M,s_h)$ we can produce the weight functions $w_h$ and the  spaces $S_n$ as above.
The decorated link $(L_C,b)$ can also be codified by  a semigroup element $s$ of the Lipman cone,
$s=\sum_{i\in {\rm Supp(s)}}n_iE_i^*$. Here $n_i$ is the number of arrows supported by $E_i$, hence $r=\sum_in_i$ is the number of components of $L_C$.
 The semigroup element $s$ identifies an increasing
filtration $\{\frX_{-d}\}_{d\geq 0}:= \{\cup(l+\R^{|\cV|}) \,:\, l\in L,\  (s,l)\leq -d\}$ of $\frX=(\R_{\geq 0})^{|\calv|}$, hence an
increasing filtration $\{S_n\cap \frX_{-d}\}_{d\geq 0}$ of any $S_n$. Hence we get a graded $\Z[U]$--module
$\bH_*(\frX_{-d},w_h)=\oplus_n H_*( S_n\cap \frX_{-d},\Z)$ for every $d$ and a sequence
 of graded $\Z[U]$--module morphisms
$\bH_*(\frX,w_h)\leftarrow \bH_*(\frX_{-1},w_h)\leftarrow \bH_*(\frX_{-2},w_h)\leftarrow \cdots.$
For each $q$, the map $\bH_q(\frX,w_h)\leftarrow \bH_q(\frX_{-d}, w_h)$ induced at lattice homology level is
homogeneous of degree zero.
These morphisms provide the following filtration of  $\Z[U]$--modules in $\bH_*(\frX,w_h)$
$${\rm F}_{-d}\bH_*(\frX,w_h):={\rm im}\big( \bH_*(\frX,w_h)\leftarrow \bH_*(\frX_{-d},w_h)\, \big).$$
This is the filtration of $\bH_*=\bH(\frX,w_h)$ induced by $s$ (or, by $(L_C,b)$).

Note that even if we know the homotopy type of the space $S_n$ (e.g. when we know that it is
contractible), this knowledge usually does not help in the determination of the filtered subspaces
$\{S_n\cap \frX_{-d}\}_d$. For this (see examples in the body of the paper)
we really need to
determine the spaces $S_n$ as embedded cubical subspace of $\frX=\R^{|\cV|}$. It might
 have  a rather complex shape with many  `lagoons' and `tentacles', and this shape makes
 the $E^k$ pages full with additional information.

\bekezdes In fact, we can do more.
Let us  fix $n$.
The morphism  $(\bH_*(\frX_{-d}, w_h))_{-2n}\to (\bH_*(\frX,w_h))_{-2n}$ is identical with
the morphisms $  H_*(S_n\cap \frX_{-d}, \Z)\to  H_*(S_n,\Z)$ induced by the inclusion
$S_n\cap \frX_{-d}\hookrightarrow S_n$.
It turns out that the filtration $\{S_n\cap \frX_{-d}\}_{d\geq 0}$ is finite.
 In particular, one can analyse the spectral sequence
associated with the filtration $\{S_n\cap \frX_{-d}\}_{d\geq 0}$ of subspaces of $S_n$.

This homological  spectral sequence will be denoted by
 $(E^k_{-d,q})_n\Rightarrow (E^\infty_{-d,q})_n$. Its terms $E^1$ and $E^\infty$ are the following:
  \begin{equation*}\begin{split}
  (E^1_{-d,q})_n=& H_{-d+q}(S_n\cap \frX_{-d}, S_n\cap \frX_{-d-1},\Z),\\
   (E^\infty_{-d,q})_n=& \frac{(F_{-d}\, \bH_{-d+q}(\frX))_{-2n}}
   { (F_{-d-1}\, \bH_{-d+q}(\frX))_{-2n}}=({\rm Gr}^F_{-d}\, \bH_{-d+q}(\frX)\,)_{-2n}.
   \end{split}\end{equation*}

Theorem  \ref{th:eq} and Theorem \ref{th:redth2} imply the following stability statements:

\begin{proposition}\label{prop:stab_int}
Each $\Z$--module $(E^k_{-d,q})_{n}$ is well-defined module associated with
$(\Gamma, s)$ and $h\in H$,  and  it is independent of the choice of $(\Gamma, s)$  up to non-base point blow ups.
%
\end{proposition}

 Thus,  for every $1\leq k\leq \infty$, we also have the Poincar\'e series associated with $(\Gamma, s)$ and $h\in H$ 
 $$PE_{h,k}(T,Q,\h)=PE_{k}(T,Q,\h):=\sum_{d,q,n} \ \rank (E_{-d,q}^k)_{n}\cdot T^dQ^n\h^{-d+q}\in\Z[[P,Q]][Q^{-1},\h].$$

For any fixed $n$, the natural inclusion $S_n\hookrightarrow S_{n+1}$
is compatible with the filtration, hence the inclusion
$\{S_n\cap \frX_{-d}\}_d\hookrightarrow \{S_{n+1}\cap \frX_{-d}\}_d$
at the level of filtered spaces  induces a morphism $U$ of spectral sequences
$(E^k_{-d.q})_n\to (E^k_{-d,q})_{n+1}$ compatibly  with the differentials
$(d^k_{-d.q})_n$ and $(d^k_{-d,q})_{n+1}$.


\begin{lemma}\label{lem:torsion_int}
For any fixed $n$ there exists $\delta(n)\in \Z_{>0}$ such that for any $(d,q)$ the morphisms
$$U^{\delta(n)}\,:\, (E^1_{-d.q})_n\to (E^1_{-d,q})_{n+\delta(n)} \ \ \mbox{is trivial}. $$
In particular, $U^{\delta(n)}$ at the level of the $E^k$ ($k\geq 1$) pages is trivial too.
In this way, for any $k\geq 1$, the triple graded $\Z$--module $(E^k_{*.*})_*$  has a torsion
$\Z[U]$--module structure.
\end{lemma}

\bekezdes In the body of the paper we  provide many examples and several  concrete computations.
In the presentation  we focus on the following facts:

$\bullet$ \ non-stability with respect to a base point blow up,
stability with respect to any other blow up;

$\bullet$ \ the proof and applications  of the `Filtered Reduction Theorem',
when the rank of the lattice can be
decreased according to the number of `bad vertices';

$\bullet$ \ computation of the degeneration index $k$;

$\bullet$ \  exemplification of the $U$--action at the level of $E^1$ or $E^\infty$;

$\bullet$ \ the structure of the page $E^1$ and $E^\infty$ and of their
 Poincar\'e series;

$\bullet$ \ connection  with the `motivic' Poincar\'e series of $M$.

 \bekezdes Regarding motivation ({\bf I}) we extract here the following sample statement:

  \begin{proposition}\label{prop:infty_int} \ 

\noindent (a) $PE_\infty(1, Q,\h)=\sum_{n\geq m_w}\, \big(\, \sum_b \, {\rm rank}\, H_b(S_n,\Z)\,\h^b\, \big)\cdot Q^n$. Hence,
 $PE_\infty(1, Q,-1)=\sum_{n\geq m_w}\, \chi_{top}(S_n)\cdot Q^n$ \ (where $\chi_{top}$ denoted the topological Euler characteristic).

\noindent (b) Let $R$ be any  rectangle of type $R(0,c')=\{0\leq x\leq c'\}$
with $c'\geq \lfloor Z_K \rfloor$. Then
\begin{equation}\label{eq:1}
PE_\infty(1, Q,-1)=\frac{1}{1-Q}\cdot \sum_{\square_q\subset R}\, (-1)^q \, Q^{w_h(\square_q)}.
\end{equation}

\noindent (c) $Pol_{SW}(Q):= PE_\infty (1,Q, -1)-\frac{1}{1-Q} $ is a polynomial in $Q$, whose value at
$Q=1$ is
$eu(\bH_*(\frX,w_h))$, the normalized Seiberg-Witten invariant of $M$ associated with the spin$^c$ structure $h*\sigma_{can}$.

Hence, $PE_\infty(T,Q,\h)$ is a high multigraded generalization of the Seiberg-Witten invariant.
\end{proposition}
In fact,  $ PE_\infty(1,Q,-1)=PE_1(1,Q,-1)$ and part {\it (b)} generalizes as
\begin{equation}\label{eq:2}
PE_1(T,Q,-1)=
\frac{1}{1-Q}\cdot\sum_{l\in L_{\geq 0}}\, T^{-(s,l)}\sum_{I\subset \cV} (-1)^{|I|}Q^{w_h((l,I))}.
\end{equation}
Here $\square =(l,I)$ is a $|I|$--cube with vertices $\{l+\sum_{j\in J}E_j\}_{J\subset I}$.

Note that in  (\ref{eq:2})
the summation is over cubes, and $\square\mapsto \max\{w_h(v)\,:\, v \ \mbox{is a vertex of $\square$}\}$ a priori is a complicated irregular arithmetical function. Still, in Theorem \ref{th:form_int}, by proving a certain regularity behaviour of this function, we replace the cube--summation by a summation over lattice points (with Jacobi theta series type summands).

 \bekezdes Regarding motivation ({\bf III}) we have the following two structure results:

\begin{theorem}\label{th:peinfty_int} Assume that ${\rm gcd}\{n_i\}_i=1$ (for the general statement see
Theorem \ref{th:peinfty}).
Let $N\in \Z_{>0}$ be the smallest integer such that $\tilde{s}:=Ns\in L$ and
set $p:=-(s,\tilde{s})$.
 Then there exist lattice points $\{l_q\}_{q=0}^{p-1}$ such that
$PE_\infty(T,Q,\h)$ --- up to finitely many terms ---  has the form
\begin{equation*}
\frac{T^{d_0}}{1-Q}\, \cdot  \Big[\, Q^{n_0}+(T-1) \cdot \sum_{q=0}^{p-1}\ T^{q}Q^{\chi_h(l_r)}\,
\sum_{m\geq 0}\  T^{mp}Q^{-(m\tilde{s}, m\tilde{s}+k_h+2l_q)/2}\,\Big].
\end{equation*}
\end{theorem}

For the $E^1$ pages it is more convenient to pack all the series $\{PE_{h,1}\}_{h\in H}$ in a single
series
belonging to  $\Z[[T^{1/|H|}, Q^{1/|H|}]]\,[Q^{-1}]\,[H]$.
\begin{theorem}\label{th:form_int}
There exist a finite index set ${\cP}$, integers $\{a_\cP\}_\cP$, $\{b_\cP\}_\cP$, sublattices $\Z^{s_\cP}$ of $L'$ and elements
$\{k_\cP\}_\cP,\ \{r_\cP\}_\cP\in L'$ such that
\begin{equation*}\begin{split}
\sum_{h\in H}\  T^{-(s,s_h)} Q^{\chi(s_h)}\cdot
& PE_{h,1}(T,Q,-1)[h]=\\
&\sum_\cP \ \sum_I \ (-1)^{|I|}\ T^{a_\cP}Q^{b_\cP}\cdot \
 \sum _{A\in (\Z_{\geq 0})^{s_\cP}}\, T^{-(s,A)}\, Q^{-(A,A+k_\cP)/2}\ [r_\cP+A].
\end{split}\end{equation*}
\end{theorem}

\section{Preliminaries. The general setup}

\subsection{Resolutions of surface singularities and resolution graphs}\label{ss:res}
\bekezdes
Let $(X,o)$ be a complex analytic
normal surface singularity whose link $M$ is a rational homology sphere.
Recall that a resolution of $(X,o)$ is a proper analytic map $\phi:\widetilde{X}\to X$,
where  $\widetilde{X}$ is smooth and  $\phi \big|_{\phi^{-1}\left( X \setminus o \right)}: \phi^{-1}\left( X \setminus o \right) \rightarrow  X \setminus o$ is an isomorphism. For any resolution $\phi$ let $\cup_{i\in\calv}E_i$ be
the irreducible decomposition of the {\it exceptional curve} $E:=\phi^{-1}(o)$.

We order the resolutions as follows: we say that $\phi_1:\widetilde{X}_1\to X$ dominates
$\phi_2:\widetilde{X}_2\to X$, if there is a regular map $\psi:\widetilde{X}_1\to \widetilde{X}_2$ such that $\phi_2\circ \psi=\phi_1$. In such cases $\psi$ is an iterated blow up of infinitely near points of
$\phi_2^{-1}(o)$. For any $(X,o)$ there exists a unique minimal resolution, that is, any resolution dominates this minimal one.

In our  topological discussions we will use the resolution graph $\Gamma_\phi$ associated with a certain  resolution $\phi$. Doing this usually one assumes that  $\phi$ is a `good' resolution, that is,
each $E_v$ is smooth and $E$ is a normal crossing divisor. In any such situation we define the dual graph, and we order such graph similarly: $\Gamma_1$ dominates $\Gamma_2$ if $\Gamma_1$ is obtained from $\Gamma_2$ by a sequence of
blow ups of graphs. Clearly, if we have two resolutions, then $\phi_1$ dominates $\phi_2$ if and only if
$\Gamma_{\phi_1}$ dominates $\Gamma_{\phi_2}$. Note also that there is a unique minimal good resolution, hence a  unique  minimal resolution graph.
Recall also that the resolution graphs serve as plumbing graphs for the link $M$.
Since  $M$  is a rational homology sphere, necessarily
each $E_v$ is rational and the dual graph  is a tree.

\bekezdes\label{bek:invs}
Once a good resolution is fixed we will use the following notations. All of them depend merely on the graph $\Gamma_\phi$.

The lattice $L:=H_2(\widetilde{X},\mathbb{Z})$ is  endowed
with the natural  negative definite intersection form  $(\,,\,)$. It is a
free $\Z$--module generated by the classes of  $\{E_i\}_{i\in\mathcal{V}}$.
 The dual lattice is $L'={\rm Hom}_\Z(L,\Z) \simeq\{
l'\in L\otimes \Q\,:\, (l',L)\in\Z\}\subset L\otimes \Q$.
It  is generated
by the (anti)dual classes $\{E^*_i\}_{i\in\mathcal{V}}$ defined
by $(E^{*}_{i},E_j)=-\delta_{ij}$ (where $\delta_{ij}$ stays for the  Kronecker symbol).
$L'$ is also  identified with $H^2(\tX,\Z)$. 
$L$ is embedded in $L'$ with
 $ H:=L'/L\simeq H_1(M,\mathbb{Z})$. We denote the class  of $l'\in L'$ in $L'/L$  by $[l']$.

There is a natural partial ordering of $L'$ and $L$: we write $l_1'\geq l_2'$ if
$l_1'-l_2'=\sum _i c_iE_i$ with every  $c_i\geq 0$.
We set $L_{\geq 0}=\{l\in L\,:\, l\geq 0\}$ and
$L_{>0}=L_{\geq 0}\setminus \{0\}$.

We define the {\it Lipman cone} as $\calS':=\{l'\in L'\,:\, (l', E_i)\leq 0 \ \mbox{for all $i$}\}$, and we  also
set $\calS:=\calS'\cap L$.
As a monoid $\calS'$  is generated over $\bZ_{\geq 0}$ by $\{E^*_i\}_i$.
 If $s'\in\calS'\setminus \{0\}$, then
all its $E_i$--coordinates  are strictly positive. Thus, if $s \in \cS \setminus \{ 0 \}$, then $s \geq E=\sum_iE_i$.

For any $h\in L'/L\simeq H_1(M,\mathbb{Z})$ there exists a unique minimal element $s_h\in \calS'$ (with respect to the ordering $\leq $) such that $[s_h]=h$ \cite{NOSz}.

The {\it anticanonical cycle} $Z_K\in L'$ is defined by the
{\it adjunction formulae}
$(Z_K, E_i)=(E_i,E_i)+2$ for all $i\in\mathcal{V}$. The set of {\it characteristic elements} are
${\rm Char}:=\{k\in L': \ (x,x+k)\in 2\Z\ \mbox{for all $x\in L$}\}$. Clearly, $-Z_K\in{\rm Char}$. In fact,
${\rm Char}=-Z_K+2L'$. For any $h\in H$ we fix a {\it distinguished characteristic element }  defined as
$k_h:=-Z_k+2s_h\in {\rm Char}$.

The set ${\rm Spin}^c(M)$ of spin$^c$ structures of $M$ is an $H=H_1(M,\Z)$ torsor. If $\sigma_{can}$ denotes the
canonical spin$^c$ structure of $M$, then the correspondence  $h\mapsto h*\sigma_{can}$ identifies $H$ with ${\rm Spin}^c(M)$,
cf. \cite{NOSz,Nkonyv}.

In \ref{ss:3.1} we will also introduce the Riemann--Roch
weight functions and  the multivariable topological Poincar\'e (zeta)  function together with  its `motivic' version.

\subsection{Embedded curve singularities in $(X,o)$}

Next, we fix  a reduced Weil divisor $(C,o)$ in $(X,o)$. We can also think about it  as
an isolated  curve singularity in $(X,o)$.
Then we can consider the set of good embedded resolutions of the pair $(C,o)\subset (X,o)$, that is, we require that the resolution is good, and that the union of  $E$ with the strict transform $\tilde{C}$ of $(C,o)$ from a normal crossing divisor.
It is known that there is a unique  good minimal  embedded resolution, and any other one is obtained from this by blowing up. Similarly, we can consider the corresponding embedded resolution graphs with their dominance ordering.

Usually, the strict  transform will be denoted by arrows on the graph.
Their index set will be denoted by $\cA=\{1, \ldots , r\}$ corresponding to the
irreducible decomposition $\cup_{a=1}^rC_a$ of $(C,o)$.

The analytic pair $(C,o)\subset (X,o)$ provides a topological pair $L_C\subset M$, the embedded link
$L_C$  of $(C,o)$ into $M$. This embedded topological type can be read from any embedded resolution graph, and it determines the minimal good embedded resolution graph of the pair. However,
when we {\it blow up the graphs} we face an ambiguity regarding  the position of the arrows
on the graph.

Consider e.g. the following minimal good embedded resolution graph (the left graph below). It is realized by
$(X,o)=\{x^3=yz\}\subset (\C^3,o)$, and $(C,o)=\{x=y=0\}$.

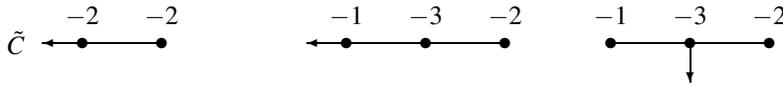
\begin{picture}(320,50)(0,5)

\put(30,30){\circle*{4}} \put(60,30){\circle*{4}}
\put(60,30){\vector(-1,0){45}}
\put(5,30){\makebox(0,0){\small{$\tilde{C}$}}}
\put(30,40){\makebox(0,0){\small{$-2$}}}
\put(60,40){\makebox(0,0){\small{$-2$}}}

\put(130,30){\circle*{4}} \put(160,30){\circle*{4}} \put(190,30){\circle*{4}}
\put(190,30){\vector(-1,0){75}}
\put(130,40){\makebox(0,0){\small{$-1$}}}
\put(160,40){\makebox(0,0){\small{$-3$}}}
\put(190,40){\makebox(0,0){\small{$-2$}}}

\put(230,30){\circle*{4}} \put(260,30){\circle*{4}} \put(290,30){\circle*{4}}
\put(290,30){\line(-1,0){60}}\put(260,30){\vector(0,-1){15}}
\put(230,40){\makebox(0,0){\small{$-1$}}}
\put(260,40){\makebox(0,0){\small{$-3$}}}
\put(290,40){\makebox(0,0){\small{$-2$}}}

\end{picture}

The second and the third graphs are obtained from the minimal one by blow up.
In the first case we blow up the intersection point of $\tilde{C}$ with $E$, in the second case we blow up a generic point of that component which supports $\tilde{C}$. In the graph language: in the first case we blow up the edge of the arrow, in the second case
the vertex supporting the arrow.
Both non-minimal graphs represent topologically the very same pair  $(M,L_C)$, still,
from analytic point of view, or even at the level of rational (Mumford) divisors associated with  the strict transforms $C_a$, they are different. E.g., in the lattice $L'$ of the resolution
the cycles  associated with the arrows in the two cases are different (that is, if the arrow is supported by the vertex $i(a)$, then the rational cycles $E_{i(a)}^*\in L'$ in the two cases are different).
The filtrations what we will consider later will also
be unstable with respect to blow ups of edges supporting arrowheads.

Since in the definition of several invariants we  will use cycles from $L'$ associated with an
embedded resolution,
we wish to make distinction between the two cases from above, even though  the
 pair $(M,L_C)$ is the same.
For this reason we will introduce a decoration for the components of $(C,o)$.
Motivated by the analytic theory, in the first case we say that we blow up a `base point' (namely $\tilde{C}\cap E$), while in the second case a `non-base point'
(namely, a `generic' point of $E_{i(a)}$).

\subsection{Decorations of the components of $(C,o)$.}\label{ss:1.2}
 Let $\widetilde{X}_C$  be the {\it minimal good embedded resolution} of the pair
 $(C,o)\subset (X,o)$.
Let $C=C_1\cup\cdots \cup C_r$ be the irreducible decomposition of $C$, and we also fix
nonnegative integers $b_1, \ldots, b_r\in \Z_{\geq 0}$, the `lengths  of the base points'.
Then the resolution $\tX_{C,b}$ is obtained from  $\widetilde{X}_C$ by iterated blow ups of infinitely near points of the intersection of the strict transform $\widetilde{C}_a$ of $C_a$ (for any $a$)
with the exceptional curve:
for every $a\in\cA$ the strict transform  of
 $C_a$ is blown up $b_a$ times. We call  all these modifications are {\it base point blow ups}.
We say that the output is the
 {\it minimal good resolution
of the  decorated pair  $(X,C,b)$}.   Then, if we wish to represent this  decorated pair
$(X,C,b)$  by any other resolution, that resolution should be
obtained from $\tX_{C,b}$ by any sequence of blow ups, but all of them with free centers or
centers which are  intersection points of irreducible exceptional divisors (we call them
non-base points).

Similarly, taking the corresponding resolution graphs, we have
the minimal embedded graph  $\Gamma^{min}_{M,C}$
of the (non-decorated) pair $(M,L_C)$ (the graph of $\widetilde{X}_C$). Then  the minimal embedded graph  $\Gamma^{min}_{M,C,b}$
of the decorated pair $(M,L_C,b)$ is  obtained from $\Gamma^{min}_{M,C}$ by blowing up repeatedly
the edge of each $a$--arrow $b_a$ times.
Any other graph which represent   $(M,L_C,b)$ non--minimally, is obtained from $\Gamma^{min}_{M,C,b}$ by iterated  blowing ups
arbitrarily many vertices arbitrarily many times, or edges connecting two non--arrowhead vertices.

Equivalences  of resolutions and graphs correspond to blow ups free centers or intersection points of irreducible exceptional divisors.
For any analytic type $(X,C,o)$ and decoration $b=(b_1,\dots, b_r)$ we have an equivalence class of resolutions which represent
$(X,C,b)$. Similarly, for any $(M,L_C,b)$ we have    an equivalence class of graphs  which represent the {\it decorated} $L_C$ embedded in $M$.

By this notations, in the pictures above, the left graph is $\Gamma^{min}_{M,C,0}$,
the second one is  $\Gamma^{min}_{M,C,1}$ (and they are \underline{not} equivalent), while the far right graph represents nonminimally $(M,C,0)$, in particular, it is equivalent with the far left one.

For a similar situation when decorated curves appear (though with different meaning)
see \cite{dJvS}.

\begin{definition}\label{def:DIV}
 Let $\phi$ be a good embedded resolution of $(X,C,b)$. For each strict transform $\tilde{C}_a$  of $C_a$ ($a\in\cA$) define $i(a)\in\cV_{\phi}$ so that $E_{i(a)}\cap \tilde{C}_a\not=\emptyset$. Then the (Mumford) {\it divisor of\, $C$}, as an  element of the  lattice $L'_\phi$,  is defined as
$${\rm div}_{\phi}(C):= \sum_{a\in\cA}\, E^*_{i(a)}\in \calS'_\phi\setminus \{0\}\subset L'_\phi.$$
\end{definition}
\begin{lemma}\label{lem:div}
Assume that the resolution $\phi_1$ is obtained from $\phi_2$ by a single  blow up $\psi$.

\begin{picture}(320,50)(-50,5)

\put(100,40){\makebox(0,0){\small{$(\widetilde{X}_1,E_1)$}}}
\put(200,40){\makebox(0,0){\small{$(\widetilde{X}_2,E_2)$}}}
\put(120,40){\vector(1,0){60}}
\put(150,45){\makebox(0,0){\small{$\psi$}}}
\put(150,10){\makebox(0,0){\small{$(X,C)$}}}
\put(115,30){\vector(2,-1){25}}
\put(185,30){\vector(-2,-1){25}}
\put(115,20){\makebox(0,0){\small{$\phi_1$}}}
\put(185,20){\makebox(0,0){\small{$\phi_2$}}}

\end{picture}

(a)
If $\psi$ is a non-base point blow up
then $\psi^*({\rm div}_{\phi_2}(C))={\rm div}_{\phi_1}(C)$, where $\psi^*:L'_{\phi_2}\to
L'_{\phi_1}$ is the cohomological pullback (or the pull back of local rational Cartier divisors).

(b)
If $\psi$ is a blow up of  a base point (creating $E_{new}$) then
$\psi^*({\rm div}_{\phi_2}(C))+E_{new}={\rm div}_{\phi_1}(C)$.
\end{lemma}
\begin{proof} {\it (a)}
One verifies that the strict transform of $E_{i(a),\phi_2}$ is $E_{i(a),\phi_1}$
and $\psi^*(E^*_{i(a),\phi_2})=E^*_{i(a),\phi_1}$.
In case  {\it (b)}, if the center belongs to $E_{i(a)}$, then
$\psi^*(E^*_{i(a),\phi_2})+E_{new}=E^*_{i(a),\phi_1}$.
\end{proof}


\bekezdes \label{bek:semi}
Let us take a resolution $\phi$ which represents the equivalence class of a decorated pair $(M,L_C,b)$.  Then $s:={\rm div}_\phi(C)\in \calS'_\phi\setminus \{0\}$ identifies completely (combinatorially) the position of the strict transform $\tilde{C}$. Indeed,
if $(s, E_i)<0$ for some $i\in\cV$, then $E_i$ supports exactly $-(s, E_i)$ arrows of $\tilde{C}$.

We write ${\rm Supp}(s)$ for $\{i\in \cV\,:\, (s, E_i)<0\}$. Its cardinality will be denoted
by $r'$.
It is easy to see that $r'$ is stable in the equivalence class.
Note that $r'\leq r$, and $r'<r$ happens exactly when some $E_i$ supports more arrows.
Later we will introduce series whose variables will be indexed by ${\rm Supp}(s)$, which
in the cases mentioned above can be smaller than the number of components of $(C,o)$, the usual index set (e.g. in multivariable Alexander polynomials) of the variables associated with
  embedded links.
If we wish to eliminate irreducible exceptional components supporting  more arrows (hence to realize the equality $r'=r$), we only need to increase the decorations $b_a$ of the corresponding
arrows. However, doing this, we will pass to another equivalence class, and all the invariants
associated with the class of a decorated pair will be modified, see e.g. Example \ref{ex:p2}.

Let us introduce an equivalence relation $\sim$ of the pairs $(\Gamma_\phi,s)$,
where $\Gamma_\phi$ is a good embedded resolution graph  and $s\in \calS'_\phi\setminus \{0\}$. It is generated by the following relation.
We say that $(\Gamma_{\phi_1},s_1)\sim (\Gamma_{\phi_2},s_2) $  if and only if  $\Gamma_{\phi_1}$ dominates $\Gamma_{\phi_2}$ via $\psi$, and $\psi^*(s_2)=s_1$. Then, by Lemma \ref{lem:div},
the equivalence classes $(\Gamma_\phi,s)/\sim $ \ coincide with the
equivalence classes of embedded good resolution graphs (up to non-base point blow ups) of decorated pairs.

In fact, instead of a pair $(C,b)=(\cup_{a=1}^r C_a, (b_1,\ldots , b_r))$ one can take several
pairs, namely  $(C_{(j)},b_{(j)})_{j=1}^t$ ($t\geq 1$), where each $(C_{(j)}, b_{(j)})$ has the form
$(\cup_{a=1}^{r_j}C_{j,a}, (b_{j,1},\ldots , b_{j,r_j}))$,
and their equivalence
relations correspond to the non-base point blow ups.
In the language of semigroups this corresponds to fixing $(\Gamma_\phi, s_1, \ldots s_t)$,
where $\Gamma_\phi$ is a good embedded resolution graph, each $s_j\in \calS'_\phi\setminus \{0\}$,
and their equivalence relation is generated by the following step: $(\Gamma_{\phi_1},\{s_{1,j}\}_{j=1}^t )\sim (\Gamma_{\phi_2},\{s_{2,j}\}_{j=1}^t ) $  if and only if  $\Gamma_{\phi_1}$ dominates $\Gamma_{\phi_2}$ via $\psi$, and $\psi^*(s_{2,j})=s_{1,j}$.

\section{The definition of the lattice homology}\label{ss:latweight}

 \subsection{The lattice homology associated with  a system of weights} \cite{Nlattice}

\bekezdes
 We consider a free $\Z$--module, with a fixed basis
$\{E_i\}_{i\in\calv}$, denoted by $\Z^s$. It is also convenient to fix
a total ordering of the index set $\calv$, which in the
sequel will be denoted by $\{1,\ldots,s\}$.

The lattice homology construction associates
 a graded $\Z[U]$--module with the
pair $(\Z^s, \{E_i\}_i)$ and with a set of weights.
The construction follows closely the construction of the lattice cohomology developed in \cite{Nlattice} (for more see also \cite{NOSz,NGr,Nkonyv}). For details of the homological version see \cite{NemCurves}.
Both versions can be defined in two equivalent ways: either via a chain complex, or via the construction of certain finite
cubical subspaces spaces $\{S_n\}_n$ of $\R^s$.
Here we review in short the second construction.

\bekezdes\label{9zu1} {\bf $\Z[U]$--modules.}  We will modify the usual grading of the polynomial
ring $\Z[U]$ in such a way that the new degree of $U$ is $-2$.
Besides $\calt^-_0:=\Z[U]$, considered as a graded $\Z[U]$--module,  we will consider the modules
$\calt_0(n):=\Z[U]/(U^n)$ too with the induced grading.
 Hence, $\calt_0(n)$, as a $\Z$--module, is freely
generated by $1,U^1,\ldots,U^{n-1}$, and it has finite
$\Z$--rank $n$.

More generally, for any graded $\Z[U]$--module $P$ with
$d$--homogeneous elements $P_d$, and  for any  $k\in\Z$,   we
denote by $P[k]$ the same module graded in such a way
that $P[k]_{d+k}=P_{d}$. Then set $\calt^-_k:=\calt^-_0[k]$ and
$\calt_k(n):=\calt_0(n)[k]$. Hence, for $m\in \Z$,
$\calt_{-2m}^-=\Z\langle U^{m}, U^{m+1},\ldots\rangle$ as a $\Z$-module.

\bekezdes\label{9complex} {\bf The construction of $\bH_*$.}
$\Z^s\otimes \R$ has a natural cellular decomposition into cubes. The
set of zero-dimensional cubes is provided  by the lattice points
$\Z^s$. Any $l\in \Z^s$ and subset $I\subset \calv$ of
cardinality $q$  defines a $q$-dimensional cube  $\square_q=(l, I)\subset \R^s$, which has its
vertices in the lattice points $\{l+E_J\}_J$, where  $E_J:=\sum_{i\in J}E_i$ and
$J$ runs over all subsets of $I$.
The set of
$q$-dimensional cubes defined in this way is denoted by $\calQ_q$
($0\leq q\leq s$).
Next,
 we consider a set of compatible {\em weight
functions} $w=\{w_q\}_q$,
$w_q:\calQ_q\to \Z$  ($0\leq q\leq s$), where

(a) For any integer $k\in\Z$, the set $w_0^{-1}(\,(-\infty,k]\,)$
is finite;

(b) for any $\square_q\in \calQ_q$, $w_q(\square_q)=\max\{w_0(v),\ \mbox{where $v$ runs over the vertices of $\square_q$}\}$.
%
%

For a more general definition see e.g. \cite{Nlattice,Nkonyv}.
In the sequel  we might omit the index $q$ of $w_q$.

Finally,  for each $n\in \Z$ we
define $S_n=S_n(w)\subset \R^s$ as the union of all
the cubes $\square_q$ (of any dimension) with $w(\square_q)\leq
n$. Clearly, $S_n=\emptyset$, whenever $n<m_w$, where $m_w:=\min_{l\in L}\{w(l)\}$. For any  $q\geq 0$, set
$$\bH_q(\R^s,w):=\oplus_{n\geq m_w}\, H_q(S_n,\Z).$$
Then $\bH_q$ is $\Z$ (in fact, $2\Z$)--graded: the
$(-2n)$--homogeneous elements $(\bH_q)_{-2n}$ consist of  $H_q(S_n,\Z)$.
Also, $\bH_q$ is a $\Z[U]$--module; the $U$--action is the homological morphism
$H_q(S_{n},\Z)\to H_q(S_{n+1},\Z)$ induced by the inclusion $S_n\hookrightarrow S_{n+1}$.
Moreover, for
$q=0$, a fixed base-point $l_w\in S_{m_w}$ provides an augmentation
(splitting)
 $H_0(S_n,\Z)=
\Z\oplus \widetilde{H}_0(S_n,\Z)$, hence a splitting of the graded
$\Z[U]$-module
$$\bH_0=\calt^-_{-2m_w}\oplus \bH_{0,red}=(\oplus_{n\geq m_w}\Z)\oplus (
\oplus_{n\geq m_w}\widetilde{H}_0(S_n,\Z)).$$
We also write $\bH_{q,red}=\bH_q$ for $q\geq 1$.


\bekezdes\label{9SSP} {\bf Restrictions.}  Assume that $\frR\subset \R^s$ is a subspace
of $\R^s$ consisting of a union of certain cubes.
Then instead of $\oplus_{n\geq m_w}\, H_q(S_n,\Z)$ we can take the following module
 $\oplus_{n\geq m_w}\, H_q(S_n\cap \frR,\Z)$. It has a natural graded $\Z[U]$--module structure and augmentation.
It will be denoted by  $\bH_{*}(\frR,w)$.

In some cases it can happen that the weight functions are defined only for cubes belonging to $\frR$.

Some of the possibilities (besides $\frR=\R^s$) used in  the present note are  the following:

(1) $\frR=(\R_{\geq 0})^s$ is the first quadrant; 

(2) $\frR=R(0,c)$ is the rectangle $\{x\in\R^s \,:\, 0\leq x\leq c\}$,  where $c\geq 0$ is a lattice point;

(3) $\frR=\{x\in\R^s \,:\,  x\geq l\}$ for some fixed $l\in (\Z_{\geq0})^s$.

\bekezdes \label{bek:eu}{\bf The Euler characteristic of $\bH_*$  \cite{JEMS}.}
Though
$\bH_{*,red}(\R^s,w)$ has a finite $\Z$-rank in any fixed
homogeneous degree, in general, it is not
finitely generated over $\Z$, in fact, not even over $\Z[U]$.

Let $\frR$ be  as in \ref{9SSP} and assume that each $\bH_{q,red}(\frR,w)$ has finite $\Z$--rank.
(This happens automatically when $\frR$ is a finite rectangle.)
We define the Euler  characteristic of $\bH_*(\frR,w)$ as
$$eu(\bH_*(\frR,w)):=-\min\{w(l)\,:\, l\in \frR\cap \Z^s\} +
\sum_q(-1)^q\ \rank_\Z(\bH_{q,red}(\frR,w)).$$
If $\frR=R(0, c)$ (for a lattice point $c\geq 0$), then by  \cite{JEMS},
\begin{equation}\label{eq:eu}
\sum_{\square_q\subset \frR} (-1)^{q+1}w(\square_q)=eu(\bH_*(\frR,w)).\end{equation}

\section{The main construction}

\subsection{The setup, the topological  weight functions and topological Poincar\'e series}\label{ss:3.1}\

Let us fix  a resolution graph $\Gamma$ and semigroup elements $s_1, \ldots , s_k\in\calS'\setminus \{0\}$. Recall that we assume that $\Gamma$ is a tree of rational vertices, i.e. $M$ is a $\Q HS^3$.

Then we proceed with the construction of the lattice homology.
In this note we will consider the {\it  topological} case. However,
if the topological weight function defined below (which agrees with the
weight function of \cite{Nlattice,NGr}) is replaced by the `analytic' weight
function $w_{an,h}$ considered in \cite{AgostonNemethiI,AgostonNemethiIII}, then al the theory presented here for the topological case can be reproduced in the analytic setup.
The details will be presented in a forthcoming article.

Let $L$ be as in \ref{bek:invs}: it is  the free lattice
generated by the vertices of $\Gamma$. We denote the canonical basis by $\{E_i\}_{i\in\cV}$
and the intersection form by $(\,.\,)$.
We also fix a class $h\in L'/L$ (that is, a spin$^c$ structure on the link, cf. \cite{NOSz,Nkonyv}).
Let the distinguished characteristic element $-Z_K+2s_h$ associated with  $h$ be denoted by $k_h$.
Next,   we define the
weight function $\chi_{h}:L\to \Z$ by $\chi_{h}(l)=-\frac{1}{2}(l, l+k_h)$.
We consider the cubical decomposition of $\R^s$, and we define the weight of a $q$--cube $\square_q=(l,I)$ by
$w_h(\square_q):=\max\{\chi_{h}(v), \ \mbox{where $v$ is a vertex of $\square_q$}\}$, cf. \ref{9complex}. For $h=0$ we write $\chi_h=\chi$.

Associated with $\Gamma$ we will also consider its (topological) multivariable Poincar\'e series
in variables $\{t_i\}_{i\in \cV}$ as well: $Z(\bt)=\sum_{l'\in L'}\zeta(l')\bt^{l'}\in\Z[[t_1^{1/|H|},\ldots, t_s^{1/|H|} ]]$
is the Taylor expansion at the origin
of $$Z(\bt):= \prod_{i\in\cV}\, \big( 1-\bt^{E^*_i}\big)^{\kappa_i-2},$$
where $\kappa_i$ is the valency of $i\in\cV$ (i.e. $(E_i,E-E_i)$), and
for any $l'=\sum_il'_iE_i\in L'$ we write $\bt^{l'}:=
\prod_i t_i^{l'_i}$, see e.g. \cite{CDG,CDGEq,CDGc,LN,NPo,JEMS} or \cite[8.4]{Nkonyv}. From definition it follows that
$\zeta(l')=0$ whenever $l'\not\in \calS'$.

The series $Z(\bt)$ has a natural $H$-sum decomposition (sometimes called `equivariant decomposition'), $Z(\bt)=\sum_{h\in H} Z_h(\bt)$, where
$Z_h(\bt):= \sum_{[l']=h}\zeta(l')\bt^{l'}$.

In fact, by Fourier transformation, see e.g. \cite[(8.4.2)]{Nkonyv}, one has
$$Z_h(\bt)=\frac{1}{|H|}\cdot \sum_{\rho\in\widehat{H}}\, \rho(h)^{-1}\cdot
\prod_{i\in\cV} \big(1-\rho([E^*_i])\bt^{E^*_i}\,\big)^{\kappa_i-2},$$
 where $\widehat{H}$ denotes the Pontrjagin dual (group of characters) ${\rm Hom}(H,S^1)$ of $H$.

The following relation connects  $Z_h(\bt)$ and the weight functions $w_h$
(see \cite[Th. 11.4.6]{Nkonyv}):
\begin{equation}\label{eq:Zhw}
Z_h(\bt)=\sum_{l\in L}\, \Big( \sum_{I\subset \cV} (-1)^{|I|+1} w_h((l,I))\,\Big) \cdot \bt^{l+s_h}.
\end{equation}
The series $Z_h(\bt)$ has the following `lifting' to $\Z[[t_1^{1/|H|}, \ldots,t_s^{1/|H|},q]]$:
\begin{equation}\label{eq:Zhmot}
Z^m_h(\bt,q):=\frac{1}{1-q}\cdot \sum_{l\in L}\ \sum_{I\subset \cV} \, (-1)^{|I|}\,q^{w_h((l,I))}\cdot
\bt^{l+s_h}.
\end{equation}
Indeed, since $\sum_I(-1)^{|I|}=0$, we have
$$\lim_{q\to 1} \, Z^m_h(\bt,q)=\lim_{q\to 1} \,\sum_l\ \sum_I (-1)^{|I|}\cdot \frac{q^{w_h((l,I))}-1}
{1-q} \cdot \bt^{l+s_h}=Z_h(\bt).$$
Regarding the $\bt$--support of $Z^m_h$ we have $\sum_{I\subset \cV} \, (-1)^{|I|}\,q^{w_h((l,I))}=0$ whenever $l+s_h\not\in\calS'$, cf. \ref{lem:sup}.

This equivariant `motivic' extension $Z^m_h(\bt,q)$
is not completely the one from \cite[8.4.A, 8.4.B]{Nkonyv} (where one considers the class of the complement of linear subspace arrangements in the Grothendieck ring); the relationship between the two approaches will
be discussed in another note. (Conceptually it differs from the ones introduced in the curve case too, which are still in the spirit of the Grothendieck ring extension;  for such series   see e.g. \cite{cdg3,Gorsky} and the series of articles of Campillo, Delgado and Gusein-Zade.)

\subsection{The topological lattice homology  and some of its properties}\

Let $S_n$ be the spaces associated with the weight functions $w_h$ as in subsection \ref{ss:3.1}.

\begin{theorem}\label{th:indep} (\cite{Nlattice} or \cite[Prop. 7.3.5]{Nkonyv})
The homotopy type of the tower of spaces $\{S_n\}_n$ and the lattice homology $\bH_*(\R^s,w_h)$ is independent
of the choice of the negative definite plumbing graph $\Gamma$ of $M$, hence it depends only on the link $M$ (and the choice of \ $h\in H=H_1(M,\Z)$).
\end{theorem}
In the sequel we refer to
 $\bH_*(\R^s,w_h)$ as the topological lattice homology of the link $M$ associated with the
spin$^c$ structure corresponding to\ $h$.
\begin{theorem}\label{th:euH}
(a) \cite[Th. 11.1.36]{Nkonyv}
Let $\mathfrak{sw}:{\rm Spin}^c(M)\to \Q$ be the Seiberg--Witten invariant $\sigma\mapsto
\mathfrak{sw}_\sigma(M)$ of $M$. Then
$$eu(\bH_*(\R^s,w_h))=\mathfrak{sw}_{h*\sigma_{can}}(M)-\frac{k_h^2+|\cV|}{8}.$$
(b) \cite[Th. 11.4.6]{Nkonyv} Fix some $ l\in L$ with $l+s_h\in Z_K+\calS'$. Then
$$\sum_{x\in L;\, x\not\geq l}\zeta(x+s_h)=\chi_h(l)+eu(\bH_*(\R^s,w_h)).$$
\end{theorem}

\begin{remark}\label{rem:ZtoI}
Part {\it (b)} of Theorem \ref{th:euH} shows that the series $Z_h(\bt)=\sum_{[l']=h}\zeta(l')\bt^{l'}$
determines $\chi_h(l)+eu(\bH_*)$ for any  $l+s_h\in Z_K+\calS'$. On the other hand,
$\chi_h(l'+E_i+E_j)-\chi_h(l'+E_i)-\chi_h(l'+E_j)+\chi_h(l')=-(E_i,E_j)$. Hence,
$Z_h$ determines the intersection form, hence all the topological invariants (including the
motivic $Z^m_h(\bt,q)$ too).

On the other hand, part {\it (a)} of Theorem \ref{th:euH} shows  that any filtration or
grading  of $\bH_*$ provides automatically a `grading', or
sum decomposition, of the normalized Seiberg--Witten invariant $\mathfrak{sw}_{h*\sigma_{can}}(M)-(k^2_h+|\cV|)/8$ \ as well.
\end{remark}

\bekezdes
In the computation or identification of $\bH_*$  {\bf reduction theorems} play key roles.
The first such reduction statement  is the following.
\begin{proposition}\label{lem:con} (\cite{NOSz} or \cite[Lemma 7.3.12]{Nkonyv})
For any $n$ we have the homotopy equivalence $S_n\cap (\R_{\geq 0})^s\hookrightarrow S_n$, and the graded $\Z[U]$--module isomorphism
$\bH_*((\R_{\geq 0})^s, w_h)\simeq\bH_*(\R^s, w_h)$.

In fact, one can restrict even more whenever $Z_K\geq 0$ (e.g. in the minimal good resolution): for any $c\in L$, $c\geq
\lfloor Z_K \rfloor$,  the inclusion $R(0,c)\hookrightarrow (\R_{\geq 0})^s$  induces
a homotopy equivalence $S_n\cap R(0,c)\hookrightarrow S_n\cap (\R_{\geq 0})^s$ and
an isomorphism of graded $\Z[U]$ modules
$\bH_*(R(0,c), w_h)\simeq \bH_*((\R_{\geq 0})^s, w_h)$. 

This second statement  also shows that $S_n$ is contractible for $n\gg 0$.
\end{proposition}
In particular, in the sequel we will work with the restriction of the theory to the quadrant $\frX:=(\R_{\geq 0})^s$,
and the spaces $S_n$ will denote $S_n\cap (\R_{\geq 0})^s\subset \frX$.
\bekezdes\label{bek:grroot}
The module
$\bH_0=\oplus_n H_0(S_n,\Z)$ can be enhanced  by  its `{\bf homological graded root}' $\mathfrak{R}_h$.
For the definition of the `cohomological graded root'
and its relation with $\bH^0$ see e.g. \cite{NGr,Nlattice,Nkonyv}.
For the homological version  see \cite{NemCurves}.
A graded root is an infinite tree with $\Z$--graded/weighted  vertices (and some additional properties, see
[loc.cit]).
The vertices weighted by  $(-n)$ of the root  correspond to the connected components of $S_n$. Moreover,
any edge connects two vertices with weight difference one and these edges
 correspond exactly
  to inclusions of connected components of $S_{n-1}$ into connected components of $S_n$ for some $n$.

The  relation of the graded root
and  $\bH_0$ is the following.
Each vertex $v$ weighted by $-w_h(v)$ in the root denotes a free summand $\Z=\Z\langle 1_v\rangle\in (\bH_0)_{-2w_h(v)}$, and if $[v,u]$ is an  edge of the root
connecting the vertices $v$ and  $u$ with $w_h(v)=w_h(u)+1$, then $U(1_u)=1_v\in (\bH_0)_{-2w_h(v)}$.
For different examples (and more explanations) see below.

By Theorem \ref{th:indep}, $\mathfrak{R}_h$ depends only on the link $M$ and the choice of $h\in H$ (i.e. of the choice of the
spin$^c$ structure) and not on the choice of the particular negative definite graph $\Gamma$.

\subsection{The `hat'--homology $\widehat{\bH}_*$.}\label{ss:hat}
Using the tower of  spaces $\{S_n\}_n$ we can   define  the sequence of relative homologies as well:
$$\widehat{\bH}_b(\frX,w_h):= \oplus_{n}\, H_b(S_n,S_{n-1},\Z).$$
It is a graded $\Z$--module, with $(-2n)$--homogeneous summand $H_b(S_n,S_{n-1},\Z)$. The $U$--action
induced by the inclusion of pairs $(S_{n}, S_{n-1})\hookrightarrow (S_{n+1},S_{n})$  is trivial.
By Proposition  \ref{lem:con},  $ H_b(S_n,S_{n-1},\Z)\not=0$ only for finitely many pairs $(n,b)$.
By a  homological argument, based on the contractibility of $S_n$ for $n\gg 0$ (see also
Remark \ref{rem:hat} too)
\begin{equation}\label{eq:hateu}
\sum_{n,b}\ (-1)^b \,{\rm rank}\, H_b(S_n, S_{n-1},\Z)=1.
\end{equation}
One also has the exact sequence
$\cdots \to \bH_b\stackrel{U}{\longrightarrow} \bH_b\longrightarrow \widehat{\bH}_b\longrightarrow \bH_{b-1}
\stackrel{U}{\longrightarrow} \cdots$.

\subsection{The filtration}
Next, we focus on the semigroup elements $s=\{s_1,\ldots, s_t\}$, $s_j\in\calS'\setminus \{0\}$.
They provide a $\Z^t$--filtration on each $S_n$.
Indeed, for any fixed $n$ and $d=(d_1,\ldots, d_t)\in (\Z_{\geq 0})^t$ set
$$\frX_{-d}=\frX_{-d}(s)=\frX_{-(d_1, \ldots, d_t)}(s_1, \ldots, s_t):=
\cup\{(l,I)\,:\,
l\in (\Z_{\geq 0})^{|\calv|},\ \mbox{and} \ (s_j,l)\leq -d_j \ \mbox{for all $j$}\}\subset \frX.$$
Note that if $(d_1',\ldots, d_t')\geq (d_1,\ldots, d_t)$ then  $\frX_{-(d_1', \ldots, d'_t)}(s)\subset
\frX_{-(d_1, \ldots, d_t)}(s)$.

Since  the vertices of $(l,I)$ consists of lattice points of type
$\{l+E_J\}_{J\subset I}$,  if $(s_j,l)\leq -d_j$ then necessarily   $(s_j,l+E_J)\leq -d_j$ as well. In particular,
if $l\in\frX_{-d}$ then $(l,I)\in \frX_{-d}$ too for any $I\subset \cV$.

Then, for any collection of $s$  and $d$ we have the augmented and graded $\Z[U]$--module
$$\bH_*(\frX_{-d}(s), w_h)=\oplus_n \, H_*(S_n\cap \frX_{-d},\Z)$$ and the graded $\Z[U]$--module morphism
$\bH_*(\frX_{-d}(s), w_h)\to \bH_*(\frX, w_h)$ induced by the inclusion $\frX_{-d}(s)\hookrightarrow \frX$.

In the next paragraph we will show that the homotopy type of the spaces $S_n\cap \frX_{-d}$
are stable with respect to the $\sim$-equivalence relation and a (restricted) reduction procedure.

\begin{example}
Let $t=|\cV|$ and set $s=\{s_i\}_{i\in\cV}:=\{E^*_i\}_{i\in\cV}$. Then for  any $l=\sum_il_iE_i$ and $d=\{d_i\}_{i\in\cV}$, $(s_i,l)\leq -d_i$ reads as
$l_i\geq d_i$. Hence, $\frX_{-d}(s)$ is the shifted first quadrant $\{l:\, l_i\geq d_i\ \mbox{for all $i$}\}$.
\end{example}

\bekezdes {\bf Stability with respect to $\sim$.}
Let us fix now the class of $(X,C,b)$ and consider two representatives $\Gamma_{\phi_1}$ and $\Gamma_{\phi_2}$
such that $\Gamma_{\phi_1}$ is obtained from $\Gamma_{\phi_2}$ by a non-base point blow up $\psi$.
Consider also    $\{s_{1,j}\}_{j=1}^t $, elements of $\calS'_{\phi_1}$, and
 $\{s_{2,j}\}_{j=1}^t $, elements of $\calS'_{\phi_2}$,
such that  $\psi^*(s_{2,j})=s_{1,j}$.

\begin{theorem}\label{th:eq}
The morphism $\psi_*:L_{\phi_1}\to L_{\phi_2}$  induces homotopy equivalences of (pair of) spaces:

(a) $S_{1,n}\to S_{2,n}$,

(b) $(S_{1,n}, S_{1,n-1}) \to (S_{2,n}, S_{2,n-1})$,

(c)  $S_{1,n}\cap \frX_{-d}(\{s_{1,j}\}_j)\to S_{2,n}\cap \frX_{-d}(\{s_{2,j}\}_j)$,

(d) $(S_{1,n}\cap \frX_{-d}(\{s_{1,j}\}_j), S_{1,n}\cap \frX_{-d-1}(\{s_{1,j}\}_j))
\to ( S_{2,n}\cap \frX_{-d}(\{s_{2,j}\}_j), S_{2,n}\cap \frX_{-d-1}(\{s_{2,j}\}_j))$,

(e) $(S_{1,n}\cap \frX_{-d}(\{s_{1,j}\}_j),
 S_{1,n-1}\cap \frX_{-d}(\{s_{1,j}\}_j))
 \to (S_{2,n}\cap \frX_{-d}(\{s_{2,j}\}_j),
 S_{2,n-1}\cap \frX_{-d}(\{s_{2,j}\}_j)).$
\end{theorem}
\begin{proof}
By  \cite[Prop. 7.3.5]{Nkonyv}, see also Theorem \ref{th:indep}, the restriction
of $\psi_*\otimes  \R$ to $S_{1,n}$ has its image in $S_{2,n}$,
it  is surjective onto $S_{2,n}$,  and it
is a homotopy equivalence.
Furthermore, it is compatible with the inclusions
$ S_{1,n-1}\hookrightarrow S_{1,n}$,   $ S_{2,n-1}\hookrightarrow S_{2,n}$.
Then the equivalence of the pairs follows by the `five
lemma' associated with the long exact sequences of the homotopy  groups.

Note also that the proof of \cite[Prop. 7.3.5]{Nkonyv} valid for  $\psi_*:S_{1,n}\to S_{2,n}$ works without modification for
 $S_{1,n}\cap \frX_{-d}(\{s_{1,j}\}_j)\to S_{2,n}\cap \frX_{-d}(\{s_{2,j}\}_j)$ too.
 In fact, $\psi_*: S_{1,n}\to S_{2,n} $ is a quasi-fibration, hence  the  restriction
  $S_{1,n}\cap \frX_{-d}(\{s_{1,j}\}_j)\to S_{2,n}\cap \frX_{-d}(\{s_{2,j}\}_j)$  remains  surjective and a quasi-fibration.
 Here we need $(\psi_*)^{-1}(\frX_{-d}(\{s_{2,j}\}_j))=\frX_{-d}(\{s_{1,j}\}_j)$,  which follows from
  $\psi^*(s_{2,j})=s_{1,j}$.
\end{proof}

\subsection{The filtered  reduction theorems}\label{ss:redth}

\bekezdes In this subsection we show that the Reduction Theorem
of topological lattice homology is compatible with certain filtrations.
In the first part of the next discussion we  follow \cite{LThesis,LN,Nkonyv}.

The reduction theorem reduces the rank of the lattice which produces the cubical decomposition.
 The rank of the new lattice is the cardinality of a set of `bad' vertices.
 Recall that a subset of vertices $\overline{\cV}$ is called SR--set (or `bad'), if by replacing the
 Euler decorations $e_i$ of the vertices $i\in \overline{\cV}$ by more negative integers
 $e_i'\leq e_i$ we get a rational graph.
 Let $\bar{s}$ be the cardinality of $\overline{\cV}$.
  A graph is called AR--graph ( {\bf `almost rational'})
 if it admits an SR--set of cardinality less than or equal to one.

 Once $\overline{\cV}$ is fixed,
 we consider the lattice  $\overline{L}$ generated by $\{E_{i}\}_{i\in\overline{\cV}}$, the first  quadrant
 $(\R_{\geq 0})^{\bar{s}}$ of $\overline{L}\otimes\R$ and its cubical decomposition. For any
$\bar{l}=\sum_{i\in \overline{\cV}} l_iE_i \in(\Z_{\geq 0})^{\bar{s}}\subset \overline{L}$ one defines an universal cycle $x_h(\bar{l})$, depending on $\bar{l}$ and $h$, with the following properties. For any $i\in\overline{\cV}$ the $E_i$--coefficient of $x_h(\bar{l})$ is $l_i$, $(x_h(\bar{l})+s_h, E_i)\leq 0$ for any $i\not\in\overline{\cV}$, and $x_h(\bar{l})$ is minimal with these two properties (cf. \cite[7.3.24]{Nkonyv}). Then  we define the wight function $\overline{w}_h(\bar{l}):= \chi_h(x_h(\bar{l}))$. This extends to the cubes of $\overline{L}\otimes\R$ via the usual rule: $\overline{w}_h(\square)=\max\{\overline{w}_h(v)\}$, where $v$ runs over the vertices of $\square$.

\begin{theorem}\label{th:redth}\cite{LThesis,LN} {\bf (Reduction Theorem)}
$\bH_*((\R_{\geq 0})^{s}, w_h)\simeq \bH_*((\R_{\geq 0})^{\bar{s}}, \overline{w}_h)$.
\end{theorem}

The above isomorphism is induced by the natural projection $pr: (\R_{\geq 0})^{s}\to (\R_{\geq 0})^{\bar{s}}$, $\sum_{i\in\cV}l_iE_i\mapsto\sum_{i\in\overline{\cV}}l_iE_i$. Indeed,  $pr$
 induces a homotopy equivalence
$S_n(w_h)\to S_n(\overline{w}_h)$ for every $n\in\Z$.

This homotopy equivalence is compatible with filtrations of the following type.
Let us consider  $s=\{s_1,\ldots, s_t\}$, where each $s_j\in \calS'\setminus \{0\}$ has the form
$s_j=\sum_{i\in\overline{\cV}}n_{j, i} E^*_i$ with $n_{j,i}\in\Z_{\geq 0}$ (i.e.
${\rm Supp}(s_j)\subset \overline{\cV}$). Then
each $s$ defines a filtration in both $L$ and $\overline{L}$ via the same rule:
 $\frX_{-(d_1,\ldots, d_t)}=\{l\in (\R_{\geq 0})^{s}\,:\, ( s_j,l)\leq -d_j\ \mbox{for all $j$}\}$ and
 $\overline{\frX}_{-(d_1,\ldots, d_t)}=\{\bar{l}\in (\R_{\geq 0})^{\bar{s}} \,:\, ( s_j,\bar{l})\leq -d_j\ \mbox{for all $j$}\}$.

 The improvement of the above Reduction Theorem is the following.

 \begin{theorem}\label{th:redth2} {\bf (Filtered  Reduction Theorem)}
  The projection map  induces a homotopy equivalence
  $pr:S_n(w_h)\cap \frX_{-(d_1,\ldots , d_t)}\to S_n(\overline{w}_h)\cap \overline{\frX}_{-(d_1,\ldots , d_t)}$
  for any $n$ and $(d_1,\ldots, d_t)$.
\end{theorem}

\begin{proof}
According to the proof of the Reduction Theorem \ref{th:redth}
from \cite{LThesis,LN}, the projection induces a surjection
$pr:S_n(w_h)\to S_n(\overline{w}_h)$, which is a quasifibration with contractible fibers. Note also that the
filtration of $S_n(w_h)$ is the pullback of the filtration of $S_n(\overline{w}_h)$.
In particular,  $S_n(w_h)\cap \frX_{-(d_1,\ldots , d_t)}\to S_n(\overline{w}_h)\cap \overline{\frX}_{-(d_1,\ldots , d_t)}$
is a quasifibration with contractible fibers too.
\end{proof}

\section{The filtration associated with one semigroup element $(t=1)$}\label{s:levfiltr}

\subsection{The submodules ${\rm F}_{-d}\bH_*(\frX,w_h)$.}\

Let us fix a good embedded resolution graph $\Gamma_\phi$, an element $h\in H$ and a semigroup element
$s=\sum_{i\in {\rm Supp(s)}}n_iE_i^*\in \calS'\setminus \{0\}$ (i.e. $t=1$ in the previous notation).
 The element $h$ identifies a spin$^c$ structure of the link and a
system of compatible weight functions $w=w_h$. The semigroup element identifies an increasing
filtration $\{\frX_{-d}\}_{d\geq 0}:= \{\cup(l, I) \,:\, (s,l)\leq -d\}$ of $\frX=(\R_{\geq 0})^{|\calv|}$, hence an
increasing filtration $\{S_n\cap \frX_{-d}\}_{d\geq 0}$ of any $S_n$. Hence we get a graded $\Z[U]$--module
$\bH_*(\frX_{-d},w)=\oplus_n H_*( S_n\cap \frX_{-d},\Z)$ for every $d$ and a sequence
 of graded $\Z[U]$--module morphisms
$\bH_*(\frX,w)\leftarrow \bH_*(\frX_{-1},w)\leftarrow \bH_*(\frX_{-2},w)\leftarrow \cdots.$
For each $b$, the map $\bH_b(\frX,w)\leftarrow \bH_b(\frX_{-d}, w)$ induced at lattice homology level is
homogeneous of degree zero.
These morphisms provide the following filtration of  $\Z[U]$--modules in $\bH_*(\frX,w)$
$${\rm F}_{-d}\bH_*(\frX,w):={\rm im}\big( \bH_*(\frX,w)\leftarrow \bH_*(\frX_{-d},w)\, \big).$$
By Theorem \ref{th:eq} $\{F_{-d}\bH_*(\frX,w)\}_{d}$ is independent of the choice of
$(\Gamma_\phi, s)$ in its class $\sim$. 

\subsection{The homological spectral sequence associated with the subspaces $\{S_n\cap \frX_{-d}\}_d$.}\label{ss:ss} \

Let us  fix $n$.
The morphism  $(\bH_*(\frX_{-d}, w))_{-2n}\to (\bH_*(\frX,w))_{-2n}$ is identical with
the morphisms $  H_*(S_n\cap \frX_{-d}, \Z)\to  H_*(S_n,\Z)$ induced by the inclusion
$S_n\cap \frX_{-d}\hookrightarrow S_n$. In particular, one can analyse the spectral sequence
associated with the filtration $\{S_n\cap \frX_{-d}\}_{d\geq 0}$ of subspaces of $S_n$.
Since $s\in\calS'\setminus \{0\}$ is a nontrivial sum of type
$\sum_in_iE^*_i$ ($n_i\in \Z_{\geq 0}$), and the coefficients of any  $E^*_{i}$ are strict positive (and each $S_n$ is compact) we obtain the following fact.
\begin{lemma}The filtration $\{S_n\cap \frX_{-d}\}_{d\geq 0}$ is finite.
That is, for any $n$ there exists $d$ such that $S_n\cap \frX_{-d}=\emptyset$.
\end{lemma}

This homological  spectral sequence will be denoted by
 $(E^k_{-d,q})_n\Rightarrow (E^\infty_{-d,q})_n$. Its terms $E^1$ and $E^\infty$ are the following:
  \begin{equation}\begin{split}
  (E^1_{-d,q})_n=& H_{-d+q}(S_n\cap \frX_{-d}, S_n\cap \frX_{-d-1},\Z),\\
   (E^\infty_{-d,q})_n=& \frac{(F_{-d}\, \bH_{-d+q}(\frX))_{-2n}}
   { (F_{-d-1}\, \bH_{-d+q}(\frX))_{-2n}}=({\rm Gr}^F_{-d}\, \bH_{-d+q}(\frX)\,)_{-2n}.
   \end{split}\end{equation}

Theorem  \ref{th:eq} and Theorem \ref{th:redth2} imply the following stability statements:

\begin{proposition}\label{prop:stab}
(a)
Each $\Z$--module $(E^k_{-d,q})_{n}$ is well-defined module associated with  $(\Gamma_\phi, s)$ and $h\in H$,  and  it is independent of the choice of $(\Gamma_\phi, s)$ in its $\sim$--class.

(b) If the $E^*$--support ${\rm Supp}(s)$  of $s$ is a subset of an SR--set $\overline{\cV}$,
then for any\, $n\in\Z$\,
the pages $E^k$ ($k\geq 1$) of the homological spectral sequence associated with
$\{S_n(w_h)\cap \frX_{-d}\}_{d} $ agree with the corresponding pages of the homological spectral sequence associated with
$\{ S_n(\overline{w}_h)\cap \overline{\frX}_{-d}\}_d$.
\end{proposition}

 Thus,  for every $1\leq k\leq \infty$, we have the  well-defined Poincar\'e series associated with $(\Gamma_\phi, s)/\sim$:
 $$PE_k(T,Q,\h):=\sum_{d,q,n} \ \rank (E_{-d,q}^k)_{n}\cdot T^dQ^n\h^{-d+q}\in\Z[[P,Q]][Q^{-1},\h].$$
 Note also  that all the coefficients of $PE_k(T,Q,\h)$ are nonnegative.

 The differential $d_{-d,q}^k$ acts as $(E_{-d,q}^k)_{n}\to (E_{-d-k,q+k-1}^k)_{n}$,
 hence $PE_{k+1}$ is obtained from $PE_k$ by deleting terms of type
 $Q^n(T^d\h^{-d+q}+T^{d+k}\h^{-d+q-1})=Q^nT^d\h^{-d+q-1}(\h+T^k)$ ($k\geq 1$).

 We say that two series $P$ and $P'$ satisfies
 $P\geq P'$ if and only if $P-P'$ has all of its coefficients nonnegative. Then the above discussion shows that
 (cf. \cite[page 15]{McCleary})
 $$PE_1(T,Q,\h)\geq PE_2(T,Q,\h)\geq \cdots \geq PE_\infty(T,Q,\h).$$
 If $E^2_{*,*}=E^{\infty}_{*,*}$ then $PE_1-PE_2=(\h+T)R^+$, where all the coefficients of $R^+$ are nonnegative.
 In general,  $(PE_1-PE_{\infty})|_{T=1}=(\h+1)\bar{R}^+$, where all the coefficients of $\bar{R}^+$ are nonnegative.
 Thus,
 \begin{equation}\label{eq:spseq}
  PE_1(T,Q,\h)|_{T=1,\h=-1}= PE_2(T,Q,\h)|_{T=1,\h=-1}= \cdots = PE_\infty(T,Q,\h)|_{T=1,\h=-1}.\end{equation}
Since the whole  spectral sequence is an invariant of $(\Gamma_\phi,s)$ and $h\in H$,
$$k_{(\Gamma_\phi,s)}(n):=\min\{k\,:\, (E_{*,*}^k)_n=(E_{*,*}^\infty)_n\,\} \ \mbox{and } \ \ \
k_{(\Gamma_\phi,s)}=\max_n \,k_{(\Gamma_\phi,s)}(n)$$ are  invariants of $(\Gamma_\phi,s)$ and $h\in H$ too.

\subsection{The $U$--action along the spectral sequence}\label{ss:U}

For any fixed $n$, the natural inclusion $S_n\hookrightarrow S_{n+1}$
is compatible with the filtration, hence the inclusion
$\{S_n\cap \frX_{-d}\}_d\hookrightarrow \{S_{n+1}\cap \frX_{-d}\}_d$
at the level of filtered spaces  induces a morphism of spectral sequences
$(E^k_{-d.q})_n\to (E^k_{-d,q})_{n+1}$ compatibly  with the differentials
$(d^k_{-d.q})_n$ and $(d^k_{-d,q})_{n+1}$.

These morphisms will also be denoted by $U$.

\begin{lemma}\label{lem:torsion}
For any fixed $n$ there exists $\delta(n)\in \Z_{>0}$ such that for any $(d,q)$ the morphisms
$$U^{\delta(n)}\,:\, (E^1_{-d.q})_n\to (E^1_{-d,q})_{n+\delta(n)} \ \ \mbox{is trivial}. $$
In particular, $U^{\delta(n)}$ at the level of
$(E^k_{-d.q})_n\to (E^k_{-d,q})_{n+\delta(n)}$ ($k\geq 1$) is trivial too.
In this way, for any $k\geq 1$, the triple graded $\Z$--module $(E^k_{*.*})_*$  has a torsion
$\Z[U]$--module structure.
\end{lemma}

\begin{proof}
Fix a base element $E_i$ of $L$, and chose $\delta(n)$ so that for any lattice point $l\in S_n$
we also have $\chi_h(l+E_i)\leq n+\delta(n)$. This is possible due to the compactness of $S_n$.
Then we have the following commutative diagram

\begin{picture}(320,50)(0,5)

\put(100,40){\makebox(0,0){\small{$ (S_n\cap \frX_{-d}, S_n\cap \frX_{-d-1}) $}}}
\put(320,40){\makebox(0,0){\small{$(S_{n+\delta(n)}\cap \frX_{-d-1}, S_{n+\delta(n)}\cap \frX_{-d-2})$}}}
\put(155,40){\vector(1,0){80}}
\put(200,45){\makebox(0,0){\small{$l\mapsto l+E_i$}}}
\put(200,10){\makebox(0,0){\small{$(S_{n+\delta(n)}\cap \frX_{-d}, S_{n+\delta(n)}\cap \frX_{-d-1})$}}}
\put(115,30){\vector(2,-1){25}}
\put(285,30){\vector(-2,-1){25}}
\put(115,20){\makebox(0,0){\small{$u_{\delta(n)}$}}}
\put(285,20){\makebox(0,0){\small{$j$}}}

\end{picture}

\noindent where $j$ is the inclusion and $u_{\delta(n)}$ is the composition of the inclusions
$(S_{n+i}\cap \frX_{-d}, S_{n+i}\cap \frX_{-d-1})\hookrightarrow
(S_{n+i+1}\cap \frX_{-d}, S_{n+i+1}\cap \frX_{-d-1})$, $0\leq i <\delta(n)$.
Then
 note that $j$ at homological level induces the trivial morphism.
\end{proof}

\subsection{$\bH_*$ and the spectral sequence as decoration of the vertices of the graded root}\label{ss:decgrroot}\

The construction of the lattice homology, its filtration,  and the corresponding spectra sequences
 have an additional subtlety. Recall that the starting point in all these constructions
 is the tower of spaces $\{S_n\}_{n\geq m_w}$. Now, each $S_n$ is  the disjoint union of its connected components $\sqcup_v S_n^v$. They are indexed by the vertices $\cV(\mathfrak{R}_h)$ of the graded root
$\mathfrak{R}_h$ with $w(v)=n$, cf. \ref{bek:grroot}. On the other hand, for each $n$, the filtered spaces
$\{S_n\cap \frX_{-d}\}_d$ also decomposes into a disjoint union
$\sqcup_{v\in\cV(\mathfrak{R}), \ w(v)=n}\{S_n^v\cap \frX_{-d}\}_d$. In particular, all the spectral sequence (and all its outputs) split according to this decomposition indexed by $\cV(\mathfrak{R}_h)$. That is, all the invariants which originally were graded by $\{n\in\Z\,:\, n\geq m_w\}$ will have a refined grading given by the
vertices of the graded root.

In particular, the graded root $\mathfrak{R}_h$ supports 
the following invariants partitioned as decorations of the  vertices $\cV(\mathfrak{R}_h)$:
$\bH_*$, $(E^k_{-d.q})_n$, $PE_k(T,Q,\h)$, ${\bf PE}_1({\bf T},Q,\h)$ (this latter one will be defined in
\ref{ss:PET}).

Note that $U\,:\, (E^1_{-d.q})_n\to (E^1_{-d,q})_{n+1}$  defined in subsection \ref{ss:U}
decomposes also in direct sum according to the edges of $\mathfrak{R}_h$ connecting vertices of weight
$n$ and $n+1$ respectively, corresponding to the inclusion $S_n^v\hookrightarrow S_{n+1}^u$  of the
connected components. Hence the new grading given by the vertices of the root is compatible
with the two $U$ actions (of the module, e.g. of $\bH_*$, and of the root).

For another case in the literature when an important series
 appears as decoration of a graded root  see \cite{AJK}.

\subsection{Examples}\label{ss:examples}
\begin{example}\label{ex:p} {\bf ($\mathbf{|\calv|=1}$)} \
Assume that $\Gamma$ consists of one vertex with Euler decoration
$-p$, $p\geq 1$. Then $H=L'/L=\Z/p\Z$, which will be identified with the classes $[a]\in\Z/p\Z$, $0\leq a<p$. Then for any $[a]$ its unique minimal lift $s_{[a]}\in\calS'$ is $aE^*=(a/p)E$, hence $k_h=
(2+2a-p)/p\cdot E\in L'$. Moreover, $w_h(l):=\chi_{h}(lE)=\frac{1}{2}l(l-1)p+l(1+a)$
for any $l\in \Z_{\geq 0}$.

Since $\Gamma$ is rational $\bH_0(\frX, w_h)=\Z[U]=\calt^-_0$ as a $\Z[U]$--module and $\bH_{>0}(\frX, w_h)=0$  \cite{NOSz,Nkonyv}.

We set $s:= E^*=E/p$. Then  $F_{-d}\bH_0(\frX,w_h)=\calt^-_{-2w_h(d)}$; it is embedded naturally
(via the  degree zero homogeneous morphism) as $\Z\langle U^{w_h(d)}, U^{w_h(d)+1}, \ldots \rangle \hookrightarrow \Z[U]$.

Thus,  $H_0(S_n\cap \frX_{-l}, S_n\cap \frX_{-l-1})=\Z$ if and only if $w_h(l)\leq n$ but $w_h(l+1)> n$.
Therefore,
\begin{equation}\label{eq:pep}\begin{split}
PE_1(T,Q,\h)=PE_1(T,Q)&=\sum_{l\geq 0}\, T^l(Q^{w_{h}(lE)}+Q^{w_{h}(lE)+1}+\cdots +
Q^{w_{h}((l+1)E)-1})\\
&=\frac{1}{1-Q}\Big( 1+(T-1)\cdot \sum_{l\geq 1}\, Q^{\frac{1}{2} l(l-1)p+l(1+a)}\, T^{l-1}\,\Big).
\end{split}\end{equation}
The last sum is equivalent with a Jacobi theta function.

Since the weight of the 1--cube $[l,l+1]$ is $w_h((l+1)E)$, we also get the (slightly surprising)
 identity connecting $PE_1$ with the `motivic' Poincar\'e series  $$ PE_1(T,Q)=T^{-a/p}\cdot Z^m_{[a]}(T,Q).$$
This identity will be proved for the most general situation  in Corollary \ref{cor:EP}.

Regarding the spectral sequences,
since $(E^1_{-d,q})_n=0$ for $-d+q\not=0$, the spectral sequence for any $n$ degenerates at $E^1$ level, $(E^1_{-d,q})_n=(E^\infty_{-d,q})_n$ and
$PE_1=PE_\infty$.

By the above discussion, as a $\Z[U]$--module, $(E^k_{*,*})_*=\oplus_{l\geq 0} \, \calt_{-2w_h(l)}\,(w_h(l+1)-w_h(l))$.

Compared with $\bH_0=\Z[U]$, $(E^k_{*,*})_*$ contains essentially more information.
In fact, it contains all the information about the set of values
$\{\chi_h(lE)\}_{l\geq 0}$.
Since $\chi_h(lE)=\chi(l+\frac{a}{p})-\chi(\frac{a}{p})$, this is equivalent (up to a shift)
with the set of $\chi$--values
$\{\chi(l+\frac{a}{p}))\}_{l\geq 0}=\chi(\calS'_h)$, where $\{l+\frac{a}{p}\}_{\l\geq 0}=\calS'_h=\calS'\cap \{l'\in L'\,:\, [l']=h\}$ is the set of semigroup elements in the corresponding $h$--class.

\end{example}

\begin{example}\label{ex:p2} {\bf (Behaviour under blow up)} \
Consider the case of Example \ref{ex:p} with $p=2$ and $h=0$. Then $w(lE)=l^2$ and the filtration given by $s=E^*$ is induced by $\frX_{-d}=\{lE\,:\, l\geq d\}$.

Let us blow up a generic point of $E$ (that is, if we fix a transversal cut  $\widetilde{C}$
of $E$ corresponding to $s$, then the center of the blow up is not $\widetilde{C}\cap E$).
Let the strict transform of $E$ be $E_1$ (the $(-3)$ curve), and let the newly created exceptional curve
be $E_2$. Set $s_1:=\psi^*(s)=E_1^*$. Then one can concretely verify the statement of the
stablity Theorem \ref{th:eq}, namely the homotopy equivalence of the spaces $S_n\cap \frX_{-d}$ considered at the original level and after blow up.

This can be seen via reduction theorem as well applied to the second case (the situation
after blow up). We take $\overline{\cV}=\{E_1\}$ and we write $\overline{l}=l_1E_1$, then $x(\overline{l})=l_1(E_1+E_2)$ and $\overline{w}(\overline{l})=\chi(x(\overline{l}))=l_1^2$
with filtration $\frX_{-d}=\{l_1\geq d\}$. Hence we recover perfectly the situation before the blow up. This is compatible with the Filtration Reduction Theorem \ref{th:redth2} regarding the class
$(M, L_C,0)$

Let us consider again the original situation of the $(-2)$ resolution graph together with its arrow,
and let us blow up the `base point' $\widetilde{C}\cap E$. In this way we create the class
$(M, L_C, 1)$.  Let $E_1$ and $E_2$ be again the
$(-3)$ and $(-1)$ exceptional curves respectively. Note that in this case the strict transform of
$\widetilde{C}$ will intersect $E_2$, hence $s_2=E_2^*$. By a verification one sees that the
homotopy type of the
 spaces  $\{S_n\cap \frX_{-d}\}$ is modified drastically. E.g., via the Filtered Reduction Theorem applied for this blown up situation with $\overline{\cV}=\{E_2\}$ we have to analyse the rank one lattice elements $\overline{l}=l_2E_2$. One shows that $x(\overline{l})=l_2E_2+l_1E_1$, where
$l_1=\lceil l_2/3\rceil$, $\overline{w}(\overline{l})=(3l_1^2+l_2^2-2l_1l_2-l_1+l_2)/2$,
$\frX_{-d}=\{l_2\geq d\}$. This is very different than  the original weight sequence $l\mapsto l^2$:
the first entries of $l_2\mapsto \overline{w}(\overline{l})$ are 0, 1,2,4,7,10, 14,... \
compared with
0,1,4,9,.....
That is, the $d$--degrees of $H_0(S_n,\Z)=\Z$ in the unique nontrivial
${\rm Gr}_{-d}^F(\bH_*)_{-2n}$ are different.
(If $n$ satisfies $\overline{w}(\overline{l})\leq n< \overline{w}(\overline{l}+1)$ then
$d=\overline{l}$.)

This also shows that we cannot expect stability of the spectral sequence (and its outputs) along a base point blow up (i.e. by modifying the decorations of the curve components).
\end{example}

\bekezdes {\bf (Examples regarding the degeneration index $k_{(\Gamma_\phi,s)}$.)} \
In general, for an arbitrary  $(\Gamma_\phi,s)$,
 the spectral sequence does not degenerate at $E^1$ level.
Usually  it is difficult to determine the invariant $k_{(\Gamma_\phi, s)}(n)$. Moreover,
$k_{(\Gamma_\phi,s)}(n)$
depends essentially on the choice of the semigroup element $s$ too.

The next examples target  these facts.

\begin{example}\label{ex:latNV1}\ (See \cite{Nlattice,Nkonyv} for the non-filtered discussion.)
Consider the following graph:

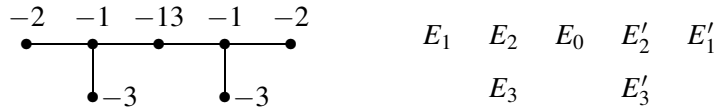
\vspace{1mm}

\begin{picture}(300,45)(60,0)
\put(125,25){\circle*{4}} \put(150,25){\circle*{4}}
\put(175,25){\circle*{4}} \put(200,25){\circle*{4}}
\put(225,25){\circle*{4}} \put(150,5){\circle*{4}}
\put(200,5){\circle*{4}} \put(125,25){\line(1,0){100}}
\put(150,25){\line(0,-1){20}} \put(200,25){\line(0,-1){20}}
\put(125,35){\makebox(0,0){$-2$}}
\put(150,35){\makebox(0,0){$-1$}}
\put(175,35){\makebox(0,0){$-13$}}
\put(200,35){\makebox(0,0){$-1$}}
\put(225,35){\makebox(0,0){$-2$}} \put(160,5){\makebox(0,0){$-3$}}
\put(210,5){\makebox(0,0){$-3$}}

\put(275,25){$E_1$} \put(300,25){$E_2$} \put(325,25){$E_0$}
\put(350,25){$E_2'$} \put(375,25){$E_1'$} \put(300,5){$E_3$}
\put(350,5){$E_3'$}
\end{picture}

\vspace{1mm}

On the right hand side we give names to the vertices and base elements of $L$. Set
$\chi$ associated with $h=0$.  We write any $l\in L$ in the
form $l=l_x+ zE_0+ l_y$, where $l_x=\sum x_iE_i$, $l_y=\sum
y_iE_i'$, $x_i,\ y_i,\ z\in \Z$ ($i=1,2,3$); or
$(x_1,x_2,x_3;z;y_1,y_2,y_3)$. Then $Z_K=(7,14,5;3;7,14,5)$.
Moreover, $\min\chi=-1$.

 $S_{-1}$ consists of two connected components, both contractible,
$S_0$ has three components, two contractible ones and one with the homotopy type of the circle  $S^1$,
 and $S_n$ is contractible for $n\geq 1$. Analysing the inclusions (see \cite{Nlattice} or below)
 $\bH_0=\calt_{2}^- \oplus \calt_2(1)\oplus \calt_0(1)^2$ and
 $\bH_1=\calt_0(1)$.

In our discussions it is helpful to use the symmetry $\chi(l)=\chi(Z_K-l)$.
Let us introduce the  triplets
$$B:=\{(6,12,4), (6,11,4), (5,11,4),(5,10,4),(5,10,3),(5,9,3),(4,9,3),(4,8,3)\} $$
and triplets of type $A=(7,14,5)-B$.
Then  $S_{-1}$ consists of two contractible components, one of them has lattice points
of type
$(A,1,A)=\{(x,1,y): x,y\in A\}$ while the other one consists of points of type $(B,2,B)$.
They are symmetric with respect to $Z_K/2$.
Next, consider the  triplets
$B':=
\{(7,14,5),(7,13,5),$ $ (6,13,5),(6,12,5),
(7,13,4),(6,13,4),
(7,12,4),(5,12,4), (6,10,4),(4,10,4),$ $(5,9,4),(4,9,4)\}$
 and
the
triplets of type $B''=(10,20,7)-B'$. Set $\tilde{B}:=B\cup B'\cup
B''$, and $\tilde{A}=(7,14,5)-\tilde{B}$. Then the points of type
$X:=(A,1,\tilde{A})\cup (\tilde{A},1,A)\cup (B,2,\tilde{B})\cup
(\tilde{B},2,B)$ are in $S_0$. Since  $\tilde{A}\cap
B$ is not empty, all the points from $X$ can be connected by
1--cubes of $X$. In fact, $S_0$ has three connected components, one of
them contains the zero cycle, the other one contains $Z_K$, and the
third one, $S_0^v$, consists of  all the points of $X$.
Finally, notice that the two intersection points $P=\tilde{A}\cap
B=(4,8,3)$ and $Q=A\cap \tilde{B}=(3,6,2)$ create a loop in
$S^v_0$. Indeed, half of the loop consists of a connecting path of $(P;1;Q)$ and
$(Q;1;P)$ through points in $X$ with $z=1$, the other half
connects  $(P;2;Q)$ with $(Q;2;P)$ through points in $X$ with
$z=2$. This loop can be contracted only in $S_1$.

{\bf Case 1.}
We choose $s=\sum_{i\in\cV}E^*_i$, i.e., $d(l)=-(s,l)=\sum_il_i$ and we focus on the component
$S_0^v$.

By the above
discussion $S^v_0\cap \frX_{-51}=\emptyset$, and
$S^v_0\cap \frX_{-50}$ consists of two points $p_1=(7,14,5;2;6,12,4)$ and
$p_2=(6,12,4;2;7,14,5)$. They can be connected by a path in $S^v_0$, however,
the smallest subspace where such a path can be constructed is $\frX_{-46}$.
(This is the path connecting first  $p_1$ with $p_0=(6,12,4;2;6,12,4)$ by a decreasing path in the first three coordinates
--- via the first members of $B'$ --- then connecting $p_0$ with $p_2$ via a symmetric path.)
In particular, $(E^4_{-50, 50})_0=(E^1_{-50, 50})_0=H_0(S^v_0\cap\frX_{-50}, S^v_0\cap\frX_{-51})$ is $\Z^2$ generated by $p_1$ and $p_2$, and the class of $p_1-p_2$ is in the image of
$d^4_{-46,47}:(E^4_{-46,47})_0\to (E^4_{-50,50})_0$. (Using the general notations of spectral sequences, this means that
$p_1-p_2\in B^5_{-50,50}$.)

In particular $(E^k_{-50, 50})_0$ stabilizes for $k\geq 5$ (but not earlier) and  $k_{(\Gamma, s)}(0)\geq 5$.

Here the following warning is appropriate. Assume that we fix a set of bad vertices $\overline{\cV}$ and we wish to
use the reduction theorem associated with $\overline{\cV}$. The point is that if
 we wish to determine  the homotopy type of a space $S_n$ then we can run the reduction theorem
without any obstruction. However,  if we wish to compute the homotopy types of
the filtered subspaces $\{S_n\cap \frX_{-d}\}_d$ (filtered by a semigroup element $s$),
then we can run the reduction theorem only if the $E^*$ support ${\rm Supp}(s)$ of $s$ is included in $\overline {\cV}$;
see the Filtered Reduction Theorem \ref{th:redth2}.

In particular, the results of this Case 1 cannot be recomputed via  a reduction theorem
with $\overline{\cV}\subsetneq \cV$.

This example also shows rather suggestively that if the Filtered Reduction Theorem cannot
be applied, or if $|\calv|$ is still large, then the structure of the spaces $S_n$ and their filtrations can be rather technical, arithmetical, hardly manageable. This behaviour emphasizes the role and the power of the reduction theorems as well.

\vspace{1mm}

{\bf Case 2.} Let us consider now $s=E_2^*+(E_2')^*$. Then $d=-(s,l)$ is the sum of the second and the sixth coordinates,
hence (using the notations $p_0, \, p_1, \, p_2$ from Case 1)
$d(p_1)=d(p_2)=26$, $d(p_0)=24$. Hence $S^v_0\cap \frX_{27}=\emptyset$, $(E^1_{-26, 26})_0=H_0(S^v_0\cap\frX_{-26}, S^v_0\cap\frX_{-27})$ is $\Z^2=\Z\langle p_1,p_2\rangle$, and $p_1-p_2$ is killed by the image of $d^2_{-24, 25}$.
Hence  $(E^k_{-26, 26})_0$ stabilizes for $k\geq 3$.

Note that the nodes $E_2 $ and $E_2'$ constitute an SR--set, hence we can apply the Reduction Theorem for them. Furthermore,
the $E^*$--support of $s=E_2^*+(E_2')^*$ is contained in this set. Hence the spectral sequence can be read
from the reduces situation as well, cf. Proposition \ref{prop:stab}.

In this case $S^v_0$ is included in the rectangle $R((0,0), (14,14))$ (the projected $R(0, Z_K)$). The weight function $\overline{w}$
in this rectangle is shown in the next diagram, from which the very same statement can be read.
The projections of the
generators $p_1$ and $p_2$ and  the connecting path are highlighted by
boldface characters.

\begin{picture}(320,160)(-50,-10)

\put(5,-5){\makebox(0,0){\small{$0$}}}
\put(20,-5){\makebox(0,0){\small{$1$}}}
\put(35,-5){\makebox(0,0){\small{$0$}}}
\put(50,-5){\makebox(0,0){\small{$0$}}}
\put(65,-5){\makebox(0,0){\small{$0$}}}
\put(80,-5){\makebox(0,0){\small{$0$}}}
\put(95,-5){\makebox(0,0){\small{$0$}}}
\put(110,-5){\makebox(0,0){\small{$1$}}}
\put(125,-5){\makebox(0,0){\small{$1$}}}
\put(140,-5){\makebox(0,0){\small{$2$}}}
\put(155,-5){\makebox(0,0){\small{$3$}}}
\put(170,-5){\makebox(0,0){\small{$4$}}}
\put(185,-5){\makebox(0,0){\small{$5$}}}
\put(200,-5){\makebox(0,0){\small{$7$}}}
\put(215,-5){\makebox(0,0){\small{$8$}}}

\put(5,5){\makebox(0,0){\small{$1$}}}
\put(20,5){\makebox(0,0){\small{$1$}}}
\put(35,5){\makebox(0,0){\small{$0$}}}
\put(50,5){\makebox(0,0){\small{$0$}}}
\put(65,5){\makebox(0,0){\small{$0$}}}
\put(80,5){\makebox(0,0){\small{$0$}}}
\put(95,5){\makebox(0,0){\small{$0$}}}
\put(110,5){\makebox(0,0){\small{$1$}}}
\put(125,5){\makebox(0,0){\small{$1$}}}
\put(140,5){\makebox(0,0){\small{$2$}}}
\put(155,5){\makebox(0,0){\small{$3$}}}
\put(170,5){\makebox(0,0){\small{$4$}}}
\put(185,5){\makebox(0,0){\small{$5$}}}
\put(200,5){\makebox(0,0){\small{$7$}}}
\put(215,5){\makebox(0,0){\small{$7$}}}

\put(5,15){\makebox(0,0){\small{$0$}}}
\put(20,15){\makebox(0,0){\small{$0$}}}
\put(35,15){\makebox(0,0){\small{$-1$}}}
\put(50,15){\makebox(0,0){\small{$-1$}}}
\put(65,15){\makebox(0,0){\small{$-1$}}}
\put(80,15){\makebox(0,0){\small{$-1$}}}
\put(95,15){\makebox(0,0){\small{$-1$}}}
\put(110,15){\makebox(0,0){\small{$0$}}}
\put(125,15){\makebox(0,0){\small{$0$}}}
\put(140,15){\makebox(0,0){\small{$1$}}}
\put(155,15){\makebox(0,0){\small{$2$}}}
\put(170,15){\makebox(0,0){\small{$3$}}}
\put(185,15){\makebox(0,0){\small{$4$}}}
\put(200,15){\makebox(0,0){\small{$5$}}}
\put(215,15){\makebox(0,0){\small{$5$}}}

\put(5,25){\makebox(0,0){\small{$0$}}}
\put(20,25){\makebox(0,0){\small{$0$}}}
\put(35,25){\makebox(0,0){\small{$-1$}}}
\put(50,25){\makebox(0,0){\small{$-1$}}}
\put(65,25){\makebox(0,0){\small{$-1$}}}
\put(80,25){\makebox(0,0){\small{$-1$}}}
\put(95,25){\makebox(0,0){\small{$-1$}}}
\put(110,25){\makebox(0,0){\small{$0$}}}
\put(125,25){\makebox(0,0){\small{$0$}}}
\put(140,25){\makebox(0,0){\small{$1$}}}
\put(155,25){\makebox(0,0){\small{$2$}}}
\put(170,25){\makebox(0,0){\small{$3$}}}
\put(185,25){\makebox(0,0){\small{$3$}}}
\put(200,25){\makebox(0,0){\small{$4$}}}
\put(215,25){\makebox(0,0){\small{$4$}}}

\put(5,35){\makebox(0,0){\small{$0$}}}
\put(20,35){\makebox(0,0){\small{$0$}}}
\put(35,35){\makebox(0,0){\small{$-1$}}}
\put(50,35){\makebox(0,0){\small{$-1$}}}
\put(65,35){\makebox(0,0){\small{$-1$}}}
\put(80,35){\makebox(0,0){\small{$-1$}}}
\put(95,35){\makebox(0,0){\small{$-1$}}}
\put(110,35){\makebox(0,0){\small{$0$}}}
\put(125,35){\makebox(0,0){\small{$0$}}}
\put(140,35){\makebox(0,0){\small{$1$}}}
\put(155,35){\makebox(0,0){\small{$2$}}}
\put(170,35){\makebox(0,0){\small{$2$}}}
\put(185,35){\makebox(0,0){\small{$2$}}}
\put(200,35){\makebox(0,0){\small{$3$}}}
\put(215,35){\makebox(0,0){\small{$3$}}}

\put(5,45){\makebox(0,0){\small{$0$}}}
\put(20,45){\makebox(0,0){\small{$0$}}}
\put(35,45){\makebox(0,0){\small{$-1$}}}
\put(50,45){\makebox(0,0){\small{$-1$}}}
\put(65,45){\makebox(0,0){\small{$-1$}}}
\put(80,45){\makebox(0,0){\small{$-1$}}}
\put(95,45){\makebox(0,0){\small{$-1$}}}
\put(110,45){\makebox(0,0){\small{$0$}}}
\put(125,45){\makebox(0,0){\small{$0$}}}
\put(140,45){\makebox(0,0){\small{$1$}}}
\put(155,45){\makebox(0,0){\small{$1$}}}
\put(170,45){\makebox(0,0){\small{$1$}}}
\put(185,45){\makebox(0,0){\small{$1$}}}
\put(200,45){\makebox(0,0){\small{$2$}}}
\put(215,45){\makebox(0,0){\small{$2$}}}

\put(5,55){\makebox(0,0){\small{$0$}}}
\put(20,55){\makebox(0,0){\small{$0$}}}
\put(35,55){\makebox(0,0){\small{$-1$}}}
\put(50,55){\makebox(0,0){\small{$-1$}}}
\put(65,55){\makebox(0,0){\small{$-1$}}}
\put(80,55){\makebox(0,0){\small{$-1$}}}
\put(95,55){\makebox(0,0){\small{$-1$}}}
\put(110,55){\makebox(0,0){\small{$0$}}}
\put(125,55){\makebox(0,0){\small{$0$}}}
\put(140,55){\makebox(0,0){\small{$0$}}}
\put(155,55){\makebox(0,0){\small{$0$}}}
\put(170,55){\makebox(0,0){\small{$0$}}}
\put(185,55){\makebox(0,0){\small{$0$}}}
\put(200,55){\makebox(0,0){\small{$1$}}}
\put(215,55){\makebox(0,0){\small{$1$}}}

\put(5,65){\makebox(0,0){\small{$1$}}}
\put(20,65){\makebox(0,0){\small{$1$}}}
\put(35,65){\makebox(0,0){\small{$0$}}}
\put(50,65){\makebox(0,0){\small{$0$}}}
\put(65,65){\makebox(0,0){\small{$0$}}}
\put(80,65){\makebox(0,0){\small{$0$}}}
\put(95,65){\makebox(0,0){\small{$0$}}}
\put(110,65){\makebox(0,0){\small{$1$}}}
\put(125,65){\makebox(0,0){\small{$0$}}}
\put(140,65){\makebox(0,0){\small{$0$}}}
\put(155,65){\makebox(0,0){\small{$0$}}}
\put(170,65){\makebox(0,0){\small{$0$}}}
\put(185,65){\makebox(0,0){\small{$0$}}}
\put(200,65){\makebox(0,0){\small{$1$}}}
\put(215,65){\makebox(0,0){\small{$1$}}}

\put(5,75){\makebox(0,0){\small{$1$}}}
\put(20,75){\makebox(0,0){\small{$1$}}}
\put(35,75){\makebox(0,0){\small{$0$}}}
\put(50,75){\makebox(0,0){\small{$0$}}}
\put(65,75){\makebox(0,0){\small{$0$}}}
\put(80,75){\makebox(0,0){\small{$0$}}}
\put(95,75){\makebox(0,0){\small{$0$}}}
\put(110,75){\makebox(0,0){\small{$0$}}}
\put(125,75){\makebox(0,0){\small{$-1$}}}
\put(140,75){\makebox(0,0){\small{$-1$}}}
\put(155,75){\makebox(0,0){\small{$-1$}}}
\put(170,75){\makebox(0,0){\small{$-1$}}}
\put(185,75){\makebox(0,0){\small{$-1$}}}
\put(200,75){\makebox(0,0){\small{$0$}}}
\put(215,75){\makebox(0,0){\small{$0$}}}

\put(5,85){\makebox(0,0){\small{$2$}}}
\put(20,85){\makebox(0,0){\small{$2$}}}
\put(35,85){\makebox(0,0){\small{$1$}}}
\put(50,85){\makebox(0,0){\small{$1$}}}
\put(65,85){\makebox(0,0){\small{$1$}}}
\put(80,85){\makebox(0,0){\small{$1$}}}
\put(95,85){\makebox(0,0){\small{$0$}}}
\put(110,85){\makebox(0,0){\small{$0$}}}
\put(125,85){\makebox(0,0){\small{$-1$}}}
\put(140,85){\makebox(0,0){\small{$-1$}}}
\put(155,85){\makebox(0,0){\small{$-1$}}}
\put(170,85){\makebox(0,0){\small{$-1$}}}
\put(185,85){\makebox(0,0){\small{$-1$}}}
\put(200,85){\makebox(0,0){\small{$0$}}}
\put(215,85){\makebox(0,0){\small{$0$}}}

\put(5,95){\makebox(0,0){\small{$3$}}}
\put(20,95){\makebox(0,0){\small{$3$}}}
\put(35,95){\makebox(0,0){\small{$2$}}}
\put(50,95){\makebox(0,0){\small{$2$}}}
\put(65,95){\makebox(0,0){\small{$2$}}}
\put(80,95){\makebox(0,0){\small{$1$}}}
\put(95,95){\makebox(0,0){\small{$0$}}}
\put(110,95){\makebox(0,0){\small{$0$}}}
\put(125,95){\makebox(0,0){\small{$-1$}}}
\put(140,95){\makebox(0,0){\small{$-1$}}}
\put(155,95){\makebox(0,0){\small{$-1$}}}
\put(170,95){\makebox(0,0){\small{$-1$}}}
\put(185,95){\makebox(0,0){\small{$-1$}}}
\put(200,95){\makebox(0,0){\small{$0$}}}
\put(215,95){\makebox(0,0){\small{$0$}}}

\put(5,105){\makebox(0,0){\small{$4$}}}
\put(20,105){\makebox(0,0){\small{$4$}}}
\put(35,105){\makebox(0,0){\small{$3$}}}
\put(50,105){\makebox(0,0){\small{$3$}}}
\put(65,105){\makebox(0,0){\small{$2$}}}
\put(80,105){\makebox(0,0){\small{$1$}}}
\put(95,105){\makebox(0,0){\small{$0$}}}
\put(110,105){\makebox(0,0){\small{$0$}}}
\put(125,105){\makebox(0,0){\small{$-1$}}}
\put(140,105){\makebox(0,0){\small{$-1$}}}
\put(155,105){\makebox(0,0){\small{$-1$}}}
\put(170,105){\makebox(0,0){\small{$-1$}}}
\put(185,105){\makebox(0,0){\small{$-1$}}}
\put(200,105){\makebox(0,0){\small{$0$}}}
\put(215,105){\makebox(0,0){\small{$0$}}}

\put(5,115){\makebox(0,0){\small{$5$}}}
\put(20,115){\makebox(0,0){\small{$5$}}}
\put(35,115){\makebox(0,0){\small{$4$}}}
\put(50,115){\makebox(0,0){\small{$3$}}}
\put(65,115){\makebox(0,0){\small{$2$}}}
\put(80,115){\makebox(0,0){\small{$1$}}}
\put(95,115){\makebox(0,0){\small{$0$}}}
\put(110,115){\makebox(0,0){\small{$0$}}}
\put(125,115){\makebox(0,0){\small{$-1$}}}
\put(140,115){\makebox(0,0){\small{$-1$}}}
\put(155,115){\makebox(0,0){\small{$-1$}}}
\put(170,115){\makebox(0,0){\small{$-1$}}}
\put(185,115){\makebox(0,0){\small{${\bf -1}$}}}
\put(200,115){\makebox(0,0){\small{${\bf 0}$}}}
\put(215,115){\makebox(0,0){\small{${\bf 0}$}}}

\put(5,125){\makebox(0,0){\small{$7$}}}
\put(20,125){\makebox(0,0){\small{$7$}}}
\put(35,125){\makebox(0,0){\small{$5$}}}
\put(50,125){\makebox(0,0){\small{$4$}}}
\put(65,125){\makebox(0,0){\small{$3$}}}
\put(80,125){\makebox(0,0){\small{$2$}}}
\put(95,125){\makebox(0,0){\small{$1$}}}
\put(110,125){\makebox(0,0){\small{$1$}}}
\put(125,125){\makebox(0,0){\small{$0$}}}
\put(140,125){\makebox(0,0){\small{$0$}}}
\put(155,125){\makebox(0,0){\small{$0$}}}
\put(170,125){\makebox(0,0){\small{$0$}}}
\put(185,125){\makebox(0,0){\small{${\bf 0}$}}}
\put(200,125){\makebox(0,0){\small{$1$}}}
\put(215,125){\makebox(0,0){\small{$1$}}}

\put(5,135){\makebox(0,0){\small{$8$}}}
\put(20,135){\makebox(0,0){\small{$7$}}}
\put(35,135){\makebox(0,0){\small{$5$}}}
\put(50,135){\makebox(0,0){\small{$4$}}}
\put(65,135){\makebox(0,0){\small{$3$}}}
\put(80,135){\makebox(0,0){\small{$2$}}}
\put(95,135){\makebox(0,0){\small{$1$}}}
\put(110,135){\makebox(0,0){\small{$1$}}}
\put(125,135){\makebox(0,0){\small{$0$}}}
\put(140,135){\makebox(0,0){\small{$0$}}}
\put(155,135){\makebox(0,0){\small{$0$}}}
\put(170,135){\makebox(0,0){\small{$0$}}}
\put(185,135){\makebox(0,0){\small{${\bf 0}$}}}
\put(200,135){\makebox(0,0){\small{$1$}}}
\put(215,135){\makebox(0,0){\small{$0$}}}

\end{picture}

But even if we rely on
 the Filtered Reduction Theorem, using the rank two first quadrant $(\R_{\geq 0})^2$,
the complete description of the $\Z$--modules $(E^k_{*,*})_*$, together with their $U$--action,
is rather hard. This is mainly due  by the complicated shape (full with lagoons and tentacles)
of the spaces $S_n$. This phenomenon is the subject of the next example (which has even more lagoons).

\end{example}

\begin{example}\label{ex:twonodes}

Consider the following graph $\Gamma$ (for the  discussion of the unfiltered case
 see Example 11.4.11 of \cite{Nkonyv}).

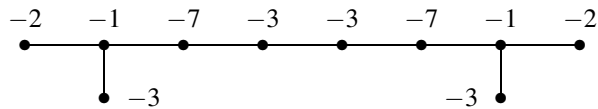
\vspace{1mm}

\begin{center}
\begin{picture}(140,40)(80,35)
\put(110,60){\circle*{4}}
\put(140,60){\circle*{4}}
\put(170,60){\circle*{4}}
\put(200,60){\circle*{4}}
\put(80,60){\circle*{4}}
\put(50,60){\circle*{4}}
\put(230,60){\circle*{4}}
\put(260,60){\circle*{4}}
\put(50,60){\line(1,0){210}}
\put(80,60){\line(0,-1){20}}
\put(80,40){\circle*{4}}
\put(230,60){\line(0,-1){20}}
\put(230,40){\circle*{4}}
\put(50,70){\makebox(0,0){\small{$-2$}}}
\put(80,70){\makebox(0,0){\small{$-1$}}}
\put(110,70){\makebox(0,0){\small{$-7$}}}
\put(140,70){\makebox(0,0){\small{$-3$}}}
\put(170,70){\makebox(0,0){\small{$-3$}}}
\put(200,70){\makebox(0,0){\small{$-7$}}}
\put(230,70){\makebox(0,0){\small{$-1$}}}
\put(260,70){\makebox(0,0){\small{$-2$}}}
\put(95,40){\makebox(0,0){\small{$-3$}}}
\put(215,40){\makebox(0,0){\small{$-3$}}}
\end{picture}
\end{center}

We take $h=0$.
The two nodes constitute an SR--set, we will apply the Filtered Reduction Theorem for $h=0$ and
$s$ the sum of the dual base elements associated with the two nodes.
The projection of $Z_K$ is $(14,14)$, however we will consider the
 $\overline{w}$--table in the larger rectangle $R((0,0), (18,18))$. Though  for any $n$
 the homotopy type of $S_n$ and $S_n\cap R((0,0),(14,14))$ agree (cf. Proposition \ref{lem:con}),
 the same  is not true for the graded pieces.
 This can be seen  below for the `fenced'  $S_0$  (and this fact motivates to take the larger rectangle).

\begin{picture}(320,210)(-50,-10)

\put(0,-10){\line(1,0){10}}
\put(0,-10){\line(0,1){10}}
\put(10,-10){\line(0,1){10}}
\put(0,0){\line(1,0){10}}

\put(0,10){\line(1,0){25}}
\put(25,10){\line(0,-1){25}}
\put(0,60){\line(1,0){25}}
\put(0,70){\line(1,0){25}}
\put(25,70){\line(0,-1){10}}
\put(0,120){\line(1,0){25}}

\put(25,120){\line(0,1){60}}
\put(100,120){\line(0,1){60}}
\put(25,180){\line(1,0){75}}
\put(100,120){\line(1,0){15}}
\put(115,180){\line(0,-1){60}}
\put(115,180){\line(1,0){75}}
\put(190,180){\line(0,-1){60}}
\put(190,120){\line(1,0){90}}

\put(205,130){\line(1,0){75}}
\put(205,180){\line(1,0){75}}
\put(205,130){\line(0,1){50}}\put(280,130){\line(0,1){50}}

\put(280,120){\line(0,-1){50}}\put(190,60){\line(0,1){10}}
\put(190,60){\line(1,0){90}}
\put(190,70){\line(1,0){90}}
\put(280,60){\line(0,-1){50}}
\put(190,10){\line(1,0){90}}\put(190,10){\line(0,-1){20}}

\put(115,10){\line(0,-1){20}}\put(100,10){\line(0,-1){20}}
\put(100,10){\line(1,0){15}}

\put(100,60){\line(1,0){15}}
\put(100,70){\line(1,0){15}}
\put(100,60){\line(0,1){10}}
\put(115,60){\line(0,1){10}}

\put(5,-5){\makebox(0,0){\small{${\bf 0}$}}}
\put(20,-5){\makebox(0,0){\small{$1$}}}
\put(35,-5){\makebox(0,0){\small{$0$}}}
\put(50,-5){\makebox(0,0){\small{$0$}}}
\put(65,-5){\makebox(0,0){\small{$0$}}}
\put(80,-5){\makebox(0,0){\small{$0$}}}
\put(95,-5){\makebox(0,0){\small{$0$}}}
\put(110,-5){\makebox(0,0){\small{$1$}}}
\put(125,-5){\makebox(0,0){\small{$0$}}}
\put(140,-5){\makebox(0,0){\small{$0$}}}
\put(155,-5){\makebox(0,0){\small{$0$}}}
\put(170,-5){\makebox(0,0){\small{$0$}}}
\put(185,-5){\makebox(0,0){\small{$0$}}}
\put(200,-5){\makebox(0,0){\small{$1$}}}
\put(215,-5){\makebox(0,0){\small{$1$}}}
\put(230,-5){\makebox(0,0){\small{$1$}}}
\put(245,-5){\makebox(0,0){\small{$1$}}}
\put(260,-5){\makebox(0,0){\small{$1$}}}
\put(275,-5){\makebox(0,0){\small{$1$}}}
\put(290,-5){\makebox(0,0){\small{$2$}}}

\put(5,5){\makebox(0,0){\small{$1$}}}
\put(20,5){\makebox(0,0){\small{$1$}}}
\put(35,5){\makebox(0,0){\small{$0$}}}
\put(50,5){\makebox(0,0){\small{$0$}}}
\put(65,5){\makebox(0,0){\small{$0$}}}
\put(80,5){\makebox(0,0){\small{$0$}}}
\put(95,5){\makebox(0,0){\small{$0$}}}
\put(110,5){\makebox(0,0){\small{$1$}}}
\put(125,5){\makebox(0,0){\small{$0$}}}
\put(140,5){\makebox(0,0){\small{$0$}}}
\put(155,5){\makebox(0,0){\small{$0$}}}
\put(170,5){\makebox(0,0){\small{$0$}}}
\put(185,5){\makebox(0,0){\small{$0$}}}
\put(200,5){\makebox(0,0){\small{$1$}}}
\put(215,5){\makebox(0,0){\small{$1$}}}
\put(230,5){\makebox(0,0){\small{$1$}}}
\put(245,5){\makebox(0,0){\small{$1$}}}
\put(260,5){\makebox(0,0){\small{$1$}}}
\put(275,5){\makebox(0,0){\small{$1$}}}
\put(290,5){\makebox(0,0){\small{$2$}}}

\put(5,15){\makebox(0,0){\small{$0$}}}
\put(20,15){\makebox(0,0){\small{$0$}}}
\put(35,15){\makebox(0,0){\small{$-1$}}}
\put(50,15){\makebox(0,0){\small{$-1$}}}
\put(65,15){\makebox(0,0){\small{$-1$}}}
\put(80,15){\makebox(0,0){\small{$-1$}}}
\put(95,15){\makebox(0,0){\small{$-1$}}}
\put(110,15){\makebox(0,0){\small{$0$}}}
\put(125,15){\makebox(0,0){\small{$-1$}}}
\put(140,15){\makebox(0,0){\small{$-1$}}}
\put(155,15){\makebox(0,0){\small{$-1$}}}
\put(170,15){\makebox(0,0){\small{$-1$}}}
\put(185,15){\makebox(0,0){\small{$-1$}}}
\put(200,15){\makebox(0,0){\small{$0$}}}
\put(215,15){\makebox(0,0){\small{$0$}}}
\put(230,15){\makebox(0,0){\small{$0$}}}
\put(245,15){\makebox(0,0){\small{$0$}}}
\put(260,15){\makebox(0,0){\small{$0$}}}
\put(275,15){\makebox(0,0){\small{$0$}}}
\put(290,15){\makebox(0,0){\small{$1$}}}

\put(5,25){\makebox(0,0){\small{$0$}}}
\put(20,25){\makebox(0,0){\small{$0$}}}
\put(35,25){\makebox(0,0){\small{$-1$}}}
\put(50,25){\makebox(0,0){\small{$-1$}}}
\put(65,25){\makebox(0,0){\small{$-1$}}}
\put(80,25){\makebox(0,0){\small{$-1$}}}
\put(95,25){\makebox(0,0){\small{$-1$}}}
\put(110,25){\makebox(0,0){\small{$0$}}}
\put(125,25){\makebox(0,0){\small{$-1$}}}
\put(140,25){\makebox(0,0){\small{$-1$}}}
\put(155,25){\makebox(0,0){\small{$-1$}}}
\put(170,25){\makebox(0,0){\small{$-1$}}}
\put(185,25){\makebox(0,0){\small{$-1$}}}
\put(200,25){\makebox(0,0){\small{$0$}}}
\put(215,25){\makebox(0,0){\small{$0$}}}
\put(230,25){\makebox(0,0){\small{$0$}}}
\put(245,25){\makebox(0,0){\small{$0$}}}
\put(260,25){\makebox(0,0){\small{$0$}}}
\put(275,25){\makebox(0,0){\small{$0$}}}
\put(290,25){\makebox(0,0){\small{$1$}}}

\put(5,35){\makebox(0,0){\small{$0$}}}
\put(20,35){\makebox(0,0){\small{$0$}}}
\put(35,35){\makebox(0,0){\small{$-1$}}}
\put(50,35){\makebox(0,0){\small{$-1$}}}
\put(65,35){\makebox(0,0){\small{$-1$}}}
\put(80,35){\makebox(0,0){\small{$-1$}}}
\put(95,35){\makebox(0,0){\small{$-1$}}}
\put(110,35){\makebox(0,0){\small{$0$}}}
\put(125,35){\makebox(0,0){\small{$-1$}}}
\put(140,35){\makebox(0,0){\small{$-1$}}}
\put(155,35){\makebox(0,0){\small{$-1$}}}
\put(170,35){\makebox(0,0){\small{$-1$}}}
\put(185,35){\makebox(0,0){\small{$-1$}}}
\put(200,35){\makebox(0,0){\small{$0$}}}
\put(215,35){\makebox(0,0){\small{$0$}}}
\put(230,35){\makebox(0,0){\small{$0$}}}
\put(245,35){\makebox(0,0){\small{$0$}}}
\put(260,35){\makebox(0,0){\small{$0$}}}
\put(275,35){\makebox(0,0){\small{$0$}}}
\put(290,35){\makebox(0,0){\small{$1$}}}

\put(5,45){\makebox(0,0){\small{$0$}}}
\put(20,45){\makebox(0,0){\small{$0$}}}
\put(35,45){\makebox(0,0){\small{$-1$}}}
\put(50,45){\makebox(0,0){\small{$-1$}}}
\put(65,45){\makebox(0,0){\small{$-1$}}}
\put(80,45){\makebox(0,0){\small{$-1$}}}
\put(95,45){\makebox(0,0){\small{$-1$}}}
\put(110,45){\makebox(0,0){\small{$0$}}}
\put(125,45){\makebox(0,0){\small{$-1$}}}
\put(140,45){\makebox(0,0){\small{$-1$}}}
\put(155,45){\makebox(0,0){\small{$-1$}}}
\put(170,45){\makebox(0,0){\small{$-1$}}}
\put(185,45){\makebox(0,0){\small{$-1$}}}
\put(200,45){\makebox(0,0){\small{$0$}}}
\put(215,45){\makebox(0,0){\small{$0$}}}
\put(230,45){\makebox(0,0){\small{$0$}}}
\put(245,45){\makebox(0,0){\small{$0$}}}
\put(260,45){\makebox(0,0){\small{$0$}}}
\put(275,45){\makebox(0,0){\small{$0$}}}
\put(290,45){\makebox(0,0){\small{$1$}}}

\put(5,55){\makebox(0,0){\small{$0$}}}
\put(20,55){\makebox(0,0){\small{$0$}}}
\put(35,55){\makebox(0,0){\small{$-1$}}}
\put(50,55){\makebox(0,0){\small{$-1$}}}
\put(65,55){\makebox(0,0){\small{$-1$}}}
\put(80,55){\makebox(0,0){\small{$-1$}}}
\put(95,55){\makebox(0,0){\small{$\hspace{-1mm}{\bf -1}$}}}
\put(110,55){\makebox(0,0){\small{${\bf 0}$}}}
\put(125,55){\makebox(0,0){\small{$-1$}}}
\put(140,55){\makebox(0,0){\small{$-1$}}}
\put(155,55){\makebox(0,0){\small{$-1$}}}
\put(170,55){\makebox(0,0){\small{$-1$}}}
\put(185,55){\makebox(0,0){\small{$\hspace{-1mm}{\bf -1}$}}}
\put(200,55){\makebox(0,0){\small{${\bf 0}$}}}
\put(215,55){\makebox(0,0){\small{$0$}}}
\put(230,55){\makebox(0,0){\small{$0$}}}
\put(245,55){\makebox(0,0){\small{$0$}}}
\put(260,55){\makebox(0,0){\small{$0$}}}
\put(275,55){\makebox(0,0){\small{${\bf 0}$}}}
\put(290,55){\makebox(0,0){\small{$1$}}}

\put(5,65){\makebox(0,0){\small{$1$}}}
\put(20,65){\makebox(0,0){\small{$1$}}}
\put(35,65){\makebox(0,0){\small{$0$}}}
\put(50,65){\makebox(0,0){\small{$0$}}}
\put(65,65){\makebox(0,0){\small{$0$}}}
\put(80,65){\makebox(0,0){\small{$0$}}}
\put(96,65){\makebox(0,0){\small{${\bf 0}$}}}
\put(110,65){\makebox(0,0){\small{$1$}}}
\put(125,65){\makebox(0,0){\small{$0$}}}
\put(140,65){\makebox(0,0){\small{$0$}}}
\put(155,65){\makebox(0,0){\small{$0$}}}
\put(170,65){\makebox(0,0){\small{$0$}}}
\put(186,65){\makebox(0,0){\small{${\bf 0}$}}}
\put(200,65){\makebox(0,0){\small{$1$}}}
\put(215,65){\makebox(0,0){\small{$1$}}}
\put(230,65){\makebox(0,0){\small{$1$}}}
\put(245,65){\makebox(0,0){\small{$1$}}}
\put(260,65){\makebox(0,0){\small{$1$}}}
\put(275,65){\makebox(0,0){\small{$1$}}}
\put(290,65){\makebox(0,0){\small{$2$}}}

\put(5,75){\makebox(0,0){\small{$0$}}}
\put(20,75){\makebox(0,0){\small{$0$}}}
\put(35,75){\makebox(0,0){\small{$-1$}}}
\put(50,75){\makebox(0,0){\small{$-1$}}}
\put(65,75){\makebox(0,0){\small{$-1$}}}
\put(80,75){\makebox(0,0){\small{$-1$}}}
\put(95,75){\makebox(0,0){\small{$-1$}}}
\put(110,75){\makebox(0,0){\small{$0$}}}
\put(125,75){\makebox(0,0){\small{$-1$}}}
\put(140,75){\makebox(0,0){\small{$-1$}}}
\put(155,75){\makebox(0,0){\small{$-1$}}}
\put(170,75){\makebox(0,0){\small{$-1$}}}
\put(185,75){\makebox(0,0){\small{$-1$}}}
\put(200,75){\makebox(0,0){\small{$0$}}}
\put(215,75){\makebox(0,0){\small{$0$}}}
\put(230,75){\makebox(0,0){\small{$0$}}}
\put(245,75){\makebox(0,0){\small{$0$}}}
\put(260,75){\makebox(0,0){\small{$0$}}}
\put(275,75){\makebox(0,0){\small{$0$}}}
\put(290,75){\makebox(0,0){\small{$1$}}}

\put(5,85){\makebox(0,0){\small{$0$}}}
\put(20,85){\makebox(0,0){\small{$0$}}}
\put(35,85){\makebox(0,0){\small{$-1$}}}
\put(50,85){\makebox(0,0){\small{$-1$}}}
\put(65,85){\makebox(0,0){\small{$-1$}}}
\put(80,85){\makebox(0,0){\small{$-1$}}}
\put(95,85){\makebox(0,0){\small{$-1$}}}
\put(110,85){\makebox(0,0){\small{$0$}}}
\put(125,85){\makebox(0,0){\small{$-1$}}}
\put(140,85){\makebox(0,0){\small{$-1$}}}
\put(155,85){\makebox(0,0){\small{$-1$}}}
\put(170,85){\makebox(0,0){\small{$-1$}}}
\put(185,85){\makebox(0,0){\small{$-1$}}}
\put(200,85){\makebox(0,0){\small{$0$}}}
\put(215,85){\makebox(0,0){\small{$0$}}}
\put(230,85){\makebox(0,0){\small{$0$}}}
\put(245,85){\makebox(0,0){\small{$0$}}}
\put(260,85){\makebox(0,0){\small{$0$}}}
\put(275,85){\makebox(0,0){\small{$0$}}}
\put(290,85){\makebox(0,0){\small{$1$}}}

\put(5,95){\makebox(0,0){\small{$0$}}}
\put(20,95){\makebox(0,0){\small{$0$}}}
\put(35,95){\makebox(0,0){\small{$-1$}}}
\put(50,95){\makebox(0,0){\small{$-1$}}}
\put(65,95){\makebox(0,0){\small{$-1$}}}
\put(80,95){\makebox(0,0){\small{$-1$}}}
\put(95,95){\makebox(0,0){\small{$-1$}}}
\put(110,95){\makebox(0,0){\small{$0$}}}
\put(125,95){\makebox(0,0){\small{$-1$}}}
\put(140,95){\makebox(0,0){\small{$-1$}}}
\put(155,95){\makebox(0,0){\small{$-1$}}}
\put(170,95){\makebox(0,0){\small{$-1$}}}
\put(185,95){\makebox(0,0){\small{$-1$}}}
\put(200,95){\makebox(0,0){\small{$0$}}}
\put(215,95){\makebox(0,0){\small{$0$}}}
\put(230,95){\makebox(0,0){\small{$0$}}}
\put(245,95){\makebox(0,0){\small{$0$}}}
\put(260,95){\makebox(0,0){\small{$0$}}}
\put(275,95){\makebox(0,0){\small{$0$}}}
\put(290,95){\makebox(0,0){\small{$1$}}}

\put(5,105){\makebox(0,0){\small{$0$}}}
\put(20,105){\makebox(0,0){\small{$0$}}}
\put(35,105){\makebox(0,0){\small{$-1$}}}
\put(50,105){\makebox(0,0){\small{$-1$}}}
\put(65,105){\makebox(0,0){\small{$-1$}}}
\put(80,105){\makebox(0,0){\small{$-1$}}}
\put(95,105){\makebox(0,0){\small{$-1$}}}
\put(110,105){\makebox(0,0){\small{$0$}}}
\put(125,105){\makebox(0,0){\small{$-1$}}}
\put(140,105){\makebox(0,0){\small{$-1$}}}
\put(155,105){\makebox(0,0){\small{$-1$}}}
\put(170,105){\makebox(0,0){\small{$-1$}}}
\put(185,105){\makebox(0,0){\small{$-1$}}}
\put(200,105){\makebox(0,0){\small{$0$}}}
\put(215,105){\makebox(0,0){\small{$0$}}}
\put(230,105){\makebox(0,0){\small{$0$}}}
\put(245,105){\makebox(0,0){\small{$0$}}}
\put(260,105){\makebox(0,0){\small{$0$}}}
\put(275,105){\makebox(0,0){\small{$0$}}}
\put(290,105){\makebox(0,0){\small{$1$}}}

\put(5,115){\makebox(0,0){\small{$0$}}}
\put(20,115){\makebox(0,0){\small{$0$}}}
\put(35,115){\makebox(0,0){\small{$-1$}}}
\put(50,115){\makebox(0,0){\small{$-1$}}}
\put(65,115){\makebox(0,0){\small{$-1$}}}
\put(80,115){\makebox(0,0){\small{$-1$}}}
\put(95,115){\makebox(0,0){\small{$\hspace{-1mm}{\bf -1}$}}}
\put(110,115){\makebox(0,0){\small{${\bf 0}$}}}
\put(125,115){\makebox(0,0){\small{$-1$}}}
\put(140,115){\makebox(0,0){\small{$-1$}}}
\put(155,115){\makebox(0,0){\small{$-1$}}}
\put(170,115){\makebox(0,0){\small{$-1$}}}
\put(185,115){\makebox(0,0){\small{$\hspace{-1mm}{\bf -1}$}}}
\put(200,115){\makebox(0,0){\small{${\bf 0}$}}}
\put(215,115){\makebox(0,0){\small{$0$}}}
\put(230,115){\makebox(0,0){\small{$0$}}}
\put(245,115){\makebox(0,0){\small{$0$}}}
\put(260,115){\makebox(0,0){\small{$0$}}}
\put(275,115){\makebox(0,0){\small{${\bf 0}$}}}
\put(290,115){\makebox(0,0){\small{$1$}}}

\put(5,125){\makebox(0,0){\small{$1$}}}
\put(20,125){\makebox(0,0){\small{$1$}}}
\put(35,125){\makebox(0,0){\small{$0$}}}
\put(50,125){\makebox(0,0){\small{$0$}}}
\put(65,125){\makebox(0,0){\small{$0$}}}
\put(80,125){\makebox(0,0){\small{$0$}}}
\put(96,125){\makebox(0,0){\small{${\bf 0}$}}}
\put(110,125){\makebox(0,0){\small{$1$}}}
\put(125,125){\makebox(0,0){\small{$0$}}}
\put(140,125){\makebox(0,0){\small{$0$}}}
\put(155,125){\makebox(0,0){\small{$0$}}}
\put(170,125){\makebox(0,0){\small{$0$}}}
\put(186,125){\makebox(0,0){\small{${\bf 0}$}}}
\put(200,125){\makebox(0,0){\small{$1$}}}
\put(215,125){\makebox(0,0){\small{$1$}}}
\put(230,125){\makebox(0,0){\small{$1$}}}
\put(245,125){\makebox(0,0){\small{$1$}}}
\put(260,125){\makebox(0,0){\small{$1$}}}
\put(275,125){\makebox(0,0){\small{$1$}}}
\put(290,125){\makebox(0,0){\small{$2$}}}

\put(5,135){\makebox(0,0){\small{$1$}}}
\put(20,135){\makebox(0,0){\small{$1$}}}
\put(35,135){\makebox(0,0){\small{$0$}}}
\put(50,135){\makebox(0,0){\small{$0$}}}
\put(65,135){\makebox(0,0){\small{$0$}}}
\put(80,135){\makebox(0,0){\small{$0$}}}
\put(95,135){\makebox(0,0){\small{$0$}}}
\put(110,135){\makebox(0,0){\small{$1$}}}
\put(125,135){\makebox(0,0){\small{$0$}}}
\put(140,135){\makebox(0,0){\small{$0$}}}
\put(155,135){\makebox(0,0){\small{$0$}}}
\put(170,135){\makebox(0,0){\small{$0$}}}
\put(185,135){\makebox(0,0){\small{$0$}}}
\put(200,135){\makebox(0,0){\small{$1$}}}
\put(215,135){\makebox(0,0){\small{$0$}}}
\put(230,135){\makebox(0,0){\small{$0$}}}
\put(245,135){\makebox(0,0){\small{$0$}}}
\put(260,135){\makebox(0,0){\small{$0$}}}
\put(275,135){\makebox(0,0){\small{$0$}}}
\put(290,135){\makebox(0,0){\small{$1$}}}

\put(5,145){\makebox(0,0){\small{$1$}}}
\put(20,145){\makebox(0,0){\small{$1$}}}
\put(35,145){\makebox(0,0){\small{$0$}}}
\put(50,145){\makebox(0,0){\small{$0$}}}
\put(65,145){\makebox(0,0){\small{$0$}}}
\put(80,145){\makebox(0,0){\small{$0$}}}
\put(95,145){\makebox(0,0){\small{$0$}}}
\put(110,145){\makebox(0,0){\small{$1$}}}
\put(125,145){\makebox(0,0){\small{$0$}}}
\put(140,145){\makebox(0,0){\small{$0$}}}
\put(155,145){\makebox(0,0){\small{$0$}}}
\put(170,145){\makebox(0,0){\small{$0$}}}
\put(185,145){\makebox(0,0){\small{$0$}}}
\put(200,145){\makebox(0,0){\small{$1$}}}
\put(215,145){\makebox(0,0){\small{$0$}}}
\put(230,145){\makebox(0,0){\small{$0$}}}
\put(245,145){\makebox(0,0){\small{$0$}}}
\put(260,145){\makebox(0,0){\small{$0$}}}
\put(275,145){\makebox(0,0){\small{$0$}}}
\put(290,145){\makebox(0,0){\small{$1$}}}

\put(5,155){\makebox(0,0){\small{$1$}}}
\put(20,155){\makebox(0,0){\small{$1$}}}
\put(35,155){\makebox(0,0){\small{$0$}}}
\put(50,155){\makebox(0,0){\small{$0$}}}
\put(65,155){\makebox(0,0){\small{$0$}}}
\put(80,155){\makebox(0,0){\small{$0$}}}
\put(95,155){\makebox(0,0){\small{$0$}}}
\put(110,155){\makebox(0,0){\small{$1$}}}
\put(125,155){\makebox(0,0){\small{$0$}}}
\put(140,155){\makebox(0,0){\small{$0$}}}
\put(155,155){\makebox(0,0){\small{$0$}}}
\put(170,155){\makebox(0,0){\small{$0$}}}
\put(185,155){\makebox(0,0){\small{$0$}}}
\put(200,155){\makebox(0,0){\small{$1$}}}
\put(215,155){\makebox(0,0){\small{$0$}}}
\put(230,155){\makebox(0,0){\small{$0$}}}
\put(245,155){\makebox(0,0){\small{$0$}}}
\put(260,155){\makebox(0,0){\small{$0$}}}
\put(275,155){\makebox(0,0){\small{$0$}}}
\put(290,155){\makebox(0,0){\small{$1$}}}

\put(5,165){\makebox(0,0){\small{$1$}}}
\put(20,165){\makebox(0,0){\small{$1$}}}
\put(35,165){\makebox(0,0){\small{$0$}}}
\put(50,165){\makebox(0,0){\small{$0$}}}
\put(65,165){\makebox(0,0){\small{$0$}}}
\put(80,165){\makebox(0,0){\small{$0$}}}
\put(95,165){\makebox(0,0){\small{$0$}}}
\put(110,165){\makebox(0,0){\small{$1$}}}
\put(125,165){\makebox(0,0){\small{$0$}}}
\put(140,165){\makebox(0,0){\small{$0$}}}
\put(155,165){\makebox(0,0){\small{$0$}}}
\put(170,165){\makebox(0,0){\small{$0$}}}
\put(185,165){\makebox(0,0){\small{$0$}}}
\put(200,165){\makebox(0,0){\small{$1$}}}
\put(215,165){\makebox(0,0){\small{$0$}}}
\put(230,165){\makebox(0,0){\small{$0$}}}
\put(245,165){\makebox(0,0){\small{$0$}}}
\put(260,165){\makebox(0,0){\small{$0$}}}
\put(275,165){\makebox(0,0){\small{$0$}}}
\put(290,165){\makebox(0,0){\small{$1$}}}

\put(5,175){\makebox(0,0){\small{$1$}}}
\put(20,175){\makebox(0,0){\small{$1$}}}
\put(35,175){\makebox(0,0){\small{$0$}}}
\put(50,175){\makebox(0,0){\small{$0$}}}
\put(65,175){\makebox(0,0){\small{$0$}}}
\put(80,175){\makebox(0,0){\small{$0$}}}
\put(95,175){\makebox(0,0){\small{${\bf 0}$}}}
\put(110,175){\makebox(0,0){\small{$1$}}}
\put(125,175){\makebox(0,0){\small{$0$}}}
\put(140,175){\makebox(0,0){\small{$0$}}}
\put(155,175){\makebox(0,0){\small{$0$}}}
\put(170,175){\makebox(0,0){\small{$0$}}}
\put(185,175){\makebox(0,0){\small{${\bf 0}$}}}
\put(200,175){\makebox(0,0){\small{$1$}}}
\put(215,175){\makebox(0,0){\small{$0$}}}
\put(230,175){\makebox(0,0){\small{$0$}}}
\put(245,175){\makebox(0,0){\small{$0$}}}
\put(260,175){\makebox(0,0){\small{$0$}}}
\put(275,175){\makebox(0,0){\small{${\bf 0}$}}}
\put(290,175){\makebox(0,0){\small{$1$}}}

\put(5,185){\makebox(0,0){\small{$2$}}}
\put(20,185){\makebox(0,0){\small{$2$}}}
\put(35,185){\makebox(0,0){\small{$1$}}}
\put(50,185){\makebox(0,0){\small{$1$}}}
\put(65,185){\makebox(0,0){\small{$1$}}}
\put(80,185){\makebox(0,0){\small{$1$}}}
\put(95,185){\makebox(0,0){\small{$1$}}}
\put(110,185){\makebox(0,0){\small{$2$}}}
\put(125,185){\makebox(0,0){\small{$1$}}}
\put(140,185){\makebox(0,0){\small{$1$}}}
\put(155,185){\makebox(0,0){\small{$1$}}}
\put(170,185){\makebox(0,0){\small{$1$}}}
\put(185,185){\makebox(0,0){\small{$1$}}}
\put(200,185){\makebox(0,0){\small{$2$}}}
\put(215,185){\makebox(0,0){\small{$1$}}}
\put(230,185){\makebox(0,0){\small{$1$}}}
\put(245,185){\makebox(0,0){\small{$1$}}}
\put(260,185){\makebox(0,0){\small{$1$}}}
\put(275,185){\makebox(0,0){\small{$1$}}}
\put(290,185){\makebox(0,0){\small{$2$}}}

\end{picture}

\vspace{1mm}

The table shows that
$\bH_0=\calt^-_{2}\oplus \calt_{2}(1)^3\oplus \calt_{0}(1)^2 \ \ \ \mbox{and} \ \ \
\bH_1=\calt_0(1).$

From the picture one can  read the contributions given by $S_0$ in $PE_1(T,Q,\h)$. They are
 $Q^0T^{ 36}\h^0$, $2Q^0T^{ 30}\h^0$, $Q^0T^{ 24}\h^1$, $2Q^0T^{ 24}\h^0$, $2Q^0T^{ 18}\h^1$, $Q^0T^{12}\h^1$ and
 $Q^0T^{ 0}\h^0$ (see the boldface generators in the diagram above). From all these only the following terms survives in $(E^7_{*,*})_0=(E^\infty_{*,*})_0$, they are
  $Q^0T^{ 36}\h^0$, $Q^0T^{ 30}\h^0$, $Q^0T^{12}\h^1$ and
 $Q^0T^{ 0}\h^0$.  The degeneration happens at the level of $E^6_{*,*}$, via the maps
 $d^6_{-24,25}:\Z=(E^6_{-24,25})_0\to \Z^2=(E^6_{-30,30})_0$ and
 $d^6_{-18,19}:\Z^2=(E^6_{-18,19})_0\to \Z^2=(E^6_{-24,24})_0$.

 Though $S_n$ is contractible  for any $n\geq 1$, its shape and the corresponding nontrivial  terms $T^lQ^n\h^b$
 in $PE_k$ are rather non-trivial, because of the existence of `lagoons' of the spaces $S_n$.
 Interestingly enough,  the presence and shape of `lagoons' of $S_0$
 are repeated in the other spaces $\{S_n\}_{n\geq 0}$ as well, and the regularity and sides of the
   recurrences are not immediate  transparent  even if we look at larger diagrams (say, on the rectangle
   of size $100\times 100$).

  However, in subsection \ref{ss:einfty} we will prove a conceptual regularity/periodicity statement. The period
  $p$ considered in  Theorem \ref{th:peinfty} applied for
   the case of Example \ref{ex:latNV1} is $-(s,s)=156$, in the case of the present  Example
  \ref{ex:twonodes} it  is 468. On the other hand, in some cases it is 2,
   e.g. for the $A_2$ singularity with either $s=E_1^*$ or $s=E_1^*+E_2^*$.
  For this last case we provide a complete description below.

\end{example}

\begin{example}\label{ex:22}{\bf (The $\Z[U]$ module structure and $PE_k(T,Q,\h)$ for  $A_2$)}

We have already seen that
different  choices of the Weil divisor, or, equivalently, of the semigroup element $s\in\calS'$,
produce very different filtrations, spectral sequences and Poincar\'e series.
We exemplify this
phenomenon by the graph
\begin{picture}(40,15)(-5,2)
\put(5,5){\circle*{4}}
\put(30,5){\circle*{4}}
\put(5,5){\line(1,0){25}}
\put(5,11){\makebox(0,0){\footnotesize{$-2$}}}
\put(30,11){\makebox(0,0){\footnotesize{$-2$}}}
\end{picture}\ .
Additionally, in this case we also show how in the expression of the
Poincar\'e series certain characters and
periodicity  appear.

We take $h=0$, hence the weight function on the lattice points is $\chi(l)=\chi(l_1,l_2)=
l_1^2+l_2^2-l_1l_2$.  The minimal value of $\chi$ is zero.
Some of the  spaces $S_n$ are listed below:

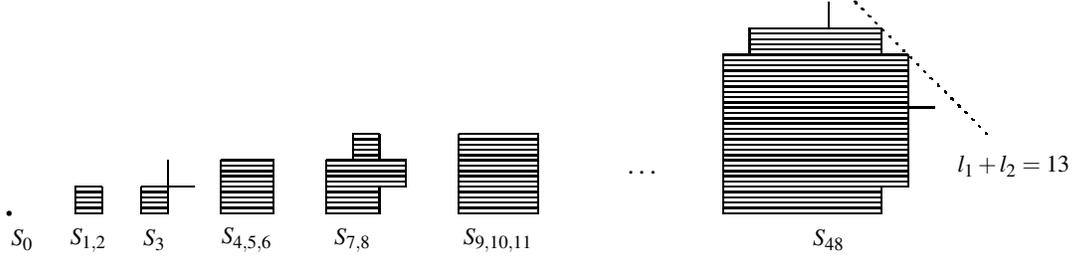
\begin{picture}(300,100)(0,-10)

\put(10,10){\circle*{2}}
\put(15,0){\makebox(0,0){\footnotesize{$S_0$}}}

\put(35,10){\line(1,0){10}}\put(35,10){\line(0,1){10}}
\put(45,10){\line(0,1){10}}\put(35,20){\line(1,0){10}}
\put(35,12){\line(1,0){10}}\put(35,14){\line(1,0){10}}\put(35,16){\line(1,0){10}}\put(35,18){\line(1,0){10}}
\put(40,0){\makebox(0,0){\footnotesize{$S_{1,2}$}}}

\put(60,10){\line(1,0){10}}\put(60,10){\line(0,1){10}}
\put(70,10){\line(0,1){20}}\put(60,20){\line(1,0){20}}
\put(60,12){\line(1,0){10}}\put(60,14){\line(1,0){10}}\put(60,16){\line(1,0){10}}\put(60,18){\line(1,0){10}}
\put(65,0){\makebox(0,0){\footnotesize{$S_3$}}}


\put(90,10){\line(1,0){20}}\put(90,10){\line(0,1){20}}
\put(110,10){\line(0,1){20}}\put(90,30){\line(1,0){20}}
\put(90,12){\line(1,0){20}}\put(90,14){\line(1,0){20}}\put(90,16){\line(1,0){20}}
\put(90,18){\line(1,0){20}}\put(90,20){\line(1,0){20}}\put(90,22){\line(1,0){20}}\put(90,24){\line(1,0){20}}
\put(90,26){\line(1,0){20}}\put(90,28){\line(1,0){20}}
\put(100,0){\makebox(0,0){\footnotesize{$S_{4,5,6}$}}}

\put(130,10){\line(1,0){20}}\put(130,10){\line(0,1){20}}
\put(150,10){\line(0,1){10}}\put(150,20){\line(1,0){10}}
\put(130,30){\line(1,0){10}}\put(140,30){\line(0,1){10}}
\put(140,40){\line(1,0){10}}
\put(150,40){\line(0,-1){10}}
\put(150,30){\line(1,0){10}}\put(160,30){\line(0,-1){10}}

\put(130,12){\line(1,0){20}}\put(130,14){\line(1,0){20}}\put(130,16){\line(1,0){20}}
\put(130,18){\line(1,0){20}}\put(130,20){\line(1,0){30}}\put(130,22){\line(1,0){30}}
\put(130,24){\line(1,0){30}}
\put(130,26){\line(1,0){30}}\put(130,28){\line(1,0){30}}
\put(140,30){\line(1,0){10}}\put(140,32){\line(1,0){10}}\put(140,34){\line(1,0){10}}
\put(140,36){\line(1,0){10}}\put(140,38){\line(1,0){10}}
\put(140,0){\makebox(0,0){\footnotesize{$S_{7,8}$}}}

\put(180,10){\line(1,0){30}}\put(180,10){\line(0,1){30}}
\put(210,10){\line(0,1){30}}\put(180,40){\line(1,0){30}}

\put(180,12){\line(1,0){30}}\put(180,14){\line(1,0){30}}\put(180,16){\line(1,0){30}}
\put(180,18){\line(1,0){30}}\put(180,20){\line(1,0){30}}\put(180,22){\line(1,0){30}}
\put(180,24){\line(1,0){30}}
\put(180,26){\line(1,0){30}}\put(180,28){\line(1,0){30}}
\put(180,30){\line(1,0){30}}\put(180,32){\line(1,0){30}}\put(180,34){\line(1,0){30}}
\put(180,36){\line(1,0){30}}\put(180,38){\line(1,0){30}}
\put(195,0){\makebox(0,0){\footnotesize{$S_{9,10,11}$}}}

\put(250,25){\makebox(0,0){$\cdots $}}

\put(280,10){\line(1,0){60}}
\put(280,10){\line(0,1){60}}
\put(280,70){\line(1,0){10}}
\put(290,70){\line(0,1){10}}

\put(290,80){\line(1,0){50}}
\put(340,80){\line(0,-1){10}}
\put(340,70){\line(1,0){10}}
\put(350,70){\line(0,-1){50}}
\put(350,20){\line(-1,0){10}}
\put(340,20){\line(0,-1){10}}
\put(320,80){\line(0,1){10}}
\put(350,50){\line(1,0){10}}

\put(280,12){\line(1,0){60}}
\put(280,14){\line(1,0){60}}
\put(280,16){\line(1,0){60}}
\put(280,18){\line(1,0){60}}
\put(280,20){\line(1,0){70}}
\put(280,22){\line(1,0){70}}
\put(280,24){\line(1,0){70}}
\put(280,26){\line(1,0){70}}
\put(280,28){\line(1,0){70}}

\put(280,30){\line(1,0){70}}
\put(280,32){\line(1,0){70}}
\put(280,34){\line(1,0){70}}
\put(280,36){\line(1,0){70}}
\put(280,38){\line(1,0){70}}

\put(280,40){\line(1,0){70}}
\put(280,42){\line(1,0){70}}
\put(280,44){\line(1,0){70}}
\put(280,46){\line(1,0){70}}
\put(280,48){\line(1,0){70}}

\put(280,50){\line(1,0){70}}
\put(280,52){\line(1,0){70}}
\put(280,54){\line(1,0){70}}
\put(280,56){\line(1,0){70}}
\put(280,58){\line(1,0){70}}

\put(280,60){\line(1,0){70}}
\put(280,62){\line(1,0){70}}
\put(280,64){\line(1,0){70}}
\put(280,66){\line(1,0){70}}
\put(280,68){\line(1,0){70}}

\put(290,70){\line(1,0){50}}
\put(290,72){\line(1,0){50}}
\put(290,74){\line(1,0){50}}
\put(290,76){\line(1,0){50}}
\put(290,78){\line(1,0){50}}

\put(390,28){\makebox(0,0){\footnotesize{$l_1+l_2=13$}}}
\dashline[60]{1}(380,40)(330,90)

\put(320,0){\makebox(0,0){\footnotesize{$S_{48}$}}}

\end{picture}

Since the graph is rational, all the spaces $S_n$ are contractible, $\bH_0=\calt_0^-$ and
$\bH_{>0}=0$. However, the filtrations of the spaces $S_n$ can be highly nontrivial.

\vspace{1mm}

{\bf Case I.}\ First, we choose $s=E_1^*$, that is, $\frX_{-d}=\{l\,:\, l_1\geq d\}$.

A computation shows that each nonempty $S_n\cap \frX_{-d}$ is contractible. Therefore,
$H_*(S_n\cap \frX_{-d}, S_n\cap \frX_{-d-1})$ is nontrivial exactly when
$S_n\cap \frX_{-d}\not=\emptyset$ and $S_n\cap \frX_{-d-1} =\emptyset$, and it is concentrated in homological degree zero.

For any fixed $l_1$ let ${\rm min}(l_1)$ be the minimal value of $l_2\mapsto \chi(l_1,l_2)$, $l_2\in\Z$.
It equals
\begin{equation}\label{eq:min}
{\rm min}(l_1)=\frac{3l_1^2+\epsilon(l_1)}{4}, \ \ \mbox{where } \ \ \
\epsilon(l_1)=\left\{ \begin{array}{ll} 0 & \mbox{if \ $l_1$ \ is even,} \\
 1 & \mbox{if \ $l_1$ \ is odd.}  \end{array} \right.
\end{equation}
Therefore, $EP_1(T,Q,\h)=EP_1(T,Q)$ equals
\begin{equation*}\label{eq:ep22}
\begin{split}
\sum_{l_1\geq 0} \, T^{l_1}( \, Q^{{\rm min}(l_1)}+\cdots +
 Q^{{\rm min}(l_1+1)-1}\,)
&=\frac{1}{1-Q}\cdot \big( T^0( Q^{{\rm min}(0)}- Q^{{\rm min}(1)}) +
 T^1( Q^{{\rm min}(1)}- Q^{{\rm min}(2)})+\cdots \\
&=\frac{1}{1-Q}\cdot \big( \, 1+(T-1)\cdot \sum_{l\geq 1} T^{l-1}Q^{{\rm min}(l)}\,\big).\\
\end{split}
\end{equation*}
If $(E^1_{-d,q})_n\not=0$ then necessarily $q=d=l_1$ and $\min(d)\leq n<\min(d+1)$.
As a  $\Z[U]$--module $$\oplus _{d\geq 0} (E^1_{-d,d})_n=
\oplus_{d\geq 0} \, \calt_{-2\min(d)} (\min(d+1)-\min(d)).$$
 Moreover, the spectral sequence degenerates at $E^1$ level, $(E^1_{*,*})_n=(E^\infty_{*,*})_n$ for any $n$ compatibly with the $U$--action,
and  $EP_k(T,Q,\h)=EP_1(T,Q,\h)=EP_1(T,Q,\h=0)$ for any $k\geq 1$.

\vspace{1mm}

{\bf Case II.} Next, assume that $s=E_1^*+E_2^*$, i.e.  $\frX_{-d}=\{l\,:\, l_1+l_2\geq d\}$. Write
$|l|:=l_1+l_2$.

We first compute the filtration $\{F_{-d}\bH_0(\frX)_{-2n}\}_{d\geq 0}$. For any fixed $n\geq 0$ the space $S_n$ is contractible, hence
$H_*(S_n,\Z)=H_0(S_n,\Z)=\Z$. Let $d(n)$ be such that $S_n\cap \frX_{-d(n)}\not=\emptyset$, but $S_n\cap \frX_{-d(n)-1}=\emptyset$.
Then $F_{-d}\bH_0(\frX)_{-2n}=\Z$ for $d\leq d(n)$ and it is zero for $d>d(n)$.
The weights of the vertices of a `diagonal cube' are: $\chi(l,l)=l^2$,
$\chi(l+1,l)=\chi(l, l+1)=l^2+l+1$ and $\chi(l+1,l+1)=(l+1)^2$.
Therefore (via a computation),  the level $d=2l$ serves as $d(n)$  for any $n$ with $l^2\leq n<l^2+l+1$ and $d=2l+1$ serves as $d(n)$ whenever
$l^2+l+1\leq n<(l+1)^2$.
Note that by these inequalities in the nontrivial terms of
 $\oplus _{n,d}(E^\infty _{-d,d})_n$
  the  filtration degree $d$ is uniquely determined by the weight
degree $n$.
In particular,
\begin{equation}\label{eq:epinfty}\begin{split}
EP_{\infty}(T,Q,\h)&=\sum_{l\geq 0} \, T^{2l}\cdot (Q^{l^2}+\cdots + Q^{l^2+l})+
T^{2l+1}\cdot (Q^{l^2+l+1}+\cdots + Q^{l^2+2l})\\
&=\frac{1}{1-Q}\cdot \sum_{l\geq 0}\, T^{2l} ( Q^{l^2}-Q^{l^2+l+1})+T^{2l+1}(Q^{l^2+l+1}-Q^{l^2+2l+1})\,\big)\\
&=\frac{1}{1-Q}+ \frac{T-1}{1-Q}\cdot \Big(\, \sum_{l\geq 0}T^{2l}Q^{l^2+l+1} +\sum_{l\geq 0}T^{2l+1}Q^{(l+1)^2}\, \Big).
\end{split}\end{equation}
and, as a $\Z[U]$--module,
$${\rm Gr}^F_{*}\bH_0(\frX)_*=\bigoplus _{n,d}(E^\infty _{-d,d})_n=
\bigoplus_{d\geq 0} \, \Big( \ \calt_{-2l^2}(l+1)\,\oplus \, \calt_{-2(l^2+l+1)}(l)\, \Big).$$


In Case II,
 $EP_1(T,Q,\h)$ is more complicated than in Case I.  E.g., from the above picture of spaces $S_n$  one reads directly that
the coefficient of $Q^{48}$ in $EP_1(T,Q,\h)$ is $2T^{13}\h^0+2T^{12}\h^0+T^{12}\h^1+2T^{11}\h^1$.
Hence, for $n=48$  the spectral sequence has the following page $(E^1_{*,*})_{48}$:

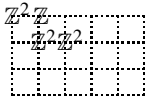
\begin{picture}(300,55)(-40,42)

 \dashline[200]{1}(30,50)(30,80)
  \dashline[200]{1}(20,50)(20,80)\dashline[200]{1}(10,50)(10,80)
  \dashline[200]{1}(0,50)(0,80)
  \dashline[200]{1}(-10,50)(-10,80)
 \dashline[200]{1}(40,50)(40,80)

  \dashline[200]{1}(-10,50)(40,50)\dashline[200]{1}(-10,60)(40,60)
  \dashline[200]{1}(-10,70)(40,70)\dashline[200]{1}(-10,80)(40,80)

    \put(2,71){\makebox(0,0){\footnotesize{$\Z^2$}}} \put(12,71){\makebox(0,0){\footnotesize{$\Z^2$}}}
\put(-8,81){\makebox(0,0){\footnotesize{$\Z^2$}}}
 \put(1,80){\makebox(0,0){\footnotesize{$\Z$}}}

  \put(150,60){\makebox(0,0){\footnotesize{$\mbox{where the upper-left corner is $(-13,13)$}$}}}

  \end{picture}

Since $(E^\infty_{*,*})_{48}$ has rank one supported at $(-d,q)=(-13,13)$, we get that the differential
$(d^1_{*,*})_{48}$ is nonzero, but all the other differentials $(d^k_{*,*})_{48}$, $k\geq 2$, are zero.

Next, we compute all the entries of $EP_1(T,Q,\h)$.
Consider an arbitrary 2-cube $\square$ with vertices $A=l$, $B=l+E_1$, $C=l+E_2$, $D=l+E_1+E_2$. A computation shows that
\begin{equation}\label{eq:matr}
\chi(B)+\chi(C)=\chi(A)+\chi(D)+1.
\end{equation}
The non-zero terms of $PE_1(T,Q,\h)$ can be localized  in certain cubes.
One can distinguish exactly  two cases (where the vertices are $A=l, B,C,D$).  They are the following:

\vspace{1mm}

(a) There exists $m\in\Z$ such that
 $\chi(A)\leq  m$, and both $\chi(B), \chi(C)>m$. Note that in this case by (\ref{eq:matr}) $\chi(D)>m$ too.
In this case the contribution in $EP_1(T,Q,\h)$ is $T^{|l|}Q^n\h^0$  for all $\chi(A)\leq n< \min\{\chi(B),\chi(C)\}$.

(b) There exists $m\in\Z$ such that
$\chi(A),\ \chi(B), \ \chi(C)\leq m$ but $\chi(D)>m$. In this case the
contribution in $EP_1(T,Q,h)$ is $T^{|l|}Q^n\h^1$  for all $\max\{\chi(B),\chi(C)\}\leq n< \chi(D)$.

\vspace{1mm}

In fact, in both cases $\chi(B)>\chi(A)$ and $\chi(C)>\chi(A)$ (in case (b) use
$\chi(B)-\chi(A)=\chi(D)-\chi(C)+1\geq 2$).
This means that in all these cases $A=l\in\calS$.

Conversely, if $l\in \calS=\{(l_1,l_2)\in\Z^2\,:\,
-2l_1+l_2\leq 0,\ l_1-2l_2\leq 0\}$ then
$$\chi(A)<\min\{\chi(B),\chi(C)\}\leq \max\{\chi(B),\chi(C)\}\leq \chi(D).$$
In conclusion:
\begin{equation}\label{eq:ep1}
EP_1(T,Q,\h)=\sum_{l\in \calS}\ E(l), \ \ \mbox{where $E(l)$ equals }
\end{equation}
\begin{equation*}   \begin{split}
&  T^{|l|}\h^0\big( Q^{\chi(l)}+\cdots + Q^{\min\{\chi(l+E_1),\chi(l+E_2)-1\}}\big)
+T^{|l|}\h^1\big(  Q^{\max\{\chi(l+E_1),\chi(l+E_2)\}}  +\cdots + Q^{\chi(l+E_1+E_2)-1}\big)\\
=&
T^{|l|}\h^0\cdot \frac{Q^{\chi(l)}- Q^{\min\{\chi(l+E_1),\chi(l+E_2)\}}}{1-Q}
+T^{|l|}\h^1\cdot \frac{  Q^{\max\{\chi(l+E_1),\chi(l+E_2)\}} - Q^{\chi(l+E_1+E_2)}}{1-Q}.\end{split}
\end{equation*}

From the expression ${\rm E}(l)$ one can eliminate the $\max$ and $\min$ symbols rewriting it in the two regions
$l_1\geq l_2$ and $l_1\leq l_2$. The computation of the new expressions are left to the reader.

As a $\Z[U]$ module, $\oplus (E^1_{*,*})_*$ has a direct sum decomposition into two summands, these are:
$$\bigoplus_{n,d}(E^1_{-d,d})_n=\bigoplus_{l\in \calS} \, \calt_{-2\chi(l)}\Big(
\min\{\chi(l+E_1),\chi(l+E_2)\}-\chi(l)\Big), $$
$$\bigoplus_{n,d}(E^1_{-d,d+1})_n=\bigoplus_{l\in \calS} \, \calt_{-2\chi(l)}\Big(
\chi(l+E_1+E_2))-\max\{\chi(l+E_1),\chi(l+E_2)\}\Big). $$
Note also that  $n$ does not determine the semigroup element $l$.

Finally, we claim that $E^2=E^\infty$, i.e. $k_{(\Gamma_\phi,s)}=2$.
Indeed,
analysing the geometric meaning (basically the definition) of the spectral sequence,
the claim reduces to the  verification of the fact that there exists no weight $n$
and lattice point $l$ such that
$l, \ l+2E_1, \ l+2E_2$ are in $S_n$ but $l+E_1+E_2\not\in S_n$. (Compare with the arguments from Example \ref{ex:latNV1}.)
\end{example}

 \begin{example}\label{ex:237} Here we exemplify with
 some additional  details the Filtered Reduction Theorem.

 Consider the following minimal good resolution graph.
 The link is the integral homology sphere $\Sigma(2,3,7)$, so $H=L'/L=0$:
 we take $h=0$, $k_h=-Z_K$ and $\chi(l)=(l,l-Z_K)/2$.

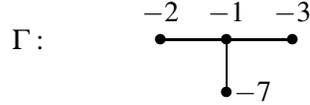
\begin{picture}(300,45)(0,0)
\put(150,25){\circle*{4}}
\put(175,25){\circle*{4}}
\put(200,25){\circle*{4}}
\put(175,5){\circle*{4}}
\put(150,25){\line(1,0){50}}
\put(175,25){\line(0,-1){20}}
\put(100,25){\makebox(0,0){$\Gamma:$}}
\put(150,35){\makebox(0,0){$-2$}}
\put(175,35){\makebox(0,0){$-1$}}
\put(200,35){\makebox(0,0){$-3$}}
\put(185,5){\makebox(0,0){$-7$}}
\end{picture}

 Let $E_0\in L$ be the base cycle associated with the  central $(-1)$ vertex.
We set $s=E^*_0\in \calS$, and filter $\bH_0$ via the induced filtration $\{\frX_{-d}\}_d$. (This is the most natural filtration:
in the analytic realization of $\Gamma$ as a
weighted homogeneous singularity $\{x^2+y^3+z^7=0\}\subset \C^3$,  the corresponding analytic divisorial
filtration associated with $E_0$  coincides with the filtration of the
 local graded ring provided by the $\C^*$--action.)

Note that $\Gamma$ is an AR-graph where  $E_0$ is  an SR set, cf. \ref{ss:redth}.
 For any $l'\in L'$ let
 $m_0(l')$ be the coefficient  of $E_0$ in $l'$. Then for any $\bar{l}\in\Z_{\geq 0}$ there exists a
 unique minimal cycle $x(\bar{l}) \in L$ such that $m_0(x(\bar{l}))=\bar{l}$ and
 $(E_v,  x(\bar{l}))\leq 0$ for any $E_v\not=E_0$, see \cite{NOSz,Nkonyv} or subsection \ref{ss:redth} here.
The cycle $x(\bar{l})$ satisfies another universal property too:
for any $l\in L$ with $m_0(l)=\bar{l}$ one has $\chi(l)\geq \chi(x(\bar{l}))$, cf. \cite[Prop. 7.3.28]{Nkonyv}.
Moreover, $\bar{l}\mapsto x(\bar{l})$ is a sequence of increasing cycles.

Finally, one defines $\overline{w}(\bar{l}):=\chi(x(\bar{l}))$, and by `reduction theorem' (see Theorem 7.3.37 and
Theorem 11.3.5 from \cite{Nkonyv} or subsection \ref{ss:redth} here)
$\bH_*(\frX,\chi)$ agrees with  the lattice homology of $\Z_{\geq 0}\subset \R_{\geq 0}$ associated with the weight function $\bar{l}\mapsto \overline{w}(\bar{l})$.

The new weight function $\bar{l}\mapsto \overline{w}(\bar{l})$ can be computed inductively as follows
(see again \cite{NOSz} or \cite[\S 11.3.A]{Nkonyv}):
$\overline{w}(\bar{l}+1)-\overline{w}(\bar{l})=1+ N(\bar{l})$, where
$N(\bar{l})=\bar{l}-\lceil \bar{l}/2\rceil-\lceil \bar{l}/3\rceil
-\lceil \bar{l}/7\rceil$, i.e. $\tau(0)=0$ and
$$\overline{w}(\bar{l})=\sum_{j=0}^{\bar{l}-1} \big(1+j-\lceil j/2\rceil-\lceil j/3\rceil
- \lceil j/ 7\rceil\big).$$
It is a quasi (or, periodic) quadratic function.  The $\overline{w}$--values for
 $0\leq \bar{l}\leq 8$ are 0, 1, 0, 0, 0, 0, 0, 1, 1.
The function $\bar{l}\mapsto \overline{w}(\bar{l})$ is non-decreasing for $\bar{l}\geq 2$.
The lattice homology is given by
 $\bH_0(\frX, w)=\calt^-_0\oplus \calt_0(1)$  and $\bH_{>0}(\frX, w)=0$. The homological graded root is

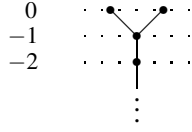
\begin{picture}(300,62)(80,330)

\put(180,380){\makebox(0,0){\footnotesize{$0$}}}
\put(177,370){\makebox(0,0){\footnotesize{$-1$}}}
\put(177,360){\makebox(0,0){\footnotesize{$-2$}}}
\dashline{1}(200,370)(240,370)
\dashline{1}(200,380)(240,380)
\dashline{1}(200,360)(240,360)
\put(220,345){\makebox(0,0){$\vdots$}} \put(220,360){\circle*{3}}
\put(210,380){\circle*{3}}
\put(230,380){\circle*{3}} 
\put(220,370){\circle*{3}} \put(220,370){\line(0,-1){20}}
\put(220,370){\line(1,1){10}} \put(210,380){\line(1,-1){10}}



\end{picture}


Since the reduced rank is one, the spectral sequence degenerates at $E^1$ level.

The graded $\Z[U]$--modules  ${\rm F}_{-d}\bH_0(\frX)$ for $d\geq 0$ are determined  as follows.

$S_0$ has two components, one of them, $S_0'$, consists of the  lattice point $0$, the other one, $S_0''$,
has more lattice points, e.g. $x(\bar{l})$ for $2\leq \bar{l}\leq 6$.
The submodule $F_{-1}\bH_0$ is obtained from $\bH_0$ by deleting the
component $S_0'$, hence $F_{-1}\bH_0\simeq \Z[U]$, where the generator 1 corresponds to the
component $S_0''$.
From the above values
$\overline{w}(\bar{l})$ one obtains  $F_{-d}\bH_0=F_{-1}\bH_0$ for $1\leq d\leq 6$.

Assume next that for $d>1$ and  $n>0$ we have  $S_n \cap \frX_{-d}\not=\emptyset$
(while in the case of  $n=0$ we have the similar assumption  $S''_0\cap \frX_{-d}\not=\emptyset$).
Then there exists $l\in L$ such that $\chi(l)\leq n$ and $m_0(l)\geq d$. But,
from the universal property of $x(\bar{l})$  and from the monotonicity of $\overline{w}(\bar{l})$
 we also have $n\geq \chi(l)\geq \chi(x(m_0(l)))\geq
 \chi(x(d))$, hence $x(d)\in S_n \cap \frX_{-d}$ too. That is, the condition
 $S_n \cap \frX_{-d}\not=\emptyset$ can be tested by the cycles of type $x(\bar{l})$. In particular,
$$\max\{d\,:\, S_n\cap \frX_{-d}\not=\emptyset\}=\max\{d\,:\, \overline{w}(d)\leq n\}.$$
E.g., this value for  $S_0''$ is $d=6$, for $S_1$ it is $d=12$,  for $S_2$ it is $d=14$.

In particular, in order to get $F_{-d}\bH_0$ we have to delete those components from
$\bH_0$ (vertices from the graded root associated with $S_n$)
which satisfy $n<\overline{w}(d)$. (Recall also that $(-n)$ is the weight of the vertex in the graded root.)
 The graded $\Z[U]$--modules  ${\rm F}_{-d}\bH_0$ for $d>0$ are illustrated below
via their graded root.
 From the root of $\bH_0$ one has to delete those edges which intersect
the `cutting line'. The $\Z[U]$--module  ${\rm F}_{-d}\bH_0$ sits below the cutting line,
 where the $U$--action is determined from the remaining edges by the usual principle described above.

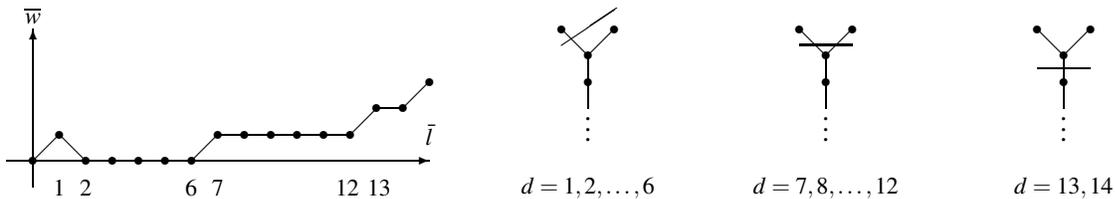
\begin{picture}(500,90)(-170,310)

\put(-180,330){\vector(1,0){160}}
\put(-170,320){\vector(0,1){60}}
\put(-170,385){\makebox(0,0){\footnotesize{$\overline{w}$}}}
\put(-20,340){\makebox(0,0){\footnotesize{$\bar{l}$}}}
\put(-170,330){\line(1,1){10}}
\put(-160,340){\line(1,-1){10}}
\put(-150,330){\line(1,0){40}}
\put(-110,330){\line(1,1){10}}
\put(-100,340){\line(1,0){50}}
\put(-50,340){\line(1,1){10}}\put(-40,350){\line(1,0){10}}
\put(-30,350){\line(1,1){10}}
\put(-160,320){\makebox(0,0){\footnotesize{$1$}}}
\put(-150,320){\makebox(0,0){\footnotesize{$2$}}}
\put(-110,320){\makebox(0,0){\footnotesize{$6$}}}
\put(-100,320){\makebox(0,0){\footnotesize{$7$}}}
\put(-51,320){\makebox(0,0){\footnotesize{$12$}}}
\put(-39,320){\makebox(0,0){\footnotesize{$13$}}}

\put(-170,330){\circle*{3}}
\put(-160,340){\circle*{3}}
\put(-150,330){\circle*{3}}
\put(-140,330){\circle*{3}}
\put(-130,330){\circle*{3}}
\put(-120,330){\circle*{3}}
\put(-110,330){\circle*{3}}
\put(-100,340){\circle*{3}}
\put(-90,340){\circle*{3}}
\put(-80,340){\circle*{3}}
\put(-70,340){\circle*{3}}
\put(-60,340){\circle*{3}}
\put(-50,340){\circle*{3}}
\put(-40,350){\circle*{3}}
\put(-30,350){\circle*{3}}
\put(-20,360){\circle*{3}}

\put(40,345){\makebox(0,0){$\vdots$}} \put(40,360){\circle*{3}}
\put(30,380){\circle*{3}}
\put(50,380){\circle*{3}} 
\put(40,370){\circle*{3}} \put(40,370){\line(0,-1){20}}
\put(40,370){\line(1,1){10}} \put(30,380){\line(1,-1){10}}

\put(130,345){\makebox(0,0){$\vdots$}} \put(130,360){\circle*{3}}
\put(120,380){\circle*{3}}
\put(140,380){\circle*{3}} 
\put(130,370){\circle*{3}} \put(130,370){\line(0,-1){20}}
\put(130,370){\line(1,1){10}} \put(120,380){\line(1,-1){10}}

\put(220,345){\makebox(0,0){$\vdots$}} \put(220,360){\circle*{3}}
\put(210,380){\circle*{3}}
\put(230,380){\circle*{3}} 
\put(220,370){\circle*{3}} \put(220,370){\line(0,-1){20}}
\put(220,370){\line(1,1){10}} \put(210,380){\line(1,-1){10}}

%

\put(40,320){\makebox(0,0){\footnotesize{$d=1,2,\ldots, 6$}}}
\put(130,320){\makebox(0,0){\footnotesize{$d=7, 8, \ldots, 12 $}}}
\put(220,320){\makebox(0,0){\footnotesize{$d=13,14$}}}

\put(30,374){\line(3,2){21}}
\put(120,374){\line(1,0){20}}
\put(210,365){\line(1,0){20}}
\end{picture}

In $PE_\infty(T,Q,\h)$ each pair $(\bar{l}, n)$ gives a contribution $T^{\bar{l}}Q^n\h^0$ whenever $\overline{w}(\bar{l})\leq n<\overline{w}(\bar{l}+1)$.
Thus,
$$PE_\infty(T,Q,\h)=T^0Q^0+\sum_{\bar{l}\geq 2}\ T^{\bar{l}}\cdot \frac{Q^{\overline{w}(\bar{l})}-Q^{\overline{w}(\bar{l}+1)}}{1-Q}
=1+\frac{1}{1-Q}\Big(T^2+(T-1)\cdot \sum  _{\bar{l}\geq 3}T^{\bar{l}-1}Q^{\overline{w}(\bar{l})}\Big).$$
$PE_\infty(T=1,Q,\h)-1/(1-Q)=1=eu(\bH_*)$.
Note also that $$\lim _{Q\to 1}PE_{\infty}(T,Q,\h)
=1+\sum_{\bar{l}\geq 2}(1+N(\bar{l}))\cdot T^{\bar{l}}=
\sum_{\bar{l}\geq 0}\max\{\,0,1+N(\bar{l})\,\} \cdot T^{\bar{l}}.$$
This expression equals $Z_0(\bt)$ reduced to the variable of the central vertex of the resolution, and it also equals the (analytic)  Poincar\'e series of the graded ring associated with the corresponding analytic weighted homogeneous singularity,
see e.g. \cite[(5.1.70]{Nkonyv}.

All the above discussion can be repeated for any star shaped graph (in which case $N(\bar{l})$ is
computed in terms of the Seifert invariants), or even for
an almost--rational graph, with $E_0$ its unique bad vertex (for more details regarding $N(\bar{l})$ and $\overline{w}(\bar{l})$ see  \cite[11.3]{Nkonyv}), or section \ref{s:rat} here.


 \end{example}

\section{The pages  $E^\infty_{*,*}$ (continuation of the $t=1$ case).} \label{ss:einfty}\

\subsection{The structure theorem}\

One of the goals of the present part is to provide structure theorems for $PE_\infty(T,Q,\h)$.
First we prove that $PE_\infty(T,Q,\h)$,  up to finitely many terms,  can be written as a finite
sum of `unsided'
 Jacobi theta function. Then we compute $PE_\infty(T=1,Q,\h)$ and we relate it with several Euler characteristic type invariants (e.g. with $eu(\bH_*)$).

We start with the following observation.
Let  $d_0\in\Z_{\geq 0}$ be a certain integer. For any $d\leq d_0$, if $n$ is sufficiently large, then $S_n$ is contractible and $S_n\cap \frX_{-d_0-1}\not=\emptyset$, hence
 $S_n\cap \frX_{-d-1}\not=\emptyset$ too. Hence for such $n$ we have
 $F_{-d}\bH_*(\frX)_{-2n}=F_{-d-1}\bH_*(\frX)_{-2n}=\Z$, that is, $(E^\infty_{-d,*})_n=0$.
 In other words, for any  $d\leq d_0$ there are only finitely many integers $n$  such that
  $(E^\infty_{-d,*})_n\not =0$.

For the next discussion let us  choose $d_0$ as follows. First,
choose $n_0$ so that $S_n$ is contractible for any $n\geq n_0$ (cf. Prop. \ref{lem:con}).
Then set $d_0$ so large that $S_{n_0}\cap \frX_{-d_0}=\emptyset$. This means that if $d\geq d_0$ and
$S_n\cap\frX_{-d}\not=\emptyset$, then $S_n$ is contractible.
For such $d_0$,
$PE_\infty(T,Q,\h)$ ---  up to finitely many terms (which correspond to $d<d_0$) ---,
has the form
$\sum_{d\geq d_0} T^dQ^n\h^0$. Here a pair $(d,n)$ contributes with $T^dQ^n$  in the sum
if and only if  $S_n\cap \frX_{-d}\not=\emptyset$ and  $S_n\cap \frX_{-d-1}=\emptyset$.

Clearly, this fact remains valid if we replace $d_0$ by any larger integer. In the final choice of $d_0$ we will take into account several additional  properties too, which will be introduces  next.

Since $l\mapsto -(s,l)$ is linear, its image is a subgroup of $\Z$, say $e\Z$ with $e>0$.
In particular,
$S_n\cap \frX_{-ed-a}=S_n\cap \frX_{-ed-e}$ for any $1\leq a\leq e$,
hence ${\rm Gr}^F_{-d}(\bH_*)_*$ is nonzero only for $d\in e\Z$.

Choose also $d_0$ in the form  $e\bar{d}_0$. For any $e\bar{d}\geq e \bar{d}_0$ set
$\min(e\bar{d}):=\min\{\chi_h(l)\,:\, l\in\frX_{-e\bar{d}}\}$. Then automatically
 $\min(e(\bar{d}+1))\geq \min(e\bar{d})$.
 If  $\min(e(\bar{d}+1))> \min(e\bar{d})$ then the monomials in $PE_\infty$ containing $T^{e\bar{d}}$ are the
 following:
$$ T^{e\bar{d}}(Q^{\min(e\bar{d})}+\cdots +Q^{\min(e(\bar{d}+1))-1})=
T^{e\bar{d}}\cdot \frac{Q^{\min(e\bar{d})}-Q^{\min(e(\bar{d}+1))}}{1-Q}.$$
 The right hand side gives the right answer even if  $\min(e(\bar{d}+1))= \min(e\bar{d})$.

For the next formula it is convenient to choose $d_0=e\bar{d}_0$ with the following additional  arithmetical property.
Let $N\in\Z_{>0}$ be the smallest integer  with $\tilde{s}:=N\cdot s\in L$. Let $ep:=-(s,\tilde{s})$. Then we chose $\bar{d}_0$ in the form
$\bar{d}_0=k_0p-1$ for some $k_0\in \Z_{>0}$. For $\bar{d}\geq \bar{d}_0+1$ we write $\bar{d}$ as $kp+q$ with $k\geq k_0$ and $0\leq q<p$.

With all these notations we have that $PE_\infty(T,Q,\h)$, up to finitely many terms,
has the form
\begin{equation*}
\sum_{\bar{d}\geq \bar{d}_0}T^{e\bar{d}}\cdot \frac{Q^{\min(e\bar{d})}-Q^{\min(e(\bar{d}+1))}}{1-Q}=
\frac{1}{1-Q} T^{d_0} Q^{\min(d_0)}+\frac{T^e-1}{T^e(1-Q)} \cdot \sum_{q=0}^{p-1}\
\sum_{k\geq k_0}\  T^{e(kp+q)}Q^{\min (e(kp+q))}.\end{equation*}
A priori, the map $\bar{d}\mapsto \min( e\bar{d})$ might be very irregular, chaotic. However, we
will show that this is not the case, we will find a rather interesting  regularity property.
We start with the following lemma.

\begin{lemma}\label{lem:induc}
For any $k\geq k_0$ and $0\leq q<p$ \ define $\delta(q,k)\in\Z_{\geq 0}$ by
$$\delta(q,k):=\min\{-(s, l_{q,k})-e(kp+q)\,:\ \mbox{where $l_{q,k}$ realizes $\min(e(kp+q))$}\}.$$
Then for any fixed $q$, the map $k\mapsto \delta(q,k)$ is non-increasing.
\end{lemma}
\begin{proof}
Set  $k_2\geq k_1\geq k_0$ and assume that $l_{q,k_1}$  realises  $\min(e(k_1p+q))$ with
$-(s, l_{q,k_1})-e(k_1p+q)=\delta(q,k_1)$. This means that $\chi_h(l_{q,k_1})\leq \chi_h(l)$ for any
$l$ with $-(s,l)\geq e(k_1p+q)$.
We rewrite $\chi_h(l_{q,k_1})\leq \chi_h(l)$ as
$\chi_h(l_{q,k_1})\leq \chi_h(l-l_{q,k_1}+l_{q,k_1})=
 \chi_h(l-l_{q,k_1})+\chi_h(l_{q,k_1})-(l-l_{q,k_1},l_{q,k_1})$, or
\begin{equation}\label{eq:chiin}
\chi_h(l-l_{q,k_1})-(l-l_{q,k_1}, l_{q,k_1})\geq 0.\end{equation}
Set $\bar{l}_{q,k_2}:=l_{q, k_1}+(k_2-k_1)\tilde{s}$ \ and choose any $\bar{l}\in \frX_{-e(k_2p+q)-\delta(q,k_1)}$.
Write $\bar{l}$ as $l+(k_2-k_1)\tilde{s}$. Then  $l\in \frX_{-e(k_1p+q)-\delta(q,k_1)}$. Then, by the constructions,
$-(s, l-l_{q, k_1})\geq 0$.
Thus, using this and (\ref{eq:chiin}),
\begin{equation*}\begin{split}
\chi(\bar{l})&= \chi_h(\bar{l}_{q,k_2})+\chi_h(\bar{l}-\bar{l}_{q,k_2})-(\bar{l}-\bar{l}_{q,k_2},\bar{l}_{q,k_2})\\
&=\chi_h(\bar{l}_{q,k_2})+\chi_h(l-l_{q,k_1})-(l-l_{q,k_1},l_{q,k_1}+(k_2-k_1)Ns)\geq \chi_h(\bar{l}_{q,k_2}).
\end{split}\end{equation*}
Hence $\min(e(k_2p+q))$ can be realized by some
`smallest degree'   $l_{q,k_2}$ with degree  $-(s, l_{q,k_2})\leq -(s,\bar{l}_{q,k_2})=e(k_2p+q)+\delta(q,k_1)$.
Hence $\delta(q,k_2)\leq \delta(q,k_1)$.
\end{proof}
Since $\delta(q,k)\geq 0$, by  the above lemma $k\mapsto \delta(q,k)$ stabilizes. Let us choose $d_0$ (additionally to
 all the above restriction) so large that for any $q$ the map $k\mapsto \delta(q,k)$ is constant for $k\geq k_0$.

For any fixed  $0\leq q<p$ let $l_q$ be a lattice point  for which $\min(k_0p+q)$ is realized and $-(s, l_q)=e(k_0p+q)+\delta(k_0,q)$.
Then, by the above discussion,  for any $k\geq k_0$ the cycle $l_q+(k-k_0)\tilde{s}$ \
 realizes $\min(kp+q)$ for any $k\geq k_0$.
If we substitute this in the perviously proved expression  we
obtain the  following formula  (where we write $m:=k-k_0$).

\begin{theorem}\label{th:peinfty} There exist lattice points $\{l_q\}_{q=0}^{p-1}$ such that for $d_0=k_0p-1$ sufficiently large
$PE_\infty(T,Q,\h)$ --- up to finitely many terms ---
has the form
\begin{equation*}
\frac{T^{d_0}}{1-Q}\, \cdot  \Big[\, Q^{\min(d_0)}+(T^e-1) \cdot \sum_{q=0}^{p-1}\ T^{eq}Q^{\chi_h(l_q)}\,
\sum_{m\geq 0}\  T^{emp}Q^{\chi_h(l_q+m\tilde{s})-\chi_h(l_q)}\,\Big].\end{equation*}
Note also that $emp=-(s,m\tilde{s})$ and  $\chi_h(l_q+m\tilde{s})-\chi_h(l_q)$ is the following quadratic function in $m$:
$$\chi_h(l_q+m\tilde{s})-\chi_h(l_q)=-(m\tilde{s}, m\tilde{s}-Z_K+2s_h+2l_q)/2.$$
\end{theorem}

\begin{remark} Write $s\in\calS'\setminus\{0\}$ as $\sum_i n_iE_i^*$  for certain $n_i\in\Z_{\geq 0}$.

(a) The integer  $e$ used  above equals ${\rm gcd}_i\{n_i\}$.

(b) Assume that  $n_i\in\{0,1\}$ for all $i$, in particular $e=1$.
 Suppose  that above we choose $d_0$ so large that $\min(d)>0$  for any $d\geq d_0$.
Write ${\rm Min}(d):= \min\{\chi_h(l)\,:\, -(s,l)=d\}$.  Note that  for any fixed $d=kp+q$ (see previous proof)
$\delta(q,k)=\min\{d'-d\,:\, d'\geq d,\ {\rm Min}(d')=\min(d)\}$. In particular,
$\delta(q,k)=0$ if and only if $\min(d)$ is realized on $\{-(s,l)=d\}$.

We claim that under the above assumption
  ${\rm Min}(d+1)\geq {\rm Min}(d)$ for any $d\geq d_0$ (hence
   $\delta(q,k)=0$ in the above proof of Theorem \ref{th:peinfty}).
   Indeed, if ${\rm Min}(d+1) <{\rm Min}(d)$ for some $d$, then take $n:={\rm Min}(d+1)$.
Then $0\in S_n$ and there exists $l\in S_n$ with $-(s,l)=d+1$. However, $l$ cannot be connected with 0 in $S_n$  by a path  since there exists no cycle $\bar{l}\in S_n$ with $-(s,\bar{l})=d$ (and any
connecting  path should have at least one element $\bar{l}$ with $-(s,\bar{l})=d$).  This contradicts the contractibility of $S_n$.

In particular, the cycles $l_q$  in Theorem \ref{th:peinfty} can be chosen with $-(s,l_q)=k_0p+q$.
\end{remark}

\begin{corollary}\label{cor:fin}
Up to finitely many  terms, $PE_\infty(T=1,Q,\h)$ has the form $1/(1-Q)$.
\end{corollary}
This theorem describes the asymptotic behaviour of $PE_\infty(T,Q,\h)$ as a finite combinations of sums of type
$\sum_{m\geq 0} T^{-(s,m\tilde{s})}Q^{-(m\tilde{s},m\tilde{s}+k)/2}$. 
But, naturally, the remaining finitely many terms of $PE_\infty$, not covered by this formula but present in
$PE_\infty$, are equally important.
The next group of statement target a complete set of information (under the substitution $T=1$).

\subsection{Properties of $PE_\infty(T=1,Q,\h)$ and  $PE_\infty(T=1,Q,\h=1)$} \

The next Proposition   relates   $PE_\infty(T=1,Q,\h)$  with $\bH_*$.
First recall that for any fixed $b$ and $n$
\begin{equation}\label{eq:pqr}
\sum_{-d+q=b}{\rm rank}\big( E_{-d, q}^\infty\big)_{n}={\rm rank }\, H_b(S_n, \Z)={\rm rank} \,(\bH_b(\frX,w))_{-2n}.
\end{equation}
\begin{proposition}\label{prop:infty} \ 

\noindent (a) $PE_\infty(1, Q,\h)=\sum_{n\geq m_w}\, \big(\, \sum_b \, {\rm rank}\, H_b(S_n,\Z)\,\h^b\, \big)\cdot Q^n.$

\noindent (b) $PE_\infty(1, Q,-1)=\sum_{n\geq m_w}\, \chi_{top}(S_n)\cdot Q^n$ \ (where $\chi_{top}$ denoted the topological Euler characteristic).

\noindent (c) Let $R$ be any  rectangle of type $R(0,c')$ with $c'\geq \lfloor Z_K \rfloor$. Then
$$PE_\infty(1, Q,-1)=\frac{1}{1-Q}\cdot \sum_{\square_q\subset R}\, (-1)^q \, Q^{w_h(\square_q)}.$$
(d) $$\lim_{Q\to 1}\Big( PE_\infty (1,Q, -1)-\frac{1}{1-Q}\Big)=eu(\bH_*(\frX,w_h)).$$
That is, since $PE_\infty (1,Q, -1)-\frac{1}{1-Q}$  is finitely supported (cf.
Corollary \ref{cor:fin}), this expression is
a polynomial in $Q$,  whose value at $Q=1$ is the
normalized Seiberg--Witten invariant of $(M,h*\sigma_{can})$.

\noindent (e) $$PE_\infty(1, Q,\h)\geq PE_\infty(1, Q,\h=0)\geq \frac{1}{1-Q}.$$
\end{proposition}
\begin{proof}
{\it (a)} follows from (\ref{eq:pqr}), while {\it (b)} from {\it (a)}. Next we prove {\it (c)}.

Set  $Eu(Q):=\sum_{\square_q\subset R}\, (-1)^q \, Q^{w_h(\square_q)}\in \Z[Q,Q^{-1}]$ and write
$Eu(Q)/(1-Q)$ as $\sum _{n\geq m_w} a_nQ^n$.
Then
$$a_n=\sum_{\square_q\subset R,\, w_h(\square_q)\leq n}\,(-1)^q=\chi_{top}(S_n\cap R).$$
But $S_n\cap R\hookrightarrow S_n$ is a homotopy equivalence, cf. Proposition \ref{lem:con}. Then use part {\it (b)}.

For {\it (d)} use {\it (c)} and $eu(\bH_*(C,o))=\sum_{\square_q\subset R}(-1)^{q+1}w_h(\square_q)$, cf. (\ref{eq:eu}).
For {\it (e)} use Prop. \ref{lem:con}.
\end{proof}
\begin{remark}\label{rem:hat} (a)
From part {\it (b)-- (c)} of the above proposition
$$\sum_{\square_q\subset R}\, (-1)^q \, Q^{w(\square_q)}=\sum_n\chi_{top}(S_n)(Q^n-Q^{n+1})=\sum_n \chi_{top}(S_n,S_{n-1})
Q^n=\sum_n\sum_b (-1)^b {\rm rank}(\hat{\bH}_b)_{-2n}Q^n.
$$
That is, $\sum_{\square_q\subset R}\, (-1)^q \, Q^{w_h(\square_q)}$ is a finite sum which is independent of the choice of
$R$, and its  categorification is  the bigraded $\hat{\bH}_*$.
The above identity  for $Q=1$ gives (\ref{eq:hateu}).

(b) Property {\it (c)} above  in this  form might be  `misleading': conceptually
the right hand side of the identity is related with
$PE_1$ (see Theorem \ref{th:betti} for the general statement).
However,  after the substitution
$T=1$ and $\h=-1$ one has $PE_1(1,Q,-1)=PE_\infty(1,Q,-1)$, cf. (\ref{eq:spseq}).
\end{remark}

\section{The multigraded $\Z$--module  $(E^1_{-\widetilde{l},q})_{n}$ and the series ${\bf PE}_1({\bf T}, Q,\h)$}\label{ss:PET}\

\subsection{The extended lattice  filtration}\

Write $s=\sum_{i\in{\rm Supp}(s)}n_iE^*_i= \sum_{a\in\cA} E^*_{i(a)}$ as above.
We also consider the natural projection
$pr:L_\phi=\Z\langle E_i\rangle _{i\in\cV}\to L_{{\rm Supp}(s)}=\Z\langle E_i\rangle_{i\in {\rm Supp}(s)}\simeq \Z^{r'}$, given by
$pr(\sum_{i\in\cV}l_iE_i)=\sum_{i\in {\rm Supp}(s)}l_iE_i$. We also regard $\Z\langle E_i\rangle_{i\in {\rm Supp}(s)}$
as a sublattice of $L_\phi$.
Note that $(l,s)=(pr(l),s)$. Define  $d(l):=-(l,s)$.

Furthermore, for any $\widetilde{l}\in  L_{{\rm Supp}(s)}$  set
$$\frX_{-\widetilde{l}}\ = \prod_{i\in{\rm Supp}(s)}\ [\widetilde{l}_i,\infty)\ \times\
 \prod_{i\not \in{\rm Supp}(s)}\ [0,\infty).$$

Recall that
for any fixed $n\geq m_w$ the spectral sequence $E_{*,*}^k$ ($k\geq 0$)
associated with the  level filtration  $\{S_n\cap \frX_{-d}\}_{d\geq 0}$  of $S_n$
has its first terms
 $$(E^0_{-d,q})_{n}=\calC_{-d+q}(S_n\cap \frX_{-d}, S_n\cap \frX_{-d-1}),\ \ \
 (E^1_{-d,q})_{n}=H_{-d+q}(S_n\cap \frX_{-d}, S_n\cap \frX_{-d-1},\Z).$$
The term  $(E^0_{-d,q})_{n}$ is generated  freely over $\Z$ by
 $(-d+q)$--cubes
of the form $\square =(l, I)$ with $d(l):=-(l,s)=d$ and $w((l,I))\leq n$. These cubes can be grouped
according to the index set $\{\widetilde{l}\in\Z^{r'}_{\geq 0},\ d(\widetilde{l})=d\}$. This decomposition can be seen at
the level of the homology of the  pair of  spaces  $(\frX_{-d},\frX_{-d-1})$ too. By excision, this can be rewritten as the homology of the disjoint union of  pairs
$\sqcup _{\widetilde{l}\in\Z^{r'}_{\geq 0},\ d(\widetilde{l})=d}(\frX_{\widetilde{l}}, \frX_{\widetilde{l}}\cap\frX_{-d-1})$. Intersecting with $S_n$ provides a similar disjoint union of spaces.
Hence,  we  automatically have  the following direct sum decomposition:
 \begin{equation*}\label{eq:sum}
 (E^1_{-d,q})_{n}= \bigoplus_{\widetilde{l}\in\Z^{r'}_{\geq 0},\ d(\widetilde{l})=d}\
 (E^1_{-\widetilde{l},q})_{n}, \ \mbox{where} \ (E^1_{-\widetilde{l},q})_{n}:=
 H_{-d(\widetilde{l})+q}(S_n\cap \frX_{-\widetilde{l}}\ , S_n\cap \frX_{-\widetilde{l}}\ \cap \frX_{-d(\widetilde{l})-1},\Z).\end{equation*}

In this way we assign to any decorated pair $(X,C,b)$ (or, decorated link $(M,L_C,b)$) a multigraded $\Z$--module
$(E^1_{-\widetilde{l},q})_{n}$, graded by $\widetilde{l}\in L_{{\rm Supp}(s),\geq 0}=
(\Z_{\geq 0})^{r'}$  and by two other  integers,
the weight degree  $n\in \Z$  and the homological degree
 $b=-d(\widetilde{l})+q\in \Z$.

 Accordingly to this direct sum decomposition  we also define the series
 $${\bPE}_1({\bf T}, Q,\h)={\bPE}_1(\{T_i\}_{i\in {\rm Supp}(s)}, Q,\h):=\sum_{\widetilde{l}\in\Z^{r'}_{\geq 0},\, n,q}\ \rank\big((E^1_{-\widetilde{l},q})_{n}\big)\cdot
 (\prod_{i\in{\rm Supp}(s)}T_i^{n_i\widetilde{l}_i})\, Q^n\, \h^{-d(\widetilde{l})+q}.$$
 Clearly, $${\bPE}_1({\bf T}, Q,\h)|_{T_i\mapsto T \mbox{\footnotesize{\ for all $i\in{\rm Supp}(s)$ }}}=PE_1(T,Q,\h).$$

If there are no pairs of  arrows of $(C,o)$ (or of $L_C$) which are supported on  the same exceptional irreducible component, then
${\rm Supp}(s)$ is identified with the index set $\cA$ of the irreducible components of $(C,o)$ (or, the number of link components of
 $L_C$), hence $r'=r$, and
the first index set is $(\Z_{\geq 0})^r$, the first quadrant of the lattice associated with the link components $\{L_{C,a}\}_{a\in\cA}$.

In this case the variables of the  series ${\bf PE}_1({\bf T},Q,\h) $ are $\{T_a\}_{a\in\cA}$, $Q$ and $\h$.

\begin{example}
In the situation of Example \ref{ex:22}, Case II when  $s=E_1^*+E_2^*$,
\begin{equation}\label{eq:ep1b}
{\bf EP}_1({\bf T},Q,\h)=\sum_{l\in \calS}\ {\bf E}(l), \ \ \mbox{where ${\bf T}=(T_1,T_2)$ and ${\bf E}(l)$ equals }
\end{equation}
\begin{equation*}  
{\bf T}^l\h^0\cdot \frac{Q^{\chi(l)}- Q^{\min\{\chi(l+E_1),\chi(l+E_2)\}}}{1-Q}
+{\bf T}^l\h^1\cdot \frac{  Q^{\max\{\chi(l+E_1),\chi(l+E_2)\}} - Q^{\chi(l+E_1+E_2)}}{1-Q}.
\end{equation*}
\end{example}

\subsection{The series $PE_1(T,Q,\h=-1)$ and  ${\bf PE}_1({\bf T},Q,\h=-1)$}\label{ss:hegy}\

From definition
$$PE_1(T,Q,\h)=\sum_{d,n,b} \, {\rm rank} \, H_{b}(S_n \cap \frX_{-d}, S_n\cap \frX_{-d-1},\Z)\, \cdot T^dQ^n \h^b.$$
In particular,
$$PE_1(T,Q,\h)_{\h=-1}=\sum_{d,n} \, \chi_{top}(S_n \cap \frX_{-d}, S_n\cap \frX_{-d-1})\, \cdot T^dQ^n.$$
Recall that in the classical topology
the topological Euler characteristic of a finite simplicial  complex
can be computed either from its  Betti numbers or from its cell decomposition.
 In our case we have a similar statement as well
(compare also with \ref{prop:infty}{\it (c)}).

First we start with the following support statement. As usual, for any
$\square=(l,I)$ set
$$\mbox{
$d(l,I)=d(l):=-(s,l)$ and  $w_h(l,I)=\max\{\chi_h(v)\,:\, \mbox{$v$ is a vertex of $\square$}\}$.}$$
\begin{lemma}\label{lem:sup}
Fix $l\in L$. Then $\sum_{I\subset \cV} (-1)^{|I|} Q^{w_h(l,I)}=0$ whenever $l+s_h\not\in \calS'$.
In particular, for any fixed $d$, there are only finitely many lattice points $l$ with
$d(l)\leq d$ and  $\sum_{I\subset \cV} (-1)^{|I|} Q^{w_h(l,I)}\not=0$.
\end{lemma}
\begin{proof}
If $l+s_h\not\in \calS'$ then there exists $i\in\cV$ such that $(l+s_h, E_i)\geq 1$.
Take any $J\subset \cV$ with $i\not\in J$. Then
$\chi_h(l+E_J+E_i)-\chi_h(l+E_J)=1-(l+E_J+s_h,E_i)\leq 0$, hence
$w_h(l,J)=w_h(l, J\cup\{i\})$.
\end{proof}
\begin{theorem}\label{th:betti}  (a) For any $(\Gamma, s)$
$$PE_1(T,Q,\h)_{\h=-1}=\frac{1}{1-Q}\cdot \sum_{(l,I)=\square\subset \frX} \ (-1)^{|I|} T^{d(l,I)}Q^{w_h(l,I)}=
\frac{1}{1-Q}\cdot\sum_{l\in L_{\geq 0}}\, T^{d(l)}\sum_{I\subset \cV} (-1)^{|I|}Q^{w_h(l,I)}.$$
Moreover, by Lemma \ref{lem:sup}, in the last expression the sum $\sum_{l\in L_{\geq 0}}$ can be replaces by $\sum_{l\in L}$, that is,
$$PE_1(T,Q,\h)_{\h=-1}=
\frac{1}{1-Q}\cdot\sum_{l\in L}\, T^{d(l)}\cdot\sum_{I\subset \cV} (-1)^{|I|}Q^{w_h(l,I)}.$$
(b)
Similarly,  in the ${\bf T}$--multivariable context,
$${\bf PE}_1({\bf T},Q,\h)_{\h=-1}=
\frac{1}{1-Q}\cdot
\sum_{\widetilde{l}\in\Z^{r'}_{\geq 0}}
 (\prod_{i\in{\rm Supp}(s)}T_i^{n_i\widetilde{l}_i})\, \cdot
\sum_{I\subset \cV} (-1)^{|I|}Q^{w_h(l,I)}.$$
\end{theorem}
\begin{proof} {\it (a)}
Write $\sum_{\square\subset \frX} (-1)^{\dim(\square)}T^{d(\square)}Q^{w_h(\square)})/(1-Q)$ as
$\sum_{n,d} \, a_{n,d} T^dQ^n$. Then $$a_{n,d}=
\sum_{\square\,:\, w(\square)\leq n,\ d(\square)=d}\ (-1)^{\dim(\square)}=
\chi_{top}(S_n \cap \frX_{-d}, S_n\cap \frX_{-d-1}).$$
The last statement of {\it (a)} follows from the  minimality of $s_h$ in $\calS'\cap \{l'\in L'\,:\, [l']=h\}$.
Indeed, if $l\not\geq 0$ then $l+s_h\not\in \calS'$. The proof of part {\it (b)} is identical.
\end{proof}
\begin{corollary}\label{cor:EP} ${\bf PE}_1({\bf T},Q,\h)_{\h=-1}$ is obtained from $Z^m_h(\bt,q)\cdot \bt^{-s_h}$
by substitutions $t_i\mapsto T_i^{n_i}$ for $i\in {\rm Supp}(s)$, $t_i\mapsto 1$ for $i\not\in {\rm Supp}(s)$,
and $q\mapsto Q$.
\end{corollary}

\bekezdes
Let us comment the expression from the  right hand side of the identity from Theorem  \ref{th:betti}{\it (a)}.
It is in the spirit of (weighted) lattice point computation (like in the Ehrhart theory or theory of
partition function), however, this expression uses not only lattice points but all the cubes. In this way it is a rather complicated summation. Moreover, the weights of the cubes are given by the expression  $w_h((l,I))$, which is the maximum of weight over the vertices of the cube, again a hardly computable  expression.  So, at the first glance,
we might doubt  that this expression will guide us to a conceptual understanding of the series $PE_1$.

Still, this will be the case: we will prove  that  for any fixed $I$  the expressions
$w_h((l,I))$, where  $l$ runs over certain sublattices, behave very regularly. In order to
prove this we also need a
unifying  reorganization  of a simultaneous  computation of  series $PE_1(T,Q,\h)_{|\h=-1}$ defined for all group elements  $h\in H$. The point is that if we pack together all the series indexed by all elements $h\in H$ then  we get a more manageable summation (and lattice).

Recall that above we fixed a resolution graph $\Gamma$, $h\in H$, and the Riemann--Roch expression
$\chi_h:(\Z_{\geq 0})^r\to \Z$, $\chi_h(l)=-(l,l-Z_K+2s_h)/2$. In order to emphasize the $h$--dependence, let us denote the corresponding series by $PE_{h,1}(T,Q, \h)$. We collect all of them in one formal series as follows.

First note that $\chi:L'_{\geq 0}\to \Q$ defined by $\chi(l')=-(l',l'-Z_K)/2$
uniformize all the $\{\chi_h\}_h$ expressions. Indeed, for any $h\in H$ and $l\in L$,  we have
$$\chi(l+s_h)=\chi_h(l)+\chi(s_h).$$
That is, $l'\mapsto \chi(l')-\chi(s_h)$ restricted to $L'_h:=\{l'\in L'\, :\, [l']=h\}$
is exactly $\chi_h(l'-s_h)$. Then, with the notation $l'=l+s_h$,  the
term $T^{-(s,l)}Q^{w_h(l)}$ in the expression of $PE_{h,1}(T,Q,\h)$ (evaluated at 0--cubes)  reads as
$$T^{-(s, l'-s_h)}Q^{\chi_h(l'-s_h)}=T^{(s,s_h)-(s,l')}Q^{\chi(l')-\chi(s_h)}.$$
Then, we can consider cubes in $\R^s$ of type $(l',I)$ with vertices
$\{l'+E_J\}_{J\subset I}$, $ l'\in L'$. For each class $h$, $\R^s$ has a cubical decomposition
with cubes of type $\{(l',I)\}_{l'\in s_h+L}$.
For a cube $(l',I)$   we define
its  degree $d((l',I))=d(l'):=-(l',s)\in \Q$ and its weight $w((l',I)):=\max\{\chi(l'+E_J)\,:\, J\subset I\}$.

Note also that $\{\chi(l')\,:\, [l']=h\}\subset \chi(s_h)+\Z$
and $\{(l',s)\,:\, [l']=h\}\subset -(s_h,s)+\Z$.

For any $h\in H$ and $n\in \chi(s_h)+\Z$ consider the finite CW space
$\overline{S}_{h,n}$ as the union of cubes of $\R^s$ of the form $(l', I)$, $[l']=h$ and $w(l',I)\leq n$.
 Each $\overline{S}_{n,h}$ has a filtration $\{\overline{S}_{h,n, -d}\}_d$, where
$\overline{S}_{h,n, -d}$ is the union of cubes of  $\overline{S}_{h,n}$ with $d(l')\geq d$, where $d\in -(s_h,s)+\Z$.
Then the unified sum is
$$\sum_{h\in H}\ \sum_{n,d}\chi_{top}(\overline{S}_{h, n,-d}, \overline{S}_{h, n,-d-1}) T^d Q^n [l']
\in \Z[[T^{1/|H|}, Q^{1/|H|}]]\,[Q^{-1}]\,[H].$$
By the above reinterpretation via the shifted cubes  this equals
$$\sum_{h\in H}\  T^{-(s,s_h)} Q^{\chi(s_h)}\cdot
PE_{h,1}(T,Q,\h)_{|\h=-1}[h].$$
By the above discussion and Theorem \ref{th:betti} and Lemma \ref{lem:sup} this also equals
$$\sum_{l'\in L'} \ \sum_{I\subset \cV} \, (-1)^{|I|} T^{d(l')} Q^{w(l',I)}[l'].$$
Again, by Lemma \ref{lem:sup} we can replace the first sum $\sum_{l'\in L'}$ by $\sum_{l'\in \calS'}$, hence finally we get
\begin{equation}\label{eq:uni}
\sum_{h\in H}\  T^{-(s,s_h)} Q^{\chi(s_h)}\cdot
PE_{h,1}(T,Q,\h)_{|\h=-1}[h]=
\sum_{l'\in \calS'} \ \sum_{I\subset \cV} \, (-1)^{|I|} T^{d(l')} Q^{w(l',I)}[l'].
\end{equation}
The right hand side of this expression  has two advantages.
Firstly, it uniformly provides all terms $PE_{h,1}$, hence we do not have to treat the shifted sublattices $L'_h=s_h+L$ independently (which would have rather technical and arithmetical parametrizations).
Secondly, the summation index set $\calS'$ is the  `first quadrant' of a lattice.
It is  generated by $\{E_i^*\}_i$ over $\Z_{\geq 0}$, i.e.
 $\calS'=\Z_{\geq 0}\langle E_i^*\rangle_{i=1}^s \simeq (\Z_{\geq 0})^s$.
Thus, if we write $l'=\sum_ia_iE_i^*$, then the expression becomes
\begin{equation}\label{eq:SUM}
\sum_{a\in (\Z_{\geq 0})^s} \ \sum _{I\subset \cV}\ (-1)^{|I|}\ T^{-(s, \sum_i a_iE^*_i)} \ Q^{w((\sum_ia_iE^*_i, I))}\, [\Sigma_ia_iE^*_i].\end{equation}

Note that in  (\ref{eq:SUM})
the summation is over cubes, and $\square\mapsto \max\{w_h(v)\,:\, v \ \mbox{is a vertex of $\square$}\}$ is a complicated irregular arithmetical function. Still, in the next Theorem \ref{th:form_int}, by proving a certain regularity behaviour of this function, we replace the cube--summation by a summation over lattice points (with Jacobi theta series type summands).


\begin{theorem}\label{th:form}
There exist a finite index set ${\cP}$, integers $\{a_\cP\}_\cP$, $\{b_\cP\}_\cP$, sublattices $\Z^{s_\cP}$ of $L'$ and elements
$\{k_\cP\}_\cP,\ \{r_\cP\}_\cP\in L'$ such that  the expression from (\ref{eq:SUM}) can be written as a sum
$$
\sum_\cP \ \sum_I \ (-1)^{|I|}\ T^{a_\cP}Q^{b_\cP}\cdot \
 \sum _{A\in (\Z_{\geq 0})^{s_\cP}}\, T^{-(s,A)}\, Q^{-(A,A+k_\cP)/2}\ [r_\cP+A].
$$

\end{theorem}
\begin{proof}
The cube--weight $w((l',I))$ is the maximum of the weights of the vertices of $(l',I)$.
The point is that the choice of the lattice points which realizes the
maximum has a certain regularity  along different members  of a partition of the lattice points $l'$.
 The partition is given by
the `distance' from the boundary faces of $\calS'$. In the language of $a\in (\Z_{\geq 0})^s\simeq\calS'$,
the boundary faces  of
$\calS'$ are given by $\{a_i=0\}$.

The partition what we propose is defined as follows. First, we fix an integer $\Delta>0$ so that $\Delta\geq \kappa_i-1$   for all $i\in\cV$, where $\kappa_i$ is the valency of the  vertex $i$ in $\Gamma$.
For any subset $K\subset \cV$ and integers $\{a^K_j\}_{j\not\in K}$ satisfying $\Delta>a^K_j\geq 0$ we define
$$\cP (K, \{a^K_j\}_{j\not\in K}):= \Big\{\sum _{i\in \cV}a_iE^*_i\ :\ \mbox{$a_i\geq \Delta$ \, for\, $i\in K$, and\, $a_j=a_j^K$\, for\, $j\not\in K$}\Big\}.$$
Then we have the following facts:

$\bullet$ each $\cP (K, \{a^K_j\}_{j\not\in K})$ can be identified with a quadrant $(\Z_{\geq 0})^{|K|}$.
Indeed, let the `root' of $\cP (K, \{a^K_j\}_j)$ be defined as $r=r_\cP:= \sum_{i\in K}\Delta E_i^*+\sum_{j\not\in K} a_j^K E_j^*$.
Then  $\cP (K, \{a^K_j\}_{j\not\in K})=r+(\Z_{\geq 0})^{|K|}$. Indeed, if  we denote the entries of $(\Z_{\geq 0})^{|K|}$ by $\{A_i\}_{i\in K}$, then $a_i=\Delta+A_i$ for $i\in K$ realizes the identification.

$\bullet$ $\cP (K, \{a^K_j\}_{j\not\in K})$, for different subsets $K$ and integers $\{a^K_j\}_{j\not\in K}$, is a  partition of the lattice points of $\calS'$, i.e. of
$\{a_i\}_{i\in \cV}\in (\Z_{\geq 0})^s$.\\

The key technical lemma regarding this partition is the following.
Fix  $\cP=\cP (K, \{a^K_j\}_{j\not\in K})=r+(\Z_{\geq 0})^{|K|}$ with root $r$.
Consider a cube of type $(r,I)$ ($I\subset \cV$) and assume that $w((r,I))$ (as maximum over its vertices) is realized by the vertex
$r+E_{J(r, I)}$. Usually the choice of $J(r,I)$ is not unique.
\begin{lemma}\label{lem:tech}
We can  choose $J(r,I)$ in such a way that  $J(r,I)\supset K\cap I$ and  additionally the following regularity holds:
for any $l'=r+A$, $A=\sum _{i\in K}A_iE_i^*$ ($A_i\geq 0$), and cube $(l',I)$, the maximum
$w((l',I))$ is realized by the vertex $l'+E_{J(r,I)}$ (i.e. for any fixed $I$,
the corresponding maximizing subsets $J\subset I$  of the cubes $(l',I)$  can be chosen uniformly
for any $l'\in\cP$).
\end{lemma}
\begin{proof} For any $J\subset I$  and $l'=r+A$ we have   $\chi(l'+E_J)-\chi(l')=\chi(E_J)-(r+A, E_J)$.

If $K\cap I=\emptyset$,  then $(A, E_J)=0$ for any $J\subset I$, hence
 $ \chi(l'+E_J)-\chi(l')=\chi(r+E_J)-\chi(r)$,  
   hence any maximizing $J$ of $r$ is a maximizing set for any  $l'$ as well.

Assume $K\cap I\not=\emptyset$. Fix any $J\subset I$ and $i\in (K\cap I)\setminus J$.
Then, independently of the choice of $A$
\begin{equation}\label{eq:12}
\chi(E_{J\cup \{i\}})-(r+A, E_{J\cup\{i\}})\geq \chi(E_{J})-(r+A, E_{J}).\end{equation}
Indeed, if $S\subset I$ then $\chi(E_S)$ equals the number of connected components of the full subgraph supported by $S$.
 Hence, for any   $i\in K\cap I$ we have $\chi(E_{J\cup\{i\}})-\chi(E_J)\geq -(\kappa_i-1)$ and $(r+A, E_i)\geq \Delta\geq \kappa_i-1$
 for any $A$. This, by (\ref{eq:12}), for any $A$ we can choose the maximizing $J$ such that $J\supset K\cap I$.

On the other hand, if we write $J=J'\cup(K\cap I)\subset I$ with $J'\cap K\cap I=\emptyset$, then $(E_{J'},A)=0$, hence the difference
  $\chi(l'+E_J)-\chi(l'+E_{K\cup I})$ is $A$--independent.
\end{proof}
This shows that the sum from the right hand side of (\ref{eq:uni}) can be organized as follows.
Write the lattice points of each $\cP(K, \{a^K_j\}_j)=\cP$ as $r_\cP+A_\cP$. Then the sum is
\begin{equation*}\label{eq:sumsum}
\begin{split}
&\sum_{\cP}\ \sum_{I}\ (-1)^{|I|} \ \sum _{A_{\cP}\in (\Z_{\geq 0})^{|K|}}\, T^{-(s, r_\cP+A_\cP)}\, Q^{\chi(r_\cP+A_\cP+E_{J(r_\cP,I)})}\ [r_\cP+A_\cP]=\\
&
\sum_{\cP}\ \sum_{I}\ (-1)^{|I|} \ T^{-(s, r_\cP)} \, Q^{\chi(r_\cP+E_{J(r_\cP,I)})} \
\sum _{A_{\cP}\in (\Z_{\geq 0})^{|K|}}\, T^{-(s,A_\cP)}\, Q^{\chi(A_\cP)-(A_\cP, r_\cP+E_{J(r_\cP,I)})}\ [r_\cP+A_\cP].\\
\end{split}
\end{equation*}
Since $$\chi(A_\cP)-(A_\cP, r_\cP+E_{J(r_\cP,I)})=-\frac{1}{2}(A_\cP, A_\cP-Z_K +2r_\cP+2E_{J(r_\cP,I)})$$
all the  lattice summations have the form
 $$ \sum _{A\in (\Z_{\geq 0})^{|K|}}\, T^{-(s,A)}\, Q^{-(A,A+k)/2}\ [r+A].$$
\end{proof}
\section{The almost rational case revisited}\label{s:rat}

\subsection{The general setup of AR case} \ Assume that $(X,o)$ is an {\it almost  rational} singularity and $(C,o)\subset (X,o)$ is {\it irreducible}, cf. subsection \ref{ss:redth}. Let $b$ be the decoration of $(C,o)$ as in subsection \ref{ss:1.2}. Let us fix an embedded resolution graph $\phi$ which represents
$(X,C,b)$ and let us replace $(C,b)$ by the semigroup element $s\in \calS'_{\phi}$ as in \ref{bek:semi}.
Let $E_0$ be that irreducible exceptional divisor of $\phi$ which supports the strict transform of $(C,o)$.
Hence $s=E_0^*$.  Recall that in order to be able to apply the Filtered Reduction Theorem
 \ref{th:redth2} (where the grading is induced by $s=E_0^*$) one needs $\{0\}\subset \overline{\cV}$.
In the sequel we assume that
$\{0\}$ can serve as the unique bad (i.e. SR)  vertex for the graph $\Gamma$ of $\phi$ as well.
Hence, in this very natural  (and optimal) case  $\overline{\cV}=\{0\}$.

Recall that if $(X,o)$ is rational then any vertex of $\Gamma_\phi$ might serve as an SR set.

Moreover,  in this section we assume that $h=0$ and that the class  $[E^*_0]$ in $H$ is trivial.
Based on the discussion below the reader can rewrite the general statement for arbitrary $h$ and $[E^*_0]$ easily.

 Next we  determine $PE_k(T,Q,\h)$ and connect it firstly with the embedded topological type
of the decorated  $(M,L_C,b)$ and then with the analytic Poincar\'e series of the abstract curve singularity $(C,o)$.

\bekezdes   By the Filtered Reduction Theorem
 \ref{th:redth2} we have the homotopy equivalences $S_n(w)\cap \frX_{-d}\sim S_n(\overline{w})\cap \overline{\frX}_{-d}$
 for every $n$ and $d$.
 Here $\overline{w}$ and $w$ are abridgements  for $\overline{w}_{h=0}$ and $w_{h=0}$, and $\overline{w}(\bar{l})=w(x(\bar{l}))$, where
 $\bar{l}\in \Z=L(\overline{\cV})$ and $x(\bar{l})$ is the universal cycle introduced in \ref{ss:redth} associated with $\overline{\cV}$.

  \begin{remark}\label{rem:ineqtau} By \cite[part 7. 3.36 and Example 11.4.10]{Nkonyv} the sum
  $\sum_{\bar{l}\geq 0}\max\{0, \overline{w}(\bar{l})-\overline{w}(\bar{l}+1)\}$ is finite and it equals $eu(\bH_*((\R_{\geq 0})^s,w)$.
  In particular, when $\Gamma$ is rational then $eu(\bH_*((\R_{\geq 0})^s,w)=0$, cf. \cite[Example 11.1.28]{Nkonyv}, and
  $\overline{w}(\bar{l}+1)\geq \overline{w}(\bar{l})$ for every $\bar{l}\geq 0$.
  \end{remark}

Similar  discussion as in Example \ref{ex:237} gives the following statement.
\begin{theorem}\label{th:uj}
Assume that $\{0\}$ is the bad vertex of an AR graph $\Gamma_\phi$,
and that this vertex supports the arrowhead of the strict transform of the
irreducible $(C,o)$. Then $PE_1(T,Q,\h)=PE_\infty(T,Q,\h)$ and it equals
$$PE_1(T,Q,\h)=\sum\ T^{\bar{l}}\cdot \frac{Q^{\overline{w}(\bar{l})}-Q^{\overline{w}(\bar{l}+1)}}{1-Q},$$
where the sum is over
those integers  $\bar{l}\geq 0$ for which $\overline{w}(\bar{l}+1)\geq \overline{w}(\bar{l})$.
\end{theorem}

\bekezdes  In fact, the series $\sum_{\bar{l}\geq 0}\max\{0, \overline{w}(\bar{l}+1)-\overline{w}(\bar{l})\}\cdot t^{\bar{l}}$
has a very precise topological meaning.

Let $Z(\bt)$ be the topological Poincar\'e series of $\Gamma_\phi$ and $Z_0(\bt)$ its $(h=0)$--component,
cf. subsection \ref{ss:3.1}.
Let $Z_{0,0}(t)$ be the reduction of $Z_{0}(\bt)$ to the variable of $t_0$ associated with $E_0$, that is,
$$Z_{0,0}(t):= Z_0(\bt)|_{\footnotesize{\mbox{$t_i=1$\ if $i\not=0$, and  \ $t_0=t$}}}.$$
\begin{lemma}\label{lem:tauZ}\cite[Example 11.4.10]{Nkonyv}
$$Z_{0,0}(t)=\sum_{\bar{l}\geq 0}\max\{0,
 \overline{w}(\bar{l}+1)-\overline{w}(\bar{l})\}\cdot t^{\bar{l}}.$$
\end{lemma}
The point is that from this $Z_{0,0}(t)$
one can recover the $\overline{w}$--function $\bar{l}\mapsto \overline{w}(\bar{l})$, hence via Theorem \ref{th:uj}
the series $PE_1(T,Q,\h)$ as well.

In the {\bf rational case}
when  $\overline{w}(\bar{l}+1)\geq \overline{w}(\bar{l})$ for any $\bar{l}\geq 0$, the series becomes
$Z_{0,0}(t)=\sum_{\bar{l}\geq 0} \ (\overline{w}(\bar{l}+1)-\overline{w}(\bar{l}))\cdot t^{\bar{l}}$. Thus,
$\overline{w}(0)=0$ and $\overline{w}(\bar{l})=\sum_{\tilde{l}<\bar{l}}z(\tilde{l})$, where
$Z_{0,0}(t)=\sum_{\bar{l}}z(\bar{l})t^{\bar{l}}$. Therefore, in this case,
$$PE_1(T,Q,\h)=PE_1(T,Q)= \frac{1}{1-Q}\cdot \Big[ \, 1+(T-1)\cdot \sum_{\bar{l}\geq 0} \, T^{\bar{l}}Q^{\overline{w}(\bar{l}+1)}\,\Big].$$
In the general {\bf almost rational case}
 $Z_{0,0}(t)$ can be written in a unique way as $Z^+_{0,0}(t)+
Z^-_{0,0}(t)$, where $Z^+_{0,0}$ is  polynomial and  $Z^-_{0,0}$ is a rational function of negative
degree. It turns out, cf. \cite[Example 11..4.10]{Nkonyv}, that  $Z^-_{0,0}$ equals
$\sum_{\bar{l}\geq 0} \ (\overline{w}(\bar{l}+1)-\overline{w}(\bar{l}))\cdot t^{\bar{l}}$, hence it determines
the $\overline{w}$--function (by summation of the coefficients as in the rational case).
We also note that $Z^+_{0,0}(1)=eu(\bH_*)$ too.

 \bekezdes In the {\bf rational case} we have another interpretation as well.
In the next discussion  we concentrate on the {\it abstract analytic type} of the curve singularity $(C,o)$. First we review some well--known facts.
Let $n:(\bC,0)\to (C,o)$ be the normalization. For any $g\in \calO_{C,o}$ write $g\circ n$ in the form $a_ot^o+\cdots \in \bC\{t\}=\calO_{\bC,0}$ with $a_o\in\bC^*$ and define the valuation ${\mathfrak v}(g)={\rm ord}_t(g\circ n)=o$. Then, by definitions,
$\calS_C:=\{\mathfrak{v}(g)\,:\, g\in\calO_{C,o}\}\subset \Z_{\geq 0}$ is the {\it semigroup of values}  of $(C,o)$ and
$P_C(t):=\sum_{s\in \calS_C}t^s$ is the {\it analytic Poincrar\'e series} of the abstract irreducible curve $(C,o)$.

\begin{proposition}\label{th:anal}
Assume that $(X,o)$ is rational.
Fix the resolution $\phi$ as above, and set   $m_\phi:=-(E^*_0, E^*_0)$, the $E_0$--multiplicity of
$E^*_0$. Then
$$Z_{0,0}(t)=\frac{P_C(t)}{1-t^{m_\phi}}.$$
In particular, $PE_1(T,Q,\h)$ can be recovered from the integer $m_\phi$ and from the analytic Poincar\'e series
of the abstract curve $(C,o)$ (or, from $m_\phi$  and the semigroup of values of the abstract curve $(C,o)$).
\end{proposition}
\begin{proof}
Note that $Z_{0,0}(t)\cdot (1-t^{m_\phi})$ is the relative (embedded) topological Poincar\'e series of the pair
$(C,o)\subset (X,o)$ (its definition is analogue with the definition of $Z_{0,0}(t)$ but one replaces the valency of the vertices
in $\Gamma_\phi$ with
the valency in the embedded resolution graph, that is, $\kappa_0$ is increased by one). Then the statement follows from
\cite{CDGc} or from \cite[Corollary 9.4.1]{NPo}.
\end{proof}

\begin{example} (a) Since $ \sum_{s\in \Z_{\geq 0}\setminus \calS_C} $ is a finite sum,
$P_C(t)$  can be written in the form
$\Delta(t)/(1-t)$ for a certain
 $\Delta(t)\in\Z[t]$. If $(C,o)$ is a plane curve singularity (i.e. it can be embedded into $(\bC^2,0)$, e.g.
  when $(X,o)$  is smooth),
 then by \cite{cdg}  $\Delta(t)$ is the  Alexander polynomial of the embedded topological type $(C,o)\subset (\bC^2,0)$.

 (b) If $(X,o)$ is not rational (or if we do not impose some kind of  additional analytic assumption regarding the pair $(X,C,o)$) then $P_C(t)$ cannot be recovered from the embedded topological type
 of the pair $(X,C,o)$, see e.g. \cite[Example 9.4.3]{NPo}.
\end{example}

\end{document}